\documentclass{amsart}
\usepackage[utf8]{inputenc}
\usepackage{amsmath, amssymb}
\usepackage{tikz, tikz-cd}
\usepackage{color}
\usepackage[all]{xy}
\usepackage{caption}
\usetikzlibrary{calc, intersections, through}
\usepgflibrary{arrows}
\usetikzlibrary{positioning}
\usepgflibrary{arrows}
\usepackage{enumitem}

\tikzstyle{wbullet}=[circle, draw=black, fill=white, thick, inner sep=0pt, minimum size=1.5mm]
\tikzstyle{bbullet}=[circle, draw=black, fill=black, inner sep=0pt, minimum size=1.5mm]

\usepackage{array}
\usepackage{pdflscape}
\usepackage{afterpage}
\usepackage{capt-of}
\usepackage{lipsum}
\usepackage{graphicx,caption}
\usepackage{bbm}
\usepackage[symbol]{footmisc}
\usepackage{xcolor}

\usepackage[
 hypertexnames=false,
pdfpagemode=UseNone,
breaklinks=true,
extension=pdf,
urlcolor=blue,
colorlinks=true,
linkcolor=blue,
citecolor=blue,
]{hyperref}
\hypersetup{pdfauthor={Name}}

 \usepackage{todonotes}

\newtheorem{thm}{Theorem}[section]
\newtheorem{conj}[thm]{Conjecture}
\newtheorem{lem}[thm]{Lemma}
\newtheorem{lem-defn}[thm]{Lemma-Definition}
\newtheorem{claim}[thm]{Claim}

\newtheorem{prop}[thm]{Proposition}
\newtheorem{cor}[thm]{Corollary}
\theoremstyle{definition}
\newtheorem{defn}[thm]{Definition}
\newtheorem{nota}[thm]{Notation}

\newtheorem{rmk}[thm]{Remark}

\newtheorem{ex}[thm]{Example}
\theoremstyle{remark}
\newenvironment{pf}{\paragraph{Proof}}{\par\medskip}
\numberwithin{equation}{section}

\newcommand{\op}{\mathrm}

\newcommand{\mc}{\mathcal}
\newcommand{\gb}{\mathfrak b}

\newcommand{\gm}{\mathfrak m}

\newcommand{\ti}{\tilde}

\newcommand{\mb}{\mathbb}

\newcommand{\mo}{\mathcal O}

\newcommand{\CC}{\mathbb{C}}

\newcommand{\FF}{\mathbb{F}}
\newcommand{\PP}{\mathbb{P}}
\newcommand{\Q}{\mathbb{Q}}
\newcommand{\QQ}{\mathbb{Q}}

\newcommand{\ZZ}{\mathbb{Z}}

\newcommand{\cH}{\check{H}}

\newcommand{\cS}{\check{S}}

\newcommand{\oC}{\overline{C}}
\newcommand{\oD}{\overline{D}}
\newcommand{\oE}{\overline{E}}
\newcommand{\oF}{\overline{F}}

\newcommand{\oH}{\overline{H}}

\newcommand{\oS}{\overline{S}}

\newcommand{\oY}{\overline{Y}}

\newcommand{\sA}{\mathcal{A}}

\newcommand{\sE}{\mathcal{E}}

\newcommand{\sL}{\mathcal{L}}

\newcommand{\sO}{\mathcal{O}}

\newcommand{\tC}{\widetilde{C}}

\newcommand{\tG}{\widetilde{G}}
\newcommand{\tH}{\widetilde{H}}

\newcommand{\tS}{\widetilde{S}}
\newcommand{\tT}{\widetilde{T}}

\newcommand{\au}{\op{Aut}_{\QQ}(S)}
\newcommand{\Aut}{\mathrm{Aut}}
\newcommand{\Bs}{\mathrm{Bs}}
\newcommand{\Exc}{\mathrm{Exc}}

\newcommand{\GL}{\mathrm{GL}}
\newcommand{\Hom}{\mathrm{Hom}}
\newcommand{\I}{\mathrm{I}}

\newcommand{\III}{\mathrm{III}}

\newcommand{\V}{\mathrm{V}}
\newcommand{\id}{\mathrm{id}}

\newcommand{\Proj}{\mathrm{Proj}\,}

\newcommand{\mult}{\mathrm{mult}}

\newcommand{\red}{\op{red}}
\newcommand{\rest}[1]{{}_{\left|#1\right.}}

\newcommand{\sm}{\mathrm{sm}}
\newcommand{\sing}{\mathrm{sing}}
\newcommand{\Sing}{\mathrm{Sing}}

\newcommand{\Supp}{\mathrm{Supp}}

\newcommand{\vol}{\mathrm{vol}}

\renewcommand{\tilde}{\widetilde}

\newcommand{\separate}{\medskip}

\title[Numerically trivial automorphisms of surfaces]{On regular surfaces of general type with numerically trivial automorphism group of order $4$}
\author{Jin-Xing Cai}
\address{Jin-Xing Cai\\School of  Mathematical Sciences\\Peking University\\Beijing 100871\\P.R.~China}
\email{jxcai@math.pku.edu.cn}
\author{ Wenfei Liu}
\address{Wenfei Liu \\School of Mathematical Sciences\\Xiamen University\\Xiamen, Fujian 361005\\P.R.~China}
\email{wliu@xmu.edu.cn}
\date{\today}
\begin{document}

\begin{abstract}
Let $S$ be a regular minimal surface of general type over the field of complex numbers, and $\Aut_\QQ(S)$ the subgroup of automorphisms acting trivially on $H^*(S, \QQ)$. It has been known since twenty years that $|\Aut_\QQ(S)|\leq 4$ if the invariants of $S$ are sufficiently large. Under the assumption that $K_S$ is ample, we characterize the surfaces achieving the equality, showing that they are isogenous to a product of two curves, of unmixed type, and that the group $\Aut_\QQ(S)$ is isomorphic to $(\ZZ/2\ZZ)^2$. Moreover, unbounded families of surfaces with  $\Aut_\QQ(S)\cong(\ZZ/2\ZZ)^2$ are provided.
\end{abstract}

\maketitle
\tableofcontents

\section{Introduction}
An automorphism of a smooth complex projective variety $X$ is numerically trivial if it induces trivial action on the cohomology ring $H^*(X,\QQ)$ with $\mathbb{Q}$-coefficients, and we denote by $\Aut_{\QQ}(X)$ the group of such automorphisms. For curves the situation is clear cut: an automorphism is numerically trivial if and only if it lies in the identity component of the automorphism group. On surfaces the situation is more subtle: 
\begin{itemize}[leftmargin=*]
    \item Surfaces with nontrivial $\Aut_{\QQ}(X)$ are present in each Kodaira dimension $-\infty, 0,1,2$.
    \item There are surfaces with Kodaira dimension $\kappa(X)\in\{-\infty, 1\}$ such that the number of connected components of $\Aut_{\QQ}(X)$ gets arbitrarily large (\cite{CatLiu21, CFGLS24, CLS24}).
    \item For surfaces of Kodaira dimension  $\kappa(X)\in\{0,2\}$, there is a uniform bound on the the number of connected components of $\Aut_{\QQ}(X)$ by \cite{Cai04, CLZ13, CatLiu21}. In this case, one goes on to characterize the surfaces attaining the upper bound; see \cite{MN84, Muk10, CatLiu21} for surfaces with $\kappa(X)=0$, and \cite{Liu18, CL18} for surfaces of general type.
\end{itemize}

We elaborate more on the state of the art for surfaces of general type, which is the focus of this paper. A uniform bound $|\Aut_\QQ(S)|\leq 4$ for surfaces of general type with  $\chi(\mc O_S)>188$ was obtained in \cite[Theorem A]{Cai04} by the first author. It turns out that surfaces of general type with $|\au| = 4$ satisfy $q(S) \leq 1$ (\cite{CLZ13}), and those with $q(S) = 1$ are isogenous to a product and the Albanese
map is a fibration with fiber genus 3 or 5 (\cite{CL18}).

In this paper, we investigate surfaces of general type with $q(S)=0$ and $|\au|=4$, and obtain a similar picture under some extra conditions.
\begin{thm}[Theorem~\ref{thm: k=1} and Theorem~\ref{thm: k=2}]\label{main1}
Let $S$ be a complex minimal smooth
surface of general type with $q(S)=0$ and
  $\chi(\mc O_S)\geq 56$. If $|\Aut_\QQ(S)|=4$ and $K_S$ is ample, then  $\Aut_\QQ(S)\cong (\ZZ/2\ZZ)^2$ and $S$ is isogenous to a product of curves of unmixed type.
\end{thm}
Theorem~\ref{main1} is supported by an unbounded family of regular surfaces of general type with $|\Aut_\QQ(S)=(\ZZ/2\ZZ)^2$, given in Section~\ref{sec: ex}. As in \cite{CL18}, the examples are all surfaces isogenous to a product of unmixed type, which admit an \'etale cover isomorphic to a product of two curves.

In view of Theorems~\ref{main1} as well as \cite{CL18}, the following statement seems plausible:
\begin{conj}\label{conj}
There is a positive integer $N$, such that if $S$ is a minimal regular surface
of general type with $|\au|=4$ and $\chi(\mc O_S)\geq N$, then $\Aut_\QQ(S)\cong (\ZZ/2\ZZ)^2$, and $S$ is a surface isogenous to a product of unmixed type.
\end{conj}
 We note that the condition $|\au|=4$ in the above conjecture cannot be weakened to $\au$ being nontrivial. Indeed, there do exist an unbounded series of surfaces of general type with
 $\au\cong\ZZ/2\ZZ$ such that $K_S^2=4\chi(\mc O_S)$; see \cite[Example 3.3]{Cai06} and \cite[Theorem 1.3]{Cai07}). Also, for surfaces of general type with very small cohomology ring $H^*(S, \QQ)$, such as the fake projective planes \cite{PY07} and surfaces isogenous to a product with $p_g=q=0$ (\cite{BC04}), one has $\Aut_\QQ(S)=\Aut(S)$, which can have order larger than $4$. Thus the condition $\chi(\mc O_S)\geq N$ in Conjecture~\ref{conj} cannot be dropped. It would then be interesting to find an optimal value for $N$.

Theorems~\ref{main1} will be proved in Section~\ref{sec: G4}, based on the ideas developed in our previous papers \cite{Cai04, CL18}. The basic observation here is that the canonical map $\phi_S\colon S\dashrightarrow \PP^{p_g-1}$ factors through the quotient map $\pi\colon S \to S/\au$. We use $\phi_S$ to obtain a fibration structure $f\colon S\rightarrow B$ such that $\Aut_\QQ(S)$ preserves the fibers:  Let $T$ be a smooth minimal model of the  quotient surface $S/\au$.  Since $p_g(T)=p_g(S)>1$ by assumption, the Kodaira dimension $\kappa(T)=1$ or $2$. If $\kappa(T)=1$ then the pull-back of the elliptic fibration on $T$ is a fibration on $S$. If $\kappa(T)=2$, using some results of Xiao \cite{X86, X87}, we find a genus 2 fibration on $T$ which induces a fibration on $S$. In both cases, the key point is to bound the genus $g$ of the fibration $f\colon S\rightarrow B$ thus obtained; it turns out that $g=3$ in the case $\kappa(T)=1$, and $g=5$ in the case $\kappa(T)=2$.

Then we give a thorough analysis of the vertical-curve preserving action of $G\cong(\ZZ/2\ZZ)^2$ on a fibration $f\colon S\rightarrow B$ of genus $3$ and $5$ in Sections~\ref{sec: fib g3} and \ref{sec: fib g5}.  Under the assumption that $K_S$ is ample, we are able to list all possible singular fibers of $f$ together with the action of $G$; see Corollaries~\ref{g3 ample sing fib} and \ref{g5 ample sing fib}. Using the Lefschetz fixed point formulas, there are only a few possible combinations of singular fibers. Then some more argument is enough to conclude that such a fibration is a quasi-bundle and that the surface $S$ is isogenous to a product. 

A similar analysis of the vertical-curve preserving action of $G\cong\ZZ/4\ZZ$ on a fibration $f\colon S\rightarrow B$ of genus $3$ is given in \ref{sec: fib g3c4}. We show that this case does not occur provided that $\chi(\mc O_S)>2$ (Theorem~\ref{thm: g3c4}).

The paper is organized as follows. In Section~\ref{sec: local}, we collect and prove various results about local aspects of the action of a finite group $G$ on a surface, with focus on the cases $G\cong(\ZZ/2\ZZ)^n$ or $\ZZ/4\ZZ$. In Sections~\ref{sec: fib g3}, \ref{sec: fib g5} and \ref{sec: fib g3c4}, we describe vertical-curve preserving automorphisms acting on   fibrations of genera 3 and 5, which appear later on regular surfaces of general type with $\Aut_\QQ(S)\cong(\ZZ/2\ZZ)^2$ or $\ZZ/4\ZZ$ . In Section~\ref{sec: G4}, we use the results obtained in the previous sections to prove our main results.  

In Section~\ref{sec: ex}, by extending the spherical systems of the surfaces in \cite{BC04},  we provide unbounded series of regular surfaces of general type with $\au\cong(\ZZ/2\ZZ)^2$ which have a fibration  of genus 3 or 5, with one of them already known and the others new. In our examples, the surfaces $S$ are isogenous to a product of curves of unmixed type. In particular, these surfaces have
an ample canonical class and their invariants satisfy
$K_S^2=8\chi(\mc O_S)$.

\medskip

\noindent{\bf Notation and Conventions.} 

\medskip

We work over the complex numbers $\CC$ throughout the paper. 

The following notions and invariants are attached to a smooth projective surface $S$:
\begin{itemize}[leftmargin=*]
    \item $K_S$ denotes a canonical divisor of $S$.
    \item $p_g(S):=\dim_\CC H^0(S, K_S)$ is the geometric genus of $S$.
    \item $q(S):=\dim_\CC H^1(S, \mo_S)$ is the irregularity of $S$, and the surface $S$ is called \emph{regular} if $q(S)=0$.
    \item $\chi(\mo_S)=1-q(S)+p_g(S)$ is the holomorphic Euler characteristic of $S$.
    \item $e(S)=\sum_{i=0}^{4}(-1)^i \dim_\CC H^i(S, \CC)$ denotes the topological Euler  characteristic of  $S$.
    \item $\kappa(S)$ denotes the Kodaira dimension of $S$.
    \item For two divisors $D=\sum_i m_i D_i$ and $D'=\sum_i m_i' D_i$, we write $D\wedge D':=\sum_i \min\{m_i, m_i'\} D_i$, and $D\equiv D'$ if they are numerically equivalent.
\end{itemize} 

For a complex algebraic variety $X$, 
\begin{itemize}[leftmargin=*]
   \item $X_{\sm}$ denotes the smooth locus of $X$ and $X_{\sing}$ denotes the singular locus of $X$;
    \item $\Aut(X):=\{\sigma\mid \sigma\colon X\rightarrow X \text{ is an isomorphism}\}$ denotes \emph{the} automorphism group of $X$;
    \item $\Aut_\QQ(X):=\{\sigma\in \Aut(X)\mid \sigma \text{ induces trivial action on }H^*(X,\QQ)\}$;
    \item $\Aut_\ZZ(X):=\{\sigma\in \Aut(X)\mid \sigma \text{ induces trivial action on }H^*(X,\ZZ)\}$.
\end{itemize}

Let $G$ be a finite group acting on a set $X$. 
\begin{itemize}[leftmargin=*]
\item $|G|$ denotes the order of $G$. Similarly, for $\sigma\in G$, its order is denoted by $|\sigma|$.
    \item For any $\sigma\in G$, $X^\sigma:= \{p\in X \mid \sigma(p)=p\}$ is called the \emph{$\sigma$-fixed subset}.
    \item $X^G:=\{p\in X\mid p\in X^\sigma \text{ for any }\sigma\in G\}$ is called the \emph{$G$-fixed subset}.
    \item  A subset $A\subset X$ is called \emph{$\sigma$-invariant} (resp.~\emph{$G$-invariant}) if $\sigma(p)\in A$ for any $p\in A$ (resp.~if it is $\sigma$-invariant for any $\sigma\in G$).
    \item For a subset $A\in X$, $G_A:=\{\sigma\in G\mid \sigma(p)=p \text{ for any } p\in A \}$ denotes the (pointwise) \emph{stabilizer} of $A$, and we say $A$ is $\sigma$-fixed for any $\sigma\in G_A$.
    \item For $p\in X$, $[p]_G:=\{\sigma(p)\in X\mid \sigma\in G\}$ is the $G$-\emph{orbit} of $p$, and $X/G=\{[p]_G\mid p\in X\}$ denotes the set of all $G$-orbits. In general, for $A\subset X$, the orbit of $A$ is $[A]_G:=\cup_{p\in A} G(p)$.
\end{itemize}  
Suppose, additionally, that $G$ is abelian and $X=V$ is a linear representation of $G$, we denote by $\widehat G$ the character group $\Hom(G, \CC^*)$ of $G$; for any $\chi\in \widehat{G}$ we write
$V^\chi=\{v\in V \mid \sigma(v)=\chi(\sigma)v \ \textrm{for all}\ g\in G\}$. Note that $V^\chi$ is a $G$-invariant subspace of $V$ that is fixed by $\ker\chi$, where $\ker\chi = \{g\in G\mid \chi(g)=1\}$.
\separate

\noindent{\bf Acknowledgements.} This work has been supported by the  NSFC (No.~11971399). The second author enjoyed discussions with V\'ictor Gonz\'alez-Alonso about fibered surfaces while working at Leibniz Universit\"at Hannover during the years 2013-2015.

\section{Local behavior of automorphisms}\label{sec: local}
\subsection{Cartan's lemma for finite abelian group actions on surfaces}

Suppose that $S$ is a smooth complex surface, and let $G\subset \Aut(S)$ be a finite abelian subgroup. By  Cartan's lemma \cite[Section~4]{Car57}, for any $p\in S$, there is a $G_p$-invariant neighborhood $U$ in the analytic topology together with local coordinates $x_1, x_2$ on $U$ such that $p=(0,0)$, and the action of $G_p$ takes a linear form
\[
\sigma(x_1, x_2) = (\chi_1(\sigma)x_1, \chi_2(\sigma) x_2)
\]
for any $\sigma\in G_p$, where $\chi_i\in \widehat G_p$ are characters of $G_p$ such that $\ker(\chi_1)\cap \ker(\chi_2) = \{\id_S\}$.  The point $p$ is isolated in $S^\sigma$ if and only if $\sigma\notin\ker(\chi_i)$ for $i=1,2$. On other hand, if $\sigma\in \ker(\chi_i)$ then $\sigma\notin \ker(\chi_{j})$, and $S^\sigma\cap U =(x_j=0)$, where $\{i,j\}=\{1,2\}$. It follows that $S^\sigma$ is a smooth for any $\sigma\neq \id_S\in G$, though it can be disconnected and have components of different dimensions.

In case $G_p = \langle \sigma\rangle$ is cyclic, say of  order $r$, we may assume that $\chi_1(\sigma)=\xi$ is a primitive $r$-th root and $\chi_2(\sigma) = \xi^a$ for some $0\leq a\leq r-1$; the fixed point $p$ is then said to be \emph{of type $\frac{1}{r}(1,a)$}. 

Let $C\subset S$ be an integral curve. By the discussion above, we have $G_C=\{\id_S\}$ if $C$ is singular. Now suppose that $C$ is smooth. Then for a general point $p\in C$, $G_C$ acts trivially on the tangent space $T_pC$ and faithfully on $T_pS$. It follows that $G_C$ acts faithfully on $T_pS/T_pC\cong \CC$. Thus there is an injective homomorphism $G_C\rightarrow \CC^*=\GL(T_pS/T_pC)$, and hence $G_C$ is a cyclic group.

\begin{lem}\label{lem: l-rank}
Let $S$ be a smooth surface, and $G\subset\Aut(S)$ a finite abelian subgroup of automorphisms.  Then for any $p\in S$ and prime number $l$, the $l$-component of $G_p$ has at most two direct summands.
\end{lem}
\begin{proof}
    Suppose on the contrary that the $l$-component of $G_p$ has $k$ direct summands, then it contains a subgroup $H\cong (\ZZ/l\ZZ)^k$. Take local coordinates $(x_1,x_2)$ around $p$ such that $p=(0,0)$ and the action of $H$ is of the form $\sigma(x_1,x_2)=(\chi_1(\sigma)x_1, \chi_2(\sigma)x_2)$ for any $\sigma\in H$, where $\chi_1$ and $\chi_2$ are two irreducible characters of $H$. The group $H$ can be viewed as a vector space of dimension $k$ over the finite field $\ZZ/l\ZZ$. If $k\geq 2$, then $\ker\chi_1$ and $\ker\chi_2$ are subspaces of dimension $k-1$, and their intersection has dimension $\geq k-2$. Since the action of $H$ on $S$ is faithful, we have $\ker\chi_1\cap \ker\chi_2 = \{\id_S\}$, and it follows that $k\leq 2$.
\end{proof}

\begin{lem}\label{lem: transverse}
Let $S$ be a smooth surface, and $\sigma\in \Aut(S)$ an automorphism of order $r>1$. Let $C$ be a $\sigma$-fixed curve, and $D$ another $\sigma$-invariant curve such that $C\cap D\neq \emptyset$. If $p\in C\cap D$ satisfies $\mult_p D<r$, then $D$ has exactly one tangent direction at $p$, and it is transversal to that of $C$. Consequently, in this case, $(C\cdot D)_p=\mult_p D$. 
\end{lem}
\begin{proof}
Since $C$ is fixed by $\sigma$, it is a smooth curve. Choose local coordinates $(x,y)$ around $p$ such that 
\[
p=(0,0),\quad C=\{x=0\},\quad D=\{h(x,y)=0\},\quad \sigma(x,y)=(\zeta x, y),
\]
where $\zeta$ is a primitive $r$-th root of unity and $h(x,y)$ is a holomorphic function such that $h(0,0)=0$. Denote $k=(C\cdot D)_p$. Then, by the Weierstrass  preparation theorem (see \cite[Page 68]{GR65}), we can choose $h$ to be a Weierstrass  polynomial in $y$
\begin{equation}\label{eq: h 1}
    h(x,y) = y^k + a_1(x) y^{k-1} +\cdots + a_k(x)
\end{equation}
where $a_i(x)$ is a holomorphic function in $x$ with $a_i(0)=0$ for each $1\leq i\leq k$. Since $D$ is $\sigma$-invariant, by the uniqueness part of Weierstrass  preparation theorem,  $h(\zeta x, y) = h(x,y)$ and hence
 $a_i(\zeta x)=a_i(x)$ for $1\leq i\leq k$. It follows that there are holomorphic functions $\tilde a_i(z)$ such that $\tilde a_i(x^r) = a_i(x)$ for $1\leq i\leq k$, so the order of $a_i(x)$ at $x=0$ is a multiple of $r$. Now, in view of \eqref{eq: h 1},  $\mult_p D<r$ can happen only if $\mult_p D = k<r$. Therefore, $y^k$ is the term of the lowest order in $h(x,y)$. The tangent direction of $D$ at $p$ is given by $y=0$, which is transversal to that of $C$, given by $x=0$.
\end{proof}

\begin{lem}\label{t1}
Let $S$ be a smooth surface and $\sigma\in\Aut(S)\setminus\{\id_S\}$ an automorphism of finite order. Let $D$ be the divisorial part of $S^\sigma$. Then for any smooth proper $\sigma$-invariant  curve $C\subset S$ such that $C\not\subset S^\sigma$, we have $D\cdot C\leq2g(C)+2 $.
\end{lem}
\begin{pf}
Since $C$ is smooth, $C$ and $D$ intersect transversely by Lemma~\ref{lem: transverse}. Thus
\[
D\cdot C \leq \# C^\sigma \leq 2g(C)+2
\]
where the second inequality is by the Riemann--Hurwitz formula applied to the quotient map $C\rightarrow C/\langle\sigma\rangle$.
\qed\end{pf}

\begin{lem}[{cf.~\cite[Lemma 2]{Pet79}, \cite[3.3]{Cai12}, and \cite[Lemma 1.1]{CL18}}]\label{lem: fix vs base}
Let $S$ be a \\ smooth projective surface with $p_g(S)>0$, and $\sigma\in \Aut(S)\setminus\{\id_S\}$ an automorphism of finite order $r>1$ acting trivially on $H^0(S, \omega_S)$. Then the following holds.
 \begin{enumerate}
\item If $p\in S^\sigma$ is weight $\frac{1}{r}(1, a)$ with $a\not=r-1$, then it is contained in the base locus $\Bs|K_S|$.
 \item If $\Gamma$ is a $\sigma$-fixed curve, then $(r-1)\Gamma\leq Z$, where $Z$ is the fixed part of $|K_S|$.
 \end{enumerate}
\end{lem}
\begin{proof}
(i)  We need to show that $\omega(p)=0$ for any $\omega\in H^0(S, \omega_S)$. Take suitable local coordinates $(x,y)$ of
$S$ around $p$ such that $p=(0,0)$, and $\sigma(x,y)=(\zeta x,\zeta^a y)$ where $\zeta$ is a primitive $r$-th root of unity. Locally around $p$, we can write 
$\omega=c(x, y)\op{d}x\wedge \op{d}y$ for some holomorphic function
$c(x, y)$. Since $\omega$ is $\sigma$-invariant,
we have 
\[
c(\zeta x, \zeta^ay)\zeta^{1+a}\op{d}x\wedge \op{d}y=c(x,y)\op{d}x\wedge \op{d}y.
\]
It follows that $c(0,0) =c(0,0)\zeta^{1+a}$. Since $0\leq a< r-1$, we infer that $\zeta^{1+a}\neq 1$ and $c(0,0)$ must be $0$ by the above equality. It follows that $\omega(p)=0$.

(ii) Take a general point $q$ of $\Gamma$, and local coordinates $(x,y)$ around $q\in S$ such that $q=(0,0)$, $\Gamma$ is given locally by $x=0$, and the action of $\sigma$ is given by  $\sigma(x, y)=(\zeta x, y)$, where $\zeta$ is a primitive $r$-th root of unity. Then for
 any non-zero holomorphic $2$-form $\omega\in H^0(\omega_S)$, we can write
$\omega=c(x, y)x^m\op{d}x\wedge \op{d}y$ for some holomorphic function
$c(x, y)$, where $m=\op{mult}_\Gamma Z$. Since $\omega$ is $\sigma$-invariant,
we have \[c(\zeta x, y)\zeta^{m+1}x^m\op{d}x\wedge \op{d}y=c(x,
y)x^m\op{d}x\wedge \op{d}y.\] This implies   $m+1$ is  divisible by $r$, and hence $\op{mult}_\Gamma Z\geq r-1$.
\end{proof}

\subsection{Involutions}
\begin{lem}\label{lem: inv mod 2}
Let $S$ be a smooth surface and $\sigma\in \Aut(S)$ an involution. Let $C\subset S$ be a proper integral curve that is invariant
but not fixed by $\sigma$. Let $\{p_1,\dots,p_l\}$ be the set of isolated points of $S^\sigma$ lying on $C$, and $D_1,\
\dots,\ D_m$ the $\sigma$-fixed curves intersecting $C$. Then 
\[
K_S\cdot C\equiv C^2 \equiv  \sum_{1\leq i\leq l} \mult_{p_i}C\equiv\sum_{1\leq j\leq m}D_j\cdot C  \mod 2.
\]
\end{lem}
\begin{pf} 
The first congruence is by the adjunction formula $2 p_a(C) - 2 = K_S\cdot C+C^2$. 

Let  $\rho:
\tilde{S}\to S$ be the blow-up of $\{p_1,\dots,p_l\}$. Then the action of $\sigma$ lifts to $\tS$. Let $\tilde{\pi}:
 \tilde{S}\to \tilde{S}/\langle\sigma\rangle$ be the quotient
 map.  We denote by $\tilde C$ and $\tilde D_j$ the strict
transforms of $C$ and $D_j$ on $\tilde{S}$, and let $\oC$ and $\oD_j$ be their images in $\tilde{S}/\langle\sigma\rangle$ respectively. Then we have 
\[
\tilde C^2=C^2-\sum_{1\leq i\leq l}\left(\mult_{p_i}C\right)^2.
\]
On the other hand, since $C$ is $\sigma$-invariant but not $\sigma$-fixed, we have $\tilde C=
\tilde{\pi}^*\oC$ and hence $\tilde C^2=
2\oC^2$ is even. It follows that
\[
C^2 - \sum_{1\leq i\leq l}\mult_{p_i}C \equiv C^2 - \sum_{1\leq i\leq l}\left(\mult_{p_i}C\right)^2 = \tilde C^2 \equiv 0 \mod 2
\]
This completes the proof of the second congruence relation.

Let $\tilde E_i\subset \tS$ be the $(-1)$-curve over $p_i$ for $1\leq i\leq l$ and $\oE_i$ the image of $\tilde E_i$ in $\tilde{S}/\langle\sigma\rangle$. Note that $\ti\pi$ is a double cover  branched along $\sum_j\oD_j+\sum_i \oE_i$, so  $\sum_j\oD_j+\sum_i \oE_i\sim2 L$ for some $L\in \op{Pic}(\tilde{S}/\langle\sigma\rangle)$. Therefore,
\begin{equation}\label{eq-2}\sum_{1\leq j\leq m}\oD_j\cdot\overline C+\sum_{1\leq i\leq l}\oE_i\cdot\overline C=2L\cdot\overline C\equiv0\mod2.\end{equation}
Note that 
\begin{equation}\label{eq: Dj}
\oD_j\cdot\overline C = \frac{1}{2}(\ti{\pi}^*\oD_j)\cdot(\ti{\pi}^*\overline C) = \tilde D_j\cdot \tilde C = D_j\cdot C, \quad 1\leq j\leq m
\end{equation}
where the second equality is because $\ti{\pi}^*\oD_j=2\tilde{D}_j$ and $\ti{\pi}^*\overline C=\tC$, and the third equality is because $\rho\colon\tS\rightarrow S$ is an isomorphism along $D_j$. Similarly, we have 
\begin{equation}\label{eq: Ei}
\oE_i\cdot\overline C=\ti E_i\cdot\ti C = \mult_{p_i} C,\,\, 1\leq i\leq l.
\end{equation}
Plugging \eqref{eq: Dj} and \eqref{eq: Ei} into \eqref{eq-2} gives the third congruence relation. 
\qed\end{pf}

\begin{lem}\label{lem: cusp}
Let $S$ be a smooth surface $S$ and $\sigma\in \Aut(S)$ an involution. Let $C\subset S$ be a proper integral curve that is invariant by the action of $\sigma$.  Suppose that there is a cusp $p\in C$ fixed by  $\sigma$.   Then the following holds.
\begin{enumerate}[leftmargin=*]
\item There is a $\sigma$-fixed curve $D$ with $(D\cdot C)_p= 3$.
\item Let $\bar p$ and $\oD$ be the images of $p$ and $D$ in $S/\langle\sigma\rangle$ respectively. Then both $ S/\langle\sigma\rangle$ and $\oD$ are smooth at $\bar p$.
\end{enumerate}
\end{lem}
\begin{proof} 
(i) Suppose on the contrary that $p$ is an isolated point of $S^\sigma$. Let  $\rho:
\tilde{S}\to S$ be the blow-up of $p\in S$, with exceptional divisor $E$. Then the action of $\sigma$ lifts to $\tS$, and $E$ is a $\sigma $-fixed curve. Let $\tilde C$ be the strict
transform of $C$ in $\tilde{S}$. Then $\tilde C$ and $E$ intersect at exactly one point $\tilde p$, and we have
$T_{\tilde{p}}\tC=T_{\tilde{p}}E$, where
$T_{\tilde{p}}\tC$, $T_{\tilde{p}}E$ are  the tangent spaces of
$\tC$, $E$ at point $\tilde{p}$, viewed as subspaces of $T_{\tilde p}\tS$. Since $E$ is
$\sigma$-fixed, the induced action of
$\sigma$ on $T_{\tilde{p}}E$ is trivial. So the induced
action of $\sigma$ on $T_{\tilde{p}}\tC$   is
trivial. This implies that $\tC$ is $\sigma$-fixed. This is impossible, since $\tS^\sigma$ should be smooth.

Therefore, there is a $\sigma$-fixed curve $D\subset S$ passing through $p$. We still need to show that $(D\cdot C)_p= 3$. If not, then we have $(D\cdot C)_p=2$, that is, the tangent directions of $D$ and $C$ at $p$ are different. Take local coordinates $(x,y)$ around $p\in S$ so that 
\[
p=(0,0),\, D=\{x=0\},\, C=\{y^2 + a_1(x)y + a_2(x)=0 \},\, \text{ and } \sigma(x,y)=(-x,y).
\]
where $a_1(x)$ and $a_2(x)$ are holomorphic functions such that $a_1(0)=a_2(0)=0$. Since $C$ is $\sigma$-invariant, we see that $a_i(-x)=a_i(x)$ for $i=1,2$. It follows that there are holomorphic functions $\tilde a_i(z)$ such that $a_i(x) =\tilde a_i(x^2)$. But then one verifies easily that $p\in C=\{y^2 + \tilde a_1(x^2)y + \tilde a_2(x^2)=0 \}$ is not a cusp, which is a contradiction.

(ii) By (i) we can find local coordinates $(x,y)$ around $p$ such that
\[
p=(0,0),\, D=\{x=0\},\, C=\{x^2 + y^3 =0 \},\, \text{ and } \sigma(x,y)=(-x,y).
\]
The quotient surface $S/\langle\sigma\rangle$ is smooth with local coordinates $(u,v)=(x^2, y)$ around $\bar p$ and the curve $\oD$ is defined by $u+v^3=0$, which is also smooth.
\end{proof}

\begin{lem}\label{lem: Z22}
Let $S$ be a smooth surface and  $G=\{\id_S, \sigma_1, \sigma_2, \sigma_3\}\subset\op{Aut}(S)$ be a subgroup of automorphisms such that
$G\cong(\ZZ/2\ZZ)^2$. 
\begin{enumerate}
    \item Suppose that $p\in S$ is fixed by the whole group
$G$. Then exactly two of the involutions in $G$ fix curves that pass
through $p$, and each of the these two curves is invariant under the action of $G$.
\item Let $D$ be a 1-dimensional connected component of $\bigcup_{i=1}^{3} S^{\sigma_i}$. Assume that $D$ is proper. Then $D$ is a 2-connected semi-stable curve, and it is stable if and only if $D$ contains no rational curves. In particular, we have $p_a(D)\geq 1$. 
\end{enumerate}
\end{lem}
\begin{proof}
(i) There are local coordinates $(x,y)$ around $p\in S$ such that each
$\sigma\in G$ acts as $(x,y)\mapsto
(\chi_1(\sigma)x,\chi_2(\sigma)y)$ where $\chi_1,\chi_2\in\widehat
G$ are two distinct characters of $G$. There are two involutions,
say $\sigma_1$ and $\sigma_2$, such that $\ker(\chi_i) = \langle
\sigma_i \rangle$, $i=1,2$. Then $\sigma_i$ ($i=1,2$) fixes curves
passing through $p$ while $\sigma_3$ has only
isolated fixed point at $p$. It is also clear that the $\sigma_1$-fixed curve $(y=0)$ and the $\sigma_2$-fixed curve $(x=0)$ are both invariant under the action of $G$.

(ii) Write $D=\sum_j D_j$ as the sum of irreducible components. Then each $D_j$ is smooth. If $p\in D_{j}\cap D_{j'}$ for $j\neq j'$, then $G_{D_j}\cap G_{D_{j'}}=\{\id_S\}$, and hence $G=\langle G_{D_j}, G_{D_{j'}}\rangle \subset G_p$. It follows that $G_p=G$. By the local analysis done in (i), or by Lemma~\ref{lem: transverse}, $D_{j}$ and $D_{j'}$ intersect transversely. If $D_j$ is a rational curve fixed by, say $\sigma_1$, then $\sigma_2$ acts faithfully on $D_j$, and $D-D_j$ and $D_j$ intersect exactly at the two points $(D_j)^{\sigma_2}$. The statements of (ii) follow from these observations.
\end{proof}

\begin{lem}\label{lem: Z22 chain}
 Let $G\subset\Aut(S)$ be a subgroup of automorphisms such that
$G\cong(\ZZ/2\ZZ)^2$. Suppose $D=\cup_{1\leq i\leq k} D_i \subset S$ is a connected tree of smooth rational curves intersecting transversely and every component $D_i$ is invariant under the action of $G$. Then
\begin{enumerate}
 \item $D$ is a chain, that is, its dual graph is as follows:
 \[
 \begin{tikzpicture}
\node[bbullet](D0)at(0,0){};
\node[bbullet](D1)at(1,0){};
\node (D11)at(2,0){};
\node (D12)at(3,0){};
\node[bbullet](D2)at(4,0){};
\node[bbullet](D3)at(5,0){};
\draw (D0)--(D1)--(D11);
\draw[dotted] (D11)--(D12);
\draw (D12)--(D2)--(D3);
 \end{tikzpicture}
 \]
 \item Suppose moreover that $k\geq 2$ and each $D_i$ ($1\leq i\leq k$) has an even self-intersection. Then there are two involutions in $G$, say $\sigma_1$ and $\sigma_2$, such that the stabilizer of the curve $D_i$ is $\sigma_1$ (resp.~$\sigma_2$) if $i$ is odd (resp.~even).
\end{enumerate}
\end{lem}
\begin{proof}
 (i) It suffices to prove that each $D_i$ intersects at most two other components: otherwise there are at least three distinct components, say $D_{i_1},D_{i_2}$ and $D_{i_3}$, intersecting $D_i$ at some points $P_1$, $P_2$ and $P_3$ respectively. Since each of the curves is invariant under the action of $G$, the points $P_1$, $P_2$ and $P_3$ are fixed by the whole group $G$, which implies that the curve $D_i$ is fixed by $G$. This is absurd, since the stabilizer of a curve is cyclic.

(ii) Denote by $P_i$ the unique intersection point of $D_i\cap D_{i+1}$ ($1\leq i\leq k-1$). Then $P_i$ is fixed by the whole group $G$. Each component $D_i$ being invariant under $G$, it must be fixed by some involution $\sigma\in G$. Denote the stabilizers of $D_1$ and $D_2$ by $\sigma_1,\sigma_2$ respectively.

By induction on the index $i$ it suffices to show that the stabilizer of $D_3$ is $\sigma_1$. Since the fixed locus of $\sigma_2$ is smooth, $D_3$ is not fixed by $\sigma_2$. The curve $D_3$ is not fixed by $\sigma_1\sigma_2$ either: otherwise $\{P_2\} = D_2\cap D_3$ would be the unique isolated fixed point of $\sigma_2$ on $D_2$, contradicting the fact that the self-intersection $D_2^2$ is even (Lemma~\ref{lem: inv mod 2}).
\end{proof}

\begin{cor}\label{cor: ADE on S}
 Let $S$ be a surface of general type with $\au$ containing a subgroup $G\cong (\ZZ/2\ZZ)^2$. Then the canonical model of $S$ has at most $A_n$-singularities.
\end{cor}
\begin{pf}
 The group action of $G$ leaves every ($-2$)-curve of $S$ invariant. Now apply Lemma~\ref{lem: Z22 chain} to the fundamental cycles on $S$.
\end{pf}

\begin{lem}\label{lem: Z2n quotient}
Let $S$ be a smooth surface and $G\subset\op{Aut}(S)$ be a subgroup of automorphisms such that $G\cong (\ZZ/2\ZZ)^n$ with $n\geq 1$. Let $\sA$ be the set of isolated points of $\bigcup_{\sigma\in G\setminus\{\id_S\}}S^{\sigma}$. Then the quotient surface $S/G$ has $A_1$-singularities at the points of $\pi(\sA)$ and smooth elsewhere, where $\pi\colon S\rightarrow S/G$ is the quotient map.
\end{lem}
\begin{proof}
    For any $p\in S$, the stabilizer $G_p$ is isomorphic to $(\ZZ/2\ZZ)^r$ for some $0\leq r\leq 2$ by Lemma~\ref{lem: l-rank}. If $r=0$, then $G_p$ is trivial and hence $p\notin \sA$; it is clear that $\pi(p)\in S/G$ is smooth in this case. If $r=2$, then by the proof of Lemma~\ref{lem: Z22}, $p\notin \sA$ and there are local coordinates $(x,y)$ around $p\in S$ such that $\pi(p)\in S/G$ is smooth with local coordinates $u=x^2, v=y^2$.
    
    If $r=1$ then $G_p=\langle \sigma\rangle$ for some $\sigma\in G\setminus\{\id_S\}$. If $p\notin \sA$, then we may find local coordinates $(x,y)$ around $p\in S$ such that the action of $G_p$ is given by $\sigma(x,y)=(-x,y)$, and hence $\pi(p)\in S/G$ is smooth with local coordinates $u=x^2$ and $v=y$. If $p\in \sA$ then the action is given by $\sigma(x,y)=(-x,-y)$ and $\pi(p)\in S/G$ is an $A_1$-singularity.
\end{proof}

\begin{lem}\label{lem: Z2n C}
Let $S$ be a smooth surface and $G\subset\op{Aut}(S)$ be a subgroup of automorphisms such that $G\cong (\ZZ/2\ZZ)^n$ with $n\geq 2$. Let $C$ be a proper integral curve that is invariant under the action of $G$, and $\sA=\{p_1,\dots,p_k\}$ the set of isolated points of $\bigcup_{\sigma\in G\setminus\{\id_S\}}S^{\sigma}$ lying on $C$. Then the following holds.
\begin{enumerate}[leftmargin=*]
\item If $G$ acts faithfully on $C$, then $C^2\equiv \sum_{i=1}^{k} m_i^2\mod 2^n$, and $C^2$ is divisible by $2^{n-1}$. 
\item If $G$ does not act faithfully on $C$ then $C$ is smooth, and $C^2$ is divisible by $2^{n-2}$. 
\item $C^2$ and $K_S\cdot C$ are even whenever $n\geq 3$.
\item If $n\geq 3$ and $C\cong\PP^1$, then $n=3$ and $G_C$ is nontrivial.
\item If $n\geq 3$ and $C$ is singular, then $p_a(C)\geq 2^{n-1}-1$, and equality holds if and only if the following holds:
\begin{enumerate}
    \item $C$ is a nodal curve with geometric genus $2^{n-2}-1$;
    \item $C_\sing$ consists of exactly one orbit of $G$, and $|G_p|=4$ for each $p\in C_\sing$.
\end{enumerate}
\end{enumerate} 
\end{lem}
\begin{proof}
(i) Let $\rho\colon \tS\rightarrow S$ be the blow up of the set $\sA$ and $\tilde C\subset \tS$ the strict transform of $C$. Then the action of $G$ lifts to $\tS$, and $\tS/G$ is smooth. Let $\overline C\subset \tS/G$ be the image of $\tilde C$. Then we have
\[
\tilde C^2 = C^2 - \sum_{1\leq i \leq k} m_i^2, \quad \tilde C^2 = |G|\overline C^2= 2^{n} \overline C^2.
\]
and it follows that $C^2 \equiv \sum_{i=1}^{k} m_i^2 \mod 2^n$.

Since $p_i$ is isolated in $\bigcup_{\sigma\in G\setminus\{\id_S\}}S^{\sigma}$, we have $G_{p_i}\cong \ZZ/2\ZZ$ by Lemma~\ref{lem: Z22}. It follows that each $G$-orbit in  $\sA$ consists of $2^{n-1}$ points. Note that $m_i = m_j$ if $p_i$ and $p_j$ lie in the same $G$-orbit. We have thus $2^{n-1}\mid \sum_{i=1}^{k} m_i^2 $, and $2^{n-1}\mid C^2$ by the previous congruence relation. 

\medskip

(ii) Now suppose that the the pointwise stabilizer $G_C$ of $C$ is nontrivial. Then $\tS=S$, $C$ is smooth, and $G_C$ is cyclic and hence $G_C\cong\ZZ/2\ZZ$. It follows that
\[
C^2 = \left(\frac{1}{2}\pi^*\overline C\right)^2 = \frac{1}{4}\left(\pi^*\overline C\right)^2 = \frac{1}{4} |G| \overline C^2 =2^{n-2} \overline C^2
\]
is divisible by $2^{n-2}$.

\medskip

(iii) If $n\geq 3$ then $C^2$ is even by (i) and (ii). The evenness of  $K_S\cdot C$ follows from that of $C^2$ by the adjunction formula.

\medskip

(iv) If $C\cong\PP^1$ then $C/G\cong\PP^1$. Applying the Riemann--Hurwitz theorem to the quotient map $ C\rightarrow C/G$, we obtain
\[
-2 = \frac{|G|}{|G_C|} \left(-2+\frac{k}{2}\right) =\frac{2^{n}}{|G_C|}\left(-2+\frac{k}{2}\right) 
\]
where $k$ is the number of branch points of $\pi\rest{C}$. Since $n\geq 3$ and $|G_C|\leq 2$, it can only happen that $n=3$, $|G_C|=2$, and $k=3$.

\medskip

(v) Let $p\in C$ be a singularity. Then pointwise stabilizer $G_C$ of $C$ is trivial, and $|G_p|\in\{1,2,4\}$ by Lemma~\ref{lem: l-rank}. The orbit $[p]_G$ consists of $\frac{|G|}{|G_p|}$ points, all of which are  singularities of $C$ of the same type. If $p\in C$ is neither a cusp nor a node then 
\[
p_a(C) \geq g( C^\nu) + \frac{2 |G|}{|G_p|} \geq  \frac{|G|}{2} = 2^{n-1}
\] 
where $\nu\colon C^\nu\rightarrow C$ denotes the normalization of $C$. If $p\in C$ is a cusp then $|G_p|\leq 2$ by Lemma~\ref{lem: cusp cyclic}, so
\[
p_a(C) \geq g( C^\nu) + \frac{|G|}{|G_p|} \geq  \frac{|G|}{2} = 2^{n-1}
\]
If $p\in C$ is a node and $|G_p|\leq 2$ then $p_a(C) \geq  2^{n-1}$ as in the cusp case. 

Finally, if $p\in C$ is a node and $|G_p|=4$, then there is an involution $\sigma\in G_p$ such that $[p]_G$ are isolated $\sigma$-fixed points by Lemma~\ref{lem: Z22}. Note that the action of $G$ on $C$ lifts to $C^\nu$. Since $\sigma$ preserves the local branches of $C$ at the points of $[p]_G$, $\sigma$ fixes the inverse image $\nu^{-1}([p]_G)$ which consists of $2|G|/|G_p| = 2^{n-1}$ points. By the Riemann-Hurwitz formula, 
\begin{equation}\label{eq: C1}
2g(C^\nu)-2 \geq 2(g(C^\nu/\langle\sigma\rangle)-2) + 2^{n-1} \geq 2^{n-1} -4
\end{equation}
and it follows that $g(C^\nu)\geq 2^{n-2} -1$. We obtain in this case
\begin{equation}\label{eq: C2}
p_a(C) \geq g( C^\nu) + \frac{|G|}{|G_p|} \geq  2^{n-2} -1 + \frac{|G|}{4} = 2^{n-1} -1.
\end{equation}
In case $p_a(C)=2^{n-1}-1$, the equalities of \eqref{eq: C1} and \eqref{eq: C2} hold, and hence $g(C^\nu)=2^{n-2}-1$, $|G_p|=4$, and $C_{\sing}$ consists of exactly on $G$-orbit.
\end{proof}

\begin{lem}\label{lem: cusp cyclic}
Let $C$ be an irreducible curve with a cusp $p$ and $G\subset\Aut(C)$ a finite group of automorphisms. Then  the stabilizer $G_p$ is cyclic.
\end{lem}
\begin{proof}
Let $\nu\colon C^\nu\rightarrow C$ be the normalization. Then the action of $G$ lifts to $C^\nu$, and $G_p=G_{p^\nu}$ is cyclic, where $p^\nu =\nu^{-1}(p)$.
\end{proof}

The following lemma is a consequence of the
topological and holomorphic Lefschetz fixed-points formulae.
\begin{lem}[{\cite[Lemma~2.1]{Cai12}}]\label{lem: sign inv}
Let $S$ be a smooth projective surface, and $\sigma\in\Aut(S)$ an involution acting trivially in $H^2(S, \mb
Q)$. Then 
\[K_S^2=8\chi(\mc O_S)+\sum_{i=1}^m D_i^2,
\]
where   $D_1,\
\dots,\ D_m$ ($m\geq0$) are $\sigma$-fixed curves.
\end{lem}

\subsection{Fibred surfaces and fiber-preserving automorphisms}
\begin{defn}
A \emph{fibration} in this paper means a projective morphism $f\colon S\rightarrow B$ from a normal algebraic surface onto a smooth curve such that the fibers of $f$ are connected; the surface $S$ is then called a \emph{fibred surface}; the genus of $f$ is defined to be that of a smooth fiber. 

Let $f\colon S\rightarrow B$ be a fibration.  An irreducible curve $C\subset S$ is \emph{vertical} with respect to $f$ if $f(C)$ is a point; otherwise $C$ is said to be \emph{horizontal}. For a vertical divisor $D$, we denote $D_b=D\wedge f^*b$ for $b\in B$. Here, for two divisor $D=\sum_i m_i D_i$ and $D'=\sum_i m_i' D_i$ on $S$, the divisor $D\wedge D'$ is defined to be $\sum_i \min\{m_i, m_i'\}D_i$. 

We say a divisor $D$ on $S$ is \emph{$f$-ample} (resp.~\emph{$f$-nef}) if $D\cdot C>0$ (resp.~$D\cdot C\geq 0$) for any vertical curve $C$; in this situation, we also say that $D$ is \emph{relatively ample} (resp.~\emph{relatively nef}) over $B$. The fibration is called \emph{relatively minimal} if $K_S$ is relatively nef over $B$.
\end{defn}

\begin{lem}\label{lem: subcurve smaller g}
Let $S$ be a smooth surface and $f\colon S\rightarrow B$ relatively minimal fibration of genus $g\geq 2$. The for any sub-curve $C\subsetneq f^*b$ we have $p_a(C)<g$.
\end{lem}
\begin{proof}
By the adjunction formula, we have $g=p_a(f^*b) = \frac{1}{2}K_S\cdot f^*b + 1$. Since $g\geq 2$, we have $K_S\cdot f^*b>0$.
If $C\equiv tf^*b$ for some $t$, then $0<t<1$, $C^2=0$, and hence
\[
p_a(C) = \frac{1}{2}(K_S+C)\cdot C+1 = \frac{1}{2}K_S\cdot C+1 = \frac{t}{2}K_S\cdot f^*b +1 <g.
\]
If $C$ is not numerically equivalent to $tf^*b$ for any $t$, then $C^2<0$, and hence
\[
p_a(C) = \frac{1}{2}(K_S+C)\cdot C+1 < \frac{1}{2}K_S\cdot C+1 \leq \frac{t}{2}K_S\cdot f^*b +1 =g.
\]
\end{proof}

\begin{defn}
Let $S$ be a smooth projective surface and $f\colon S\rightarrow B$ a fibration onto a curve. An automorphism $\sigma\in\Aut(S)$ is called \emph{fiber-preserving} if each fiber of $f$ is $\sigma$-invariant, that is, $f\circ \sigma =f$ holds; the group of such automorphisms is denoted by $\Aut_B(S)$. An automorphism $\sigma$ of $S$ is called \emph{vertical curve preserving} if any vertical curve of $f$ is $\sigma$-invariant. Obviously, vertical curve preserving automorphisms are fiber-preserving. 
\end{defn}

\begin{rmk}
For a fibred surface $f\colon S\rightarrow B$, the group $\Aut_\QQ(S)$ is vertical curve preserving by \cite[Lemma~3.1]{Cai04} or \cite[Principle 1]{CatLiu21}.
\end{rmk}
\begin{lem-defn}\label{lem: e S^sigma}
Let $f\colon S\rightarrow B$ be a projective fibred surface whose smooth fiber $F$ has genus $g$, and $\sigma\in \Aut_B(S)$ a fiber-preserving automorphism.  Set
\[
\deg_\sigma(f) = 
\begin{cases}
e(F) & \text{if $\sigma = \id_S$} \\
\deg(H_\sigma\rightarrow B) & \text{if $\sigma\neq \id_S$}
\end{cases}
\]
where, in the case $\sigma\neq \id_S$, $H_\sigma$ is the horizontal part of $S^\sigma$ and $\deg(H_\sigma\rightarrow B)=0$ if $H_\sigma=\emptyset$. For each $b$, let $F_b=f^*b$ be the corresponding fiber, and set
\[
\epsilon_b(S^\sigma):=e(F_b^\sigma) - \deg_\sigma(f)
\]
Then one has $e(S^\sigma) = \deg_\sigma(f) e(B) + \sum_b \epsilon_b(S^\sigma)$.
\end{lem-defn}
\begin{proof}
For the case $\sigma=\id_S$, the lemma is well-known; see for example \cite[Proposition~III.11.4]{BHPV04}. For $\sigma\neq \id_S$, we can write $S^\sigma = H \cup V$, where $H$ is the (1-dimensional) horizontal part of $S^\sigma$, and $V$ consists of the (1-dimensional) vertical part and isolated $\sigma$-fixed points. We can further write \begin{equation}\label{eq: V_sigma}
e(V) = \sum_b e(V\cap F_b) = \sum_b e(F_b^\sigma) - \#(F_b\cap H)    
\end{equation}
Using cell complex structures on $H$ and $B$, one obtains a Riemann--Hurwitz type formula:
\begin{equation}\label{eq: H_sigma}
e(H) = \deg_\sigma(f)  e(B)  +  \sum_b \left(\#(F_b\cap H) - \deg_\sigma(f) \right)
\end{equation}
Plugging \eqref{eq: V_sigma} and \eqref{eq: H_sigma} into  $e(S^\sigma) = e(H)+e(V)$, we obtain the required equality of the lemma.
\end{proof}

The proof of the following lemma is similar to that of Lemma~\ref{lem: fix vs base}.
\begin{lem}[{\cite[Lemma 1.4]{Cai09}, \cite[2.3.1]{Cai01}}]\label{lem: fix vs sing fib}
Assume that  $f\colon S\rightarrow B$ is a smooth fibred surface with a fiber-preserving automorphism $\sigma$ of finite order $r$. Let $F_b = f^*b$ be a fiber of $f$ over a point $b\in B$.
\begin{enumerate}[leftmargin=*]
\item If $p\in F_b$ is an isolated fixed point
of $\sigma$, then the fiber $F_b$ is singular at $p$.
\item If $\Lambda<F_b$ is a $\sigma$-fixed curve, then
the multiplicity $\mult_\Lambda F_b$ is divisible by $r$.
 \end{enumerate}
\end{lem}
\begin{lem}[{ \cite[Lemma 2.2]{Cai12}}]\label{lem: singular'}Let $\sigma$ be an automorphism of $S$ of order $4$.
Assume that $S$ has    a  fibration $f\colon S\rightarrow B$  with
$f\circ \sigma = f$, and
  that $p\in S$ is an isolated $\sigma$-fixed point. Let $F$ be the fiber of $f$ containing $p$. Then we have
\begin{enumerate}\item
If the action of $\sigma$  at $p$ is  of weight
 $\frac{1}{4}(1, 1)$, then  $\op{mult}_p F$ is divisible by $4$;
\item  If the action of $\sigma$ at  $p$
   is of weight $\frac{1}{4}(1, 2)$ {\rm (}resp.~$\frac{1}{4}(1, 3)${\rm )}
 and  $\op{mult}_p F=2$,
then $p\in F$ is not {\rm (}resp.~is{\rm )}  an ordinary node.
 \end{enumerate}
\end{lem}
\begin{lem}\label{k1k2} Let $S$ be a smooth surface, and $\sigma\in\Aut(S)$ be an automorphism of order $4$. Let $C\subset S$ be a proper $\sigma$-invariant curve with
$C^2\leq0$. If $K_S\cdot C>0$, then there are
 at most $2K_S\cdot C$ isolated $\sigma$-fixed points on $C$.
\end{lem}
 \begin{pf}
 Let $k$ be the number of isolated $\sigma$-fixed points on $C$. We may assume that $C$ is not $\sigma$-fixed; otherwise, $k=0$ and the lemma is trivially true.

If $C$ is smooth, then
 $k\leq\frac{2g(C)}{3}+2$
 by the Hurwitz formula.   Note
that $K_S\cdot C>0$  by assumption and
  $g(C) \leq \frac{K_S\cdot C}{2}+1$ by the adjunction formula since $C^2\leq0$.   So the
lemma is true
 in this case.

 Now  we assume that $C$ is singular.
Let $\tC$ be the normalization of $C$, and
$\tilde{\sigma}$ the induced action of $\sigma$ on $\tC$.
 Let
 $\varepsilon$ be
 the number of singular points of $C$.

Note that each  $\sigma$-fixed point on $C$  is either a
singular point of $C$ or a smooth point whose inverse image in
$\tC$ is a $\tilde\sigma$-fixed point. We have
 \[k\leq\frac{2g(\tC)}{3}+2+\varepsilon.\]

Note that $K_S\cdot C\geq 2p_a(C)-2$ by the adjunction formula
and $ p_a(C)\geq g(\tC)+\varepsilon$.
  The lemma follows immediately from these inequalities
 if $p_a(C)\geq2$.

If $p_a(C)=1$, then $C$ is a rational curve with a single
node or cusp.  We see easily that there are at most two
$\sigma$-fixed points on $C$ and the lemma is true in this
case. \qed\end{pf}

The following lemma is a consequence of the equivariant signature
formula and
  the Lefschetz fixed points formula.

\begin{lem}\label{lef} Let $S$ be a smooth projective surface, and $\sigma\in \Aut(S)$ an automorphism  of order $4$.
  Let $C_1,\ \dots,\ C_m$ ($m\geq 0$) be the
$\sigma$-fixed curves, and $C_{m+1},\ \dots,\ C_n$ ($n\geq m$) be the
$\sigma^2$-fixed curves that are not $\sigma$-fixed. For $a=1, 2, 3$, let
$k_a$ be the number of isolated $\sigma$-fixed points  of weight
$\frac{1}{4}(1, a)$. If $\sigma$ acts trivially in $H^*(S, \mb Q)$
then we have
\begin{align}
\label{sign1} K_S^2-8\chi (\mc O_S)&=\sum_{i=1}^m2C_i^2+k_3-k_1;\\
\label{eq2} 2K_S^2-20\chi (\mc O_S)&=
\sum_{i=1}^m(3C_i^2+K_S\cdot C_i)-k_2-2k_1.
\end{align}
\end{lem}
\begin{pf}
Applying  the equivariant signature formula to $\sigma$
 (cf. \cite{HZ74}, Equation (12), p.177, or \cite{Cai09}, 1.6),
we have
\begin{align*}
4\op{Sign}(S/\langle\sigma\rangle)=\op{Sign}(S)+\sum_{i=1}^m5C_i^2+ \sum_{j=m+1}^n
C_j^2+2k_3-2k_1.
\end{align*}
Note that $\op{Sign}(S)=K_S^2-8\chi (\mc O_S)$ and
$\op{Sign}(S/\langle\sigma\rangle)=\op{Sign}(S)$ since $\sigma$  acts trivially on
$H^2(S, \mathbb{R})$. So
 \begin{align*}
 3(K_S^2-8\chi (\mc O_S))=\sum_{i=1}^m5C_i^2+ \sum_{j=m+1}^n
C_j^2+2k_3-2k_1.
\end{align*}
By Lemma~\ref{lem: sign inv}, we have
\begin{align}\label{equ3.1}
K_S^2=8\chi(\mc O_S)+\sum_{i=1}^n C_i^2.
\end{align}
Combining the above two equalities, we obtain (\ref{sign1}). Applying
the topological Lefschetz formula to $\sigma $,  we have
\begin{align}\label{top}
k_1+k_2+k_3+\sum_{i=1}^me(C_i)=e(S^\sigma)=e(S) =12\chi (\mc O_S)-
K_S^2.
\end{align}
Combining (\ref{sign1}) with (\ref{top}) and noting that
$e(C_i)=-C_i^2-K_S\cdot C_i$, we obtain (\ref{eq2}). \qed\end{pf}

\section{Fibrations of genus 3 with \texorpdfstring{$(\ZZ/2\ZZ)^2$}{}-action}\label{sec: fib g3}
In this section we analyze fiber-preserving actions of $(\ZZ/2\ZZ)^2$ on fibrations of genus $3$. 

The first step is to look at the action on general fibers.

\begin{lem}\label{lem: g3 gen fib}
Let $S$ be a smooth projective surface with $p_g(S)>0$, and $f\colon S\rightarrow B$ a relatively minimal fibration of genus $3$. Let $G=\{\id_S, \sigma_1, \sigma_2, \sigma_3\}\cong (\ZZ/2\ZZ)^2$ be a group of fiber-preserving numerically trivial automorphisms. Then the following holds.
\begin{enumerate}
    \item $f$ is a hyperelliptic fibration.
    \item Exactly one involution in $G$ fixes a horizontal curve.
    \item The induced fibration $h\colon S/G\rightarrow B$ has genus one.
\end{enumerate}
\end{lem}
\begin{proof}
Let $F$ denote a smooth fiber of $f$. Since $p_g(S/G) = p_g(S)>0$ and $F/G$ is a smooth fiber of induced fibration $h\colon S/G\rightarrow B$, we have $g(F/G)\geq 1$. By \cite[Theorem~5.9]{Acc94}, we have
\[
g(F) + 2 g(F/G) = g(F/\langle\sigma_1\rangle) +   g(F/\langle\sigma_2\rangle)  +  g(F/\langle\sigma_3\rangle)
\]
Up to relabelling the $\sigma_i$, we have 
\[
g(F/G) = g(F/\langle\sigma_1\rangle) =1, \,\, g(F/\langle\sigma_2\rangle) =  g(F/\langle\sigma_3\rangle) =2
\]
By the Riemann--Hurwitz formula, $F^{\sigma_1}\neq \emptyset$ and $F^{\sigma_i}= \emptyset$ for $i=2,3$. Also, $F$ is hyperellptic by \cite[Theorem~5.10]{Acc94}.

Therefore, $f$ is hyperelliptic, and $h\colon S/G\rightarrow B$ is a genus-one-fibration. A fiber-preserving automorphism fixes a horizontal curve if and only if it fixes points on a smooth fiber. Thus $\sigma_1$ fixes a horizontal curve while the other two involutions of $G$ do not.
\end{proof}
\begin{nota}\label{nota: g3}
Let $f\colon S\rightarrow B$ and $G=\{\id_S, \sigma_1, \sigma_2, \sigma_3\}\cong(\ZZ/2\ZZ)^2$ be as in Lemma~\ref{lem: g3 gen fib}. Henceforth, we may and will assume that $\sigma_1\in G$ is the involution fixing a horizontal curve, which is denoted by $H$. For $1\leq i\leq 3$, let $\Theta_i$ be the vertical 1-dimensional part of $S^{\sigma_i}$. Then the ramification divisor of the quotient map $\pi\colon S\rightarrow S/G$ is 
\begin{equation}\label{eq: R}
     R:=H+\Theta_1+\Theta_2+\Theta_3.
\end{equation}
In other words, we have $K_S=\pi^* K_{S/G} + R$.

Let $\tau\in\Aut_B(S)$ be the hyperelliptic involution, and $\tG:=\langle G, \tau\rangle$. Since $\tau$ commutes with each $\sigma_i\in G$ and $\tau\notin G$, we have $\tG\cong(\ZZ/2\ZZ)^3$.

Let $\rho\colon\tS\rightarrow S$ be the blow up of the isolated points of $\bigcup_{i=1}^3 S^{\sigma_i}$. Then the action of $\tG$ lift to $\tS$. Denote $\tT:=\tS/G$. The induced map $\lambda\colon \tT\rightarrow S/G$ is the minimal resolution, resolving the singularities of $S/G$, all of which are of type  $A_1$. Let $h\colon S/G\rightarrow B$, $\oH\colon T\rightarrow B$, and $\tilde h\colon\tT\rightarrow B$ the induced genus 1 fibrations. Then $h$ is the relative minimal model of $\tilde h$. The following commutative diagram includes all the morphisms introduced so far:
\begin{equation}\label{diag: g3}
\begin{tikzcd}
\tS\arrow[rr, "\tilde\pi"]\arrow[d, "\rho"]&  & \tT \arrow[ld, "\lambda"']\arrow[rd, "\eta"]\arrow[dd, "\tilde h"] & \\
S \arrow[rrd, bend right, "f"] \arrow[r, "\pi"] & S/G\arrow[rd, "h"]  &  &  T \arrow[ld, "\oH"'] \\
&   & B & 
\end{tikzcd}
\end{equation}
We have the following relations among the canonical divisors:
\begin{equation}\label{eq: can div g3}
  K_S = \pi^*K_{S/G}+ R,\quad K_{\tT} = \lambda^* K_{S/G} = \eta^*K_{ T} + \sE, \quad K_{ T} = \oH^* L  + \sum_i (m_i-1)\oF_i
\end{equation}
where $\sE$ is an exceptional divisor with $\Supp(\sE) = \Exc(\eta)$, $L$ is a line bundle on $B$ of degree $ \chi(\sO_{ T})-2 =\chi(\sO_{S})-2 = p_g(S)-1$, and the $m_i\oF_i$ are the multiple fibers of $\oH$. 

Since $G$ acts trivially on $H^0(S, K_S)$, all the surfaces in the diagram \eqref{diag: g3} share the same geometric genus and holomorphic characteristic, which we will denote  by $p_g$ and $\chi$ respectively for simplicity. In fact, we have 
\begin{equation}\label{eq: can sys S vs S/G}
|K_S|=\pi^*|K_{S/G}| +R,
\end{equation}
and the base loci of the canonical systems satisfy
\[
\Bs|K_S| = \Supp(R)\bigcup \pi^{-1}(\Bs|K_{S/G}|).
\]
Since $\oH\colon T\rightarrow B$ is a genus $1$ fibration, $\Bs|K_T|$ is vertical with respect to $h$. 

Suppose that $p_g(S)>0$, so $|K_S|$ is nonempty. Let $Z$ be the fixed part of $|K_S|$. Then $M:=K_S-Z$ satisfies $H^0(S, M)\cong H^0(S, K_S)$ and $|M|$ has at most isolated base points. We have by \eqref{eq: can sys S vs S/G}
\begin{equation}\label{eq: Z1}
    Z=\pi^*Z_{S/G}+R,
\end{equation}
where $Z_{S/G}$ is the fixed part of $|K_{S/G}|$. Plugging the expression \eqref{eq: R} of $R$ into \eqref{eq: Z1},  we can write 
\begin{equation}\label{eq: Z2}
    Z=H+V,
\end{equation}
where $V=\pi^*Z_{S/G} + \sum_{i=1}^3\Theta_i$ is vertical with respect to $f$. Note also that $M=f^*L$, where $L$ is as in \eqref{eq: can div g3}, so $|M|$ is indeed base point free. In conclusion, we obtain
\begin{equation}\label{eq: K_S}
|K_S| = |M| + Z = f^*|L| + H +V
\end{equation}
\end{nota}

Recall that we have defined $\epsilon_b(S^\sigma)$ for $\sigma\in G$ in Lemma-Definition~\ref{lem: e S^sigma}.
\begin{lem}\label{lem: equality g3}
Resume Notation~\ref{nota: g3}. Suppose that $G$ is numerically trivial. Then the following equalities hold:
\begin{enumerate}
    \item $\sum_{b\in B} \frac{1}{2}\left(\epsilon_b(S^{\sigma_2})+\epsilon_b(S^{\sigma_3}) \right) -\epsilon_b(S^{\sigma_1}) = 4e(B)$.
    \item $\sum_{b\in B} \epsilon_b(S) -\epsilon_b(S^{\sigma_1})  = 8e(B)$.
    \item $(K_S+H)\cdot V - \Theta_{1}^2= 8e(B)$. For each $b\in B$, let $V_b = V\wedge F_b$ and $\Theta_{i,b}=\Theta_i\wedge F_b$. Then $\sum_b(K_S+H)\cdot V_b - \Theta_{1,b}^2= 8e(B)$.
    \item $H^2 +\Theta_1^2 =\Theta_2^2 = \Theta_3^2$.
\end{enumerate}
\end{lem}
\begin{proof}
(i) and (ii). Since the $\sigma_i$'s are numerically trivial, we have $e(S^{\sigma_i}) = e(S)$ for $1\leq i\leq 3$ by the topological Lefschetz fixed-point theorem. Note that $H$ is fixed by $\sigma_1$ and we have
\[
d_{\sigma_1}(f) = \deg(H\rightarrow B)=4;
\]
for $i\in \{2,3\}$, $\sigma_i$ does not have horizontal fixed curves, so $d_{\sigma_i}(f)=0$. Since $G$ is numerically trivial,  by Lemma~\ref{lem: e S^sigma}, we have 
\begin{equation}\label{eq: e(S^sigma_i)}
\begin{split}
&e(S) =e(B)e(F) + \sum_b\epsilon_b(S) = -4e(B) + \sum_b\epsilon_b(S), \\
&e(S)=e(S^{\sigma_1}) = 4e(B)+ \sum_{b\in B} \epsilon_b(S^{\sigma_1}), \\  &e(S)=e(S^{\sigma_i}) = \sum_{b\in B} \epsilon_b(S^{\sigma_i}) \quad \text{ for $i=2,3$.}       
\end{split}
\end{equation}
and (i) and (ii) follow.

\medskip

(iii) By \eqref{eq: K_S} we can write
\begin{equation}\label{eq: K2 vs 8chi 1}
    \begin{split}
          K_S^2 &= K_S\cdot(M+H+V) \\
    & = M\cdot H + K_S \cdot H + K_S\cdot V \\
    & = M\cdot H + (M+H+V)\cdot H + K_S\cdot V \\
    & = 8(\chi(\sO_S) -2) + H^2 + (K_S+H)\cdot V
    \end{split}
\end{equation}
On the other hand, by applying Lemma~\ref{lem: sign inv} to the $\sigma_i$'s, we have
\begin{equation}\label{eq: K2 vs 8chi 1'}
   K_S^2 = 8\chi(\sO_S) + H^2 +\Theta_1^2 = 8\chi(\sO_S) + \Theta_2^2 =  8\chi(\sO_S) + \Theta_3^2
\end{equation}
Comparing \eqref{eq: K2 vs 8chi 1} and \eqref{eq: K2 vs 8chi 1'}, we obtain the first equality of (iii). The second equality follows from the first one, because $V$ and $\Theta_1$ are vertical divisors.

\medskip

(iv) is a consequence of \eqref{eq: K2 vs 8chi 1'}.
\end{proof}

\begin{lem}\label{lem: hyp g3}
\begin{enumerate}
    \item For every irreducible component $C$ of a fiber $f^*b$ with $p_a(C)>0$, $C$ is invariant but not fixed under $\tau$. Moreover, the singularities of $C$ are contained in the fixed locus $C^\tau$.
    \item If $p\in f^*b$ is a singularity of $\Supp(f^*b)$ that is not a node, then $p\in S^\tau$.
\end{enumerate}
\end{lem} 
\begin{proof}
(i) Note that $F_b/\langle\tau\rangle$ is a fiber of the $\PP^1$-fibration $S/\langle \tau\rangle\rightarrow B$, and hence it does not contain any component with positive arithmetic genus. Suppose one the contrary that one of the following holds:
\begin{itemize}
\item $\tau(C)\neq C$;
\item $C=C^\tau$;
\item there is a singularity of $C$ not contained in $C^\tau$;
\item there is a non-nodal singularity $p$ of $\Supp(f^*b)$, and $p\not\in S^\tau$.
\end{itemize}
Then $C\rightarrow \overline C$ is birational, where $\overline C\subset S/\langle\tau\rangle$ is the image of $C$, and hence $p_a(\overline C)\geq p_a(C)>0$, which is a contradiction. 
\end{proof}

\begin{lem}\label{lem: g3 sing comp}
Let $f^*b$ be a singular fiber of $f\colon S\rightarrow B$. If $f^*b$ contains a singular component $C$, then $F=C$, and $C$ is an elliptic curve with two nodes, say $p$ and $p'$. Up to relabelling $\sigma_2$ and $\sigma_3$, there are the following possibilities for the fixed loci $C^{\sigma_i}$, $1\leq i\leq 3$:
\begin{enumerate}[label=(\alph*)]
    \item $G_p=G_{p'}=\langle\sigma_1\rangle$, $H\cap C=C^{\sigma_1}=\{p, p'\}$, $C^{\sigma_2}=C^{\sigma_3}=\emptyset$;
    \item $G_p=G_{p'}=\langle\sigma_2\rangle$, $C^{\sigma_2}=\{p,p'\}$, $C^{\sigma_3}=\emptyset$, $C^{\sigma_1}=C_\sm^{\sigma_1} = H\cap C$ consists of four points.
\end{enumerate}
\end{lem}
\begin{proof}
Since $C$ is singular, we have $p_a(C)>0$ and hence $\tau(C)=C$ by Lemma~\ref{lem: hyp g3}, where $\tau\in \Aut_B(S)$ is the hyperelliptic involution. Therefore, $\tG$ acts on $C$; the action is faithful because $C$ is singular. It follows that $p_a(C)\geq 3$ by Lemma~\ref{lem: Z2n C} (v). Since $p_a(C)\leq p_a(f^*b)=3$, we infer that  $p_a(C)= 3$, and by Lemma~\ref{lem: Z2n C} (v) again, $C$ is a nodal elliptic curve with two nodes, say $\{p,p'\}$, such that $\tG_{p} =\tG_{p'} \cong (\ZZ/2\ZZ)^2$. 

Since $[\tG:G]=2$ and $G_p=\tG_p\cap G$, we have 
\[
|G_p|\in \left\{\frac{1}{2} |\tG_p|, \, |\tG_p|\right\} = \{2,4\}.
\]
But $\tau\in \tG_p$ by Lemma~\ref{lem: g3 sing comp} and $\tau\notin G$, we have $G_p\subsetneq \tG_p$, and hence $|G_p|=2$ and $C_\sing$ consists of one orbit of $G$. Up to relabelling $\sigma_2$ and $\sigma_3$, we have two possibilities: 
\begin{enumerate}
    \item $G_p=G_{p'}=\langle \sigma_1\rangle$;
    \item $G_p =G_{p'} =\langle \sigma_2\rangle$.
\end{enumerate}
In the case where $G_p=G_{p'}=\langle \sigma_1\rangle$, we have $C^{\sigma_i}\subset C_\sm$ for $i\in\{2,3\}$. Since $\sigma_i$ does not fix any horizontal curve, we have $C^{\sigma_i} = C^{\sigma_i} \cap C_\sm =\emptyset$  for $i\in\{2,3\}$. In the case where $G_p =G_{p'} =\langle \sigma_2\rangle$, $\{p,p'\}$ are necessarily isolated in $S^{\sigma_2}$, and $\nu^{-1}\{p,p'\}=\tilde C^{\sigma_2}$ consists of four points, where $\nu\colon\tilde C\rightarrow C$ is the normalization. As before, $C^{\sigma_3} =C^{\sigma_3}\cap C_\sm=\emptyset$. Applying the Riemann--Hurwitz formula to $\tilde C\rightarrow \tilde C/G$, we know that $\tilde C^{\sigma_1}$ consists of four points and it maps bijectively to $\nu(\tilde C^{\sigma_1}) = C^{\sigma_1}\subset C_\sm$. We infer that $C^{\sigma_1} = C\cap H$.
\end{proof}

\begin{lem}\label{g3 snc}
Let $f^*b$ be a singular fiber of $f\colon S\rightarrow B$ whose irreducible components are all smooth. Then $f^*b$ is a simple normal crossing curve.
\end{lem}
\begin{proof}
Suppose one the contrary that there are two components $C_1$ and $C_2$, and a point $p\in C_1\cap C_2$ such that $(C_1\cdot C_2)_p\geq 2$. Then $G$ acts faithfully on $C_1$ and $C_2$ by Lemma~\ref{lem: transverse}. It follows that $G_p$ has order at most two and $[p]_G$ consists of two or four points. We have 
\[
p_a(C_1+C_2)=p_a(C_1)+p_a(C_2)+C_1\cdot C_2-1\geq 3.
\]
It follows that $f^*b = C_1+C_2$, $p_a(C_1)=p_a(C_2)=0$, $C_1\cap C_2=[p]_G$ consists of two points, and $(C_1\cdot C_2)_p=2$. But then, there is an $i\in \{2,3\}$ such that $\sigma_i\in G\setminus G_p$ and the fixed locus $C_1^{\sigma_i}$ is a nonempty set contained in $C_1\setminus C_2\subset(f^*b)_\sm$. Thus $C_1^{\sigma_i}$ contains a horizontal component by \eqref{lem: fix vs sing fib}, which is a contradiction to the assumption made in Notation~\ref{nota: g3}.

The other situation we need to exclude is the case when three components $D_1, D_2, D_3$ pass through one point $p$: In this case, $p\in S^\tau$ by Lemma~\ref{lem: hyp g3}. Since $|\tG_p|\leq 4$ and $\tau\notin G$, we infer that $G_p=\tG_p\cap G$ has order at most 2. Thus the orbit $[p]_G$, which is contained in $D_1\cap D_2\cap D_3$, has cardinality at least 2. It follows that 
\[
p_a(D_1+D_2+D_3)=p_a(D_1+D_2) + p_a(D_3) + (D_1+D_2)\cdot D_3 -1 \geq 1+0+4-1=4,
\]
which is absurd.
\end{proof}

\begin{lem}\label{g3 comp g2}
Let $f^*b$ be a singular fiber of $f\colon S\rightarrow B$ containing a smooth component $C$ of genus 2. Then $f^*b=2C$. Up to relabelling $\sigma_2$ and $\sigma_3$, there are the following possibilities for the fixed loci $C^{\sigma_i}$, $1\leq i\leq 3$:
\begin{enumerate}[label=(\alph*)]
    \item\label{double g2 GC nontriv} $C^{\sigma_2}=C$, $\#C^{\sigma_1}=\#C^{\sigma_3}=2$, $H\cap C=C^{\sigma_1}=C^{\sigma_3}$.
    \item\label{double g2 GC triv1} $\# C^{\sigma_1}=6$, $\#C^{\sigma_2}=\#C^{\sigma_3}=2$.
    \item\label{double g2 GC triv2} $\# C^{\sigma_2}=6$, $\#C^{\sigma_1}=\#C^{\sigma_3}=2$.
\end{enumerate}
\end{lem}
\begin{proof}
By Lemma~\ref{lem: hyp g3}, we have $\tau(C)=C$. 

\begin{claim}
We have $f^*b=2C$.
\end{claim}
\begin{proof}[Proof of the claim.]
If $\tG_C$ is nontrivial, then $2\mid m$ by Lemma~\ref{lem: fix vs sing fib}, where $m$ is the multiplicity of $C$ in $f^*b$. Since $K_S\cdot C = 2g(C)-2-C^2\geq 2$, we infer that $K_S\cdot(mC)\geq 4$. It follows that $C^2=0$, $m=2$, and $f^*b=2C$. If $\tG_C$ is trivial, then $\tG$ acts faithfully on $C$ and hence $4\mid C^2$ by Lemma~\ref{lem: Z2n C} (i). It follows from $4\geq K_S\cdot C = 2g(C)-2-C^2 =2-C^2$ that $C^2=0$ and hence $F=2C$.
\end{proof}

Suppose that $G_C$ is nontrivial. Then $G_C=\langle\sigma_i\rangle$ for some $1\leq i\leq 3$. $G_C$ cannot be $\langle\sigma_1\rangle$, since $H$ is $\sigma_1$-fixed and $H\cap C\neq \emptyset$. Up to relabelling $\sigma_2$ and $\sigma_3$, we may assume $G_C=\langle\sigma_2\rangle$, and $C^{\sigma_3}$ is an isolated subset of $S^{\sigma_3}$. Then $C^{\sigma_1}=C^{\sigma_1\sigma_2} = C^{\sigma_3}$. By Lemma~\ref{lem: Z22}, there is a horizontal $\sigma_1$-fixed curve passing  through each point of $C^{\sigma_1}$, and it is necessarily $H$. Since $2H\cdot C=H\cdot f^*b = 4$, we infer that $\#H\cap C=2$ and $H\cap C=C^{\sigma_1}=C^{\sigma_3}$ consists of two points. This gives the case (a).

Now suppose that $G_C$ is trivial, that is, $G$ acts faithfully on $C$. By the Riemann--Hurwitz formula or \cite[Section~5.9]{Acc94}, we know $g(C/\langle\sigma_i\rangle) = g(C/\langle\sigma_j\rangle)=1$ and $g(C/\langle\sigma_k\rangle)=0$ for some ordering $\{i,j,k\}$ of  $\{1,2,3\}$. Up to relabelling $\sigma_2$ and $\sigma_3$, we are left with the two possibilities (b) and (c).
\end{proof}

\begin{lem}\label{g3 ell comp}
Let $f^*b$ be a singular fiber of $f\colon S\rightarrow B$ containing a smooth component $C$ of genus 1. Then $f^*b=mC+D$ for $m\in\{1,2\}$ and $D$ is nonzero effective divisor consisting of $(-2)$-curves. 
\end{lem}
\begin{proof}
By Lemma~\ref{lem: hyp g3}, $\tG$ acts on $C$ and hence $K_S\cdot C= -C^2$ is even by Lemma~\ref{lem: Z2n C}. We may write $f^*b = mC+D$, where $m\in \ZZ_{>0}$ and $D$ is a nonzero effective divisor not containing $C$. Note that 
\begin{equation}\label{eq: 4 geq 2m}
4=K_S\cdot f^*b = K_S\cdot (mC+D)\geq K_S\cdot mK_S\cdot C \geq 2m,
\end{equation}
and we infer that $m\in \{1,2\}$. 

We claim that $K_S\cdot D=0$, and hence $D$ consists of $(-2)$-curves. In fact, if $m=1$, then $\tG$ acts faithfully on $C$ by Lemma~\ref{lem: fix vs sing fib}, and hence $4\mid K_S\cdot C$. If $m=2$, since $K_S\cdot C$ is even, we also have $4\mid m K_S\cdot C$. In view of \eqref{eq: 4 geq 2m}, we have in both cases $mK_S\cdot C=4$ and $K_S\cdot D=0$. 
\end{proof}

\begin{lem}\label{g3 rat comp}
Let $f^*b$ be a singular fiber of $f\colon S\rightarrow B$, all of whose components are smooth rational curves. Then there is $(-4)$-curve $C$ such that one of the following holds:
\begin{enumerate}
    \item If $\tau(C)=C$ then $C$ is fixed by an involution $\gamma\in \tG\setminus G$, and there is a nonzero effective divisor $D$ consisting of $(-2)$-curves such that $f^*b=2C+D$ is simple normal crossing curve.
    \item If $\tau(C)\neq C$ then $C$ and $\tau(C)$ intersect transversely at four points and $f^*b = C+\tau(C)$. The actions of $G$ on $C$ and $\tau(C)$ are faithful, 
    \[
    C^{\sigma_2}\cup C^{\sigma_3}=C\cap \tau(C) = \tau(C)^{\sigma_2}\cup \tau(C)^{\sigma_3},
    \]
    and $H\cap f^*b = C^{\sigma_1} \cup \tau(C)^{\sigma_1}$ lies on the smooth locus of $f^*b$.
\end{enumerate}
\end{lem}
\begin{proof}
Take any component $C$ of $f^*b$ such that $K_S\cdot C>0$, and let $m$ be its multiplicity. 

\medskip

(i) First suppose that $\tau(C)=C$. Then $\tG\cong (\ZZ/2\ZZ)^3$ acts on $C$, which cannot be faithful, and $K_S\cdot C$ is even (Lemma~\ref{lem: Z2n C}). Therefore, $\tG_C=\langle \gamma\rangle$ for some involution $\gamma\in \tG$, and hence $2\mid m$ by Lemma~\ref{lem: fix vs sing fib}. Since $mK_S\cdot C\leq K_S\cdot f^*b=4$, we infer that $m=2, K_S\cdot C=2$, and $f^*b=2C+D$, where $D$ is a nonzero effective divisor consisting of $(-2)$-curves. Of course, $C$ is a $(-4)$-curve by the adjunction formula, and $D$ is nonzero.

\medskip

(ii) Now suppose that $\tau(C)\neq C$. Then we may write $f^*b=mC+m\tau(C)+D$, where $D$ is a possibly zero effective divisor not containing $C$ or $\tau(C)$. Then
\[
2mK_S\cdot C = K_S\cdot(mC+m\tau(C))\leq K_S\cdot f^*b=4
\]
and hence $mK_S\cdot C\leq 2$. It follows that 
\begin{equation}\label{eq: m KC}
(m, K_S\cdot C) = (2,1) \text{ or } (1,2).
\end{equation}
\begin{claim}\label{claim: g3 GC nontriv}
$G$ acts faithfully on $C$ and $\tau(C)$.
\end{claim}
\begin{proof}[Proof of the claim.]
Suppose on the contrary that $G_C = G_{\tau(C)}$ is nontrivial, say generated by $\sigma_i$ for some $1\leq i\leq 3$.  Then $2\mid m$ by Lemma~\ref{lem: fix vs sing fib}, and hence $(m, K_S\cdot C)=(2,1)$ by \eqref{eq: m KC}. We have then $f^*b=2C+2\tau(C)+D$, where $D$ consists of $(-2)$-curves. Since $S^{\sigma_i}$ is smooth, we have $C\cap \tau(C)=\emptyset$. Let $D'$ be a connected component of $D$, which is a chain of $(-2)$-curves by Lemma~\ref{lem: Z22 chain}. We have $\#C\cap D'\leq 1$, since otherwise the image of $C+D'$ in $S/\langle\tau\rangle$ form a cycle of curves in the corresponding fiber, which is not allowed in the $\PP^1$-fibration $S/\langle\tau\rangle\rightarrow B$. Suppose that $C\cap D'\neq \emptyset$. Then it consists of one point, say $p'$, such that $G_{p'}=G$. There are exactly two points of $C$ that are fixed by the whole group $G$. It follows that $f^*b$ is  a cycle or a chain of smooth rational curve, depending on whether $D$ has exactly two or more than two connected components. Hence the intersection form on the components of $f^*b$ is negative definite, which is absurd.
\end{proof}
By Claim~\ref{claim: g3 GC nontriv} and Lemma~\ref{lem: Z2n C}, $K_S\cdot C$ is even, and hence $(m, K_S\cdot C) = (2,1)$ by \eqref{eq: m KC}. We have then 
\[
f^*b=C+\tau(C)+D
\]
where $D$ is a possibly zero effective divisor consisting of $(-2)$-curves. Suppose on the contrary that $D\neq 0$ and take a component $D_1$ of $D$ such that $D_1\cap C\neq \emptyset$. Since $G$ acts faithfully on $C$, we have $|G_p|\leq 2$ for any $p\in C\cap D$. It follows that the orbit $[p]_G\subset C\cap D_1$ consists of at least 2 points. But then the image of $C+D_1$ in $S/\langle\tau\rangle$ has a positive arithmetic genus, which is impossible. 

Therefore $f^*b= C+\tau(C)$ with $K_S\cdot C=K_S\cdot \tau(C)=2$. By the adjunction formula, $C$ and $\tau(C)$ are $(-4)$-curves. It follows that 
\[
C\cdot \tau(C) = -C^2=4.
\]
By Lemma~\ref{g3 snc}, the intersection of $C$ and $\tau(C)$ is transverse. Since $G$ acts faithfully on $C$ and $\tau(C)$, $G_p$ is cyclic and hence of order at most $2$. In other words, we have 
\[
C^{\sigma_i}\cap C^{\sigma_j}=\emptyset \text{ for $1\leq i < j\leq 3$}.
\]
On the other hand, $\# C^{\sigma_i} = \# \tau(C)^{\sigma_i} = 2$ for each $1\leq i\leq 3$ by the Riemann--Hurwitz formula.  Since $\sigma_i$ does not have horizontal fixed-curves for $i\in \{2,3\}$, their fixed loci on $f^*b$ must be contained in the singular locus of $f^*b$. It is then clear that  
\[
C^{\sigma_2}\cup C^{\sigma_3}=C\cap \tau(C) = \tau(C)^{\sigma_2}\cup \tau(C)^{\sigma_3}
\]
and $H\cap f^*b = C^{\sigma_1} \cup \tau(C)^{\sigma_1}$ lies on the smooth locus of $f^*b$.
\end{proof}

Combining Lemmas~\ref{lem: g3 sing comp}, \ref{g3 comp g2}, \ref{g3 ell comp} and \ref{g3 rat comp}, we obtain a classification of singular fibers when $K_S$ is $f$-ample.
\begin{cor}\label{g3 ample sing fib}
Resume Notation~\ref{nota: g3}. Suppose that $K_S$ is $f$-ample. Then the possible singular fibers $F_b=f^*b$ of $f\colon S\rightarrow B$ together with the action of $G$ on them are listed as follows:
\begin{enumerate}[label=(\arabic*)]
    \item $F_b=C$, where $C$ is an elliptic curve with two nodes, say $p$ and $p'$. Up to relabelling $\sigma_2$ and $\sigma_3$, there are the following possibilities for the action of $G$ on $C$:
\begin{enumerate}
    \item $H\cap C=C^{\sigma_1}=\{p, p'\}$, $C^{\sigma_2}=C^{\sigma_3}=\emptyset$.
    \item $C^{\sigma_2}=\{p,p'\}$, $C^{\sigma_3}=\emptyset$, $C^{\sigma_1} = H\cap C$ consists of four points lying on $C_\sm$.
\end{enumerate}
   \item $F_b=2C$, where $C$ is a smooth curve of genus 2. Up to relabelling $\sigma_2$ and $\sigma_3$, there are the following possibilities of the action of $G$ on $C$:
\begin{enumerate}
    \item\label{double g2 GC nontriv ample} $C^{\sigma_2}=C$, $\#C^{\sigma_1}=\#C^{\sigma_3}=2$, $H\cap C=C^{\sigma_1}=C^{\sigma_3}$.
    \item\label{double g2 GC triv1 ample} $\# C^{\sigma_1}=6$, $\#C^{\sigma_2}=\#C^{\sigma_3}=2$.
    \item\label{double g2 GC triv2 ample} $\# C^{\sigma_2}=6$, $\#C^{\sigma_1}=\#C^{\sigma_3}=2$.
\end{enumerate}
\item $F_b=C+\tau(C)$, where $C$ and $\tau(C)$ are $(-4)$-curves intersecting transversely at four points. We have 
\[
C^{\sigma_2}\cup C^{\sigma_3}=C\cap \tau(C) = \tau(C)^{\sigma_2}\cup \tau(C)^{\sigma_3},
\]
and $H\cap f^*b = C^{\sigma_1} \cup \tau(C)^{\sigma_1}$ consists of 4 points lying on the smooth locus of $f^*b$.
\end{enumerate}
\end{cor}

\begin{thm}\label{thm: g3}
Resume Notation~\ref{nota: g3}. Suppose that $G$ is numerically trivial. If $K_S$ is $f$-ample, then $g(B)\leq 1$, $f\colon S\rightarrow B$ is a quasi-bundle of genus 3, and $S$ is isogenous to a product of curves.
\end{thm}
\begin{proof}
According to the list of Corollary~\ref{g3 ample sing fib}, one can check directly that $\tilde h\colon \tT\rightarrow B$ is $f$-nef and does not have multiple fibers. It follows that $V_b = V\wedge f^*b$ is just $\Theta_b$, the curves in $f^*b$ that are fixed by some $\sigma_i\in G$. One finds that  
\begin{equation}\label{eq: 4 and 0}
(K_S+H)\cdot V_b -\Theta_{1,b}^2=
\begin{cases}
4  & \text{if $f^*b$ is of type (2a) in Corollary~\ref{g3 ample sing fib}}; \\
0 & \text{otherwise}
\end{cases}
\end{equation}

It is also straightforward to compute the numbers $\epsilon_b(S)$ and $\epsilon_b(S^{\sigma_i})$ for $1\leq i\leq 3$ for each type of singular fibers listed in Corollary~\ref{g3 ample sing fib}:
\begin{center}
\begin{tabular}{|c|c|c|c|c|c|c|}
\hline  $f^*b$  & (1a) & (1b) & (2a) & (2b) & (2c) & (3)\\
\hline $\epsilon_b(S)$ &2&2&2&2&2&4\\
\hline $\epsilon_b(S^{\sigma_1})$ &$-2$&0&$-2$&2&$-2$&0\\
\hline $\epsilon_b(S^{\sigma_2})$ &0&2&$-2$&2&6&2\\
\hline $\epsilon_b(S^{\sigma_3})$ &0&0&2&2&2&2\\
\hline
\end{tabular}

\

\tablename  \ 1. $\epsilon_b(S)$ and $\epsilon_b(S^{\sigma_i})$ for a singular fiber in Corollary~\ref{g3 ample
sing fib}
\end{center}
\

Let $n_{1a}, n_{1b}, n_{2a}, n_{2b}, n_{2c}, n_3$ be the number of singular fibers of corresponding type in Corollary~\ref{g3 ample sing fib}. 
Plugging the values of $(K_S+H)\cdot V_b -\Theta_{1,b}^2$, $\epsilon_b(S)$, $\epsilon_b(S^{\sigma_1}) $, $\epsilon_b(S^{\sigma_2})$, and $\epsilon_b(S^{\sigma_3})$ into the formulas of Lemma~\ref{lem: equality g3}, we obtain
\begin{equation}\label{eq: comb sing fib g3}
\begin{split}
    & 2n_{1a} + n_{1b} +  2n_{2a} + 6n_{2c} + 2n_3 =4e(B) \\
    & 4n_{1a} + 2n_{1b} +4n_{2a}+ 4n_{2c} + 4n_3=8e(B)\\
    & 4n_{2a}=8e(B)
\end{split}
\end{equation}
In particular, we have $e(B)\geq 0$, that is, $g(B)\leq 1$.

Suppose that $g(B)=1$, so $e(B)=0$. Therefore, the nonnegative numbers on the left of the equations of \eqref{eq: comb sing fib g3} are all $0$. It follows that $f$ can only have singular fibers of type (2b) of Corollary~\ref{g3 ample sing fib}. Hence, we have $\Theta_i=0$ for $1\leq i\leq 3$, and $K_S^2 = 8\chi(\sO_S)$ by Lemma~\ref{lem: sign inv}. By \cite[Lemma~5]{Ser95}, $f$ is a quasi-bundle and $S$ is isogenous to a product.

Finally, suppose that $g(B)=0$, so $e(B)=2$. Then the equations \eqref{eq: comb sing fib g3} imply that $n_{2a}= 4$ and all the other unknowns appearing in \eqref{eq: comb sing fib g3} are zero. This means that the singular fibers of $f$ are of type (2a) or (2b) in Corollary~\ref{g3 ample sing fib}. In particular, the singular fibers are of the form $2C$, where $C$ is a smooth curve of genus 2, and there are no vertical $\sigma_1$-fixed curves. Also, for $i\in\{2,3\}$, we have $\Theta_i^2 = 0$, where $\Theta_i$ denotes the (vertical) $\sigma_i$-fixed curves. It follows from Lemma~\ref{lem: sign inv}, applied to $\sigma_2$ or $\sigma_3$, that $K_S^2 = 8\chi(\sO_S)$, and hence $f$ is a quasi-bundle and $S$ is isogenous to a product by \cite[Lemma~5]{Ser95}.
\end{proof}

\section{Fibrations of genus 5 with \texorpdfstring{$(\ZZ/2\ZZ)^2$}{}-action}\label{sec: fib g5}
In this section we investigate genus 5 fibration with a fiber-preserving $(\ZZ/2\ZZ)^2$-action. 

We start by proving the existence of an extra fiber-preserving involution, as well as the non-existence of fixed points on the smooth fibers. 
\begin{prop}\label{prop: lifting} 
Let $f\colon S\rightarrow B$ be a relatively minimal fibration  of genus $g=5$, and $G=\{\id_S, \sigma_1, \sigma_2, \sigma_3\}\cong(\ZZ/2\ZZ)^2$ a subgroup of fiber-preserving automorphisms of $S$ such that $g(F/G)=2$ for a smooth fiber $F$ of $f$. Then the following holds.
\begin{enumerate}
    \item The fixed locus $S^{\sigma_i}$ is vertical with respect to $f$ for any $1\leq i\leq 3$.
    \item The hyperelliptic involution $\bar \tau$ of (relative canonical model) of $S/G\to B$ lifts to a fiber-preserving involution $\tau$ of $S$. 
   \item The group $\tG:=\langle G,\tau\rangle$ is isomorphic to $(\ZZ/2\ZZ)^3$.
\end{enumerate} 
 \end{prop}
 \begin{proof}
 (i) Since $g(F)=5$, $g(F/G)=2$ and $|G|=4$, the Riemann--Hurwitz formula, applied to the quotient map $F\rightarrow F/G$, implies that $F^{\sigma_i}=\emptyset$ for each $1\leq i\leq 3$. It follows that $S^{\sigma_i}$ is vertical with respect to $f$.
 
 (ii) Denote by $T_i:=S/\langle\sigma_i\rangle$ ($i=1, 2$) the quotient surfaces, and by $\pi_i: S\to T_i$ and $\pi_i': T_i\to S/G$ quotient maps.
 Then there is an unique morphism $r\colon S\to T_1\times_{S/G}T_2$ such that    the following  diagram is commutative:
\begin{equation}
    \begin{tikzcd}
    S 
    \arrow[drr, bend left, "\pi_2"]
    \arrow[ddr, bend right, "\pi_1"']
    \arrow[dr, dotted, "r" description] &&\\
    & T_1\times_{S/G}T_2 \arrow[r] \arrow[d] & T_2 \arrow[d, "\pi'_2"] \\
    & T_1 \arrow[r, "\pi'_1"] & S/G
    \end{tikzcd}
\end{equation}
Then $r$ is of degree one onto its image.

The smooth fibers $F/\langle \sigma_i\rangle$ of $T_i\rightarrow B$ has genus 3 and is hyperelliptic by \cite[Theorem~5.10]{Acc94}. For $i\in \{1, 2\}$, let $\tau_i\colon T_i\dashrightarrow T_i$ be the hyperelliptic involution over $B$. Then $\tau_i$ is a lift of the hyperelliptic involution $\bar\tau\colon S/G\dashrightarrow S/G$ over $B$. Then $T_1\times_{S/G}T_2$ is $(\tau_1\times\tau_2)$-invariant, and
$\tau_1\times\tau_2$ induces an involution $\tau$ of $S$. It is straightforward to check that $\tau$ is a lift of $\tau_i$ and $\bar \tau$.

(iii) Let $\bar \sigma_2$ be the involution of $T_1$ induced by $\sigma_2$.
Then $\bar \sigma_2\times\op{id}_{T_2}$ induces an involution of $S$, which is nothing but
 $\sigma_1$.
 Since $\bar \sigma_2$ commutates with $\tau_1$, we have  $\sigma_1$ commutates with $\ti \tau$.  Similarly, $\sigma_2$ commutates with $\ti \tau$. Since $G\cong (\ZZ/2\ZZ)^2$ and $\tau\notin G$, we infer that $\langle G,\tau\rangle\cong(\ZZ/2\ZZ)^3$.
 \end{proof}

\begin{nota}\label{nota: g5}
Unless otherwise specified, the following assumptions and notation are used in the remaining part of this section.
 \begin{itemize}[leftmargin=*]
     \item $f\colon S\rightarrow B$ denotes a smooth fibred surface of genus $g=5$ such that $K_S$ is $f$-ample.
     \item $G=\{\op{id}_S, \sigma_1, \sigma_2, \sigma_3 \}\cong(\ZZ/2\ZZ)^2$ is a subgroup of vertical-curve-preserving automorphisms of $S$ such that $g(F/G)=2$ for a smooth fiber $F$ of $f$.
     \item $\tau\in \Aut_B(S)$ denotes a lift of the hyperelliptic involution of (relative canonical model) of $S/G\to B$, whose existence is guranteed by Proposition~\ref{prop: lifting} (ii).
     \item As in Proposition~\ref{prop: lifting} (iii), denote $\tG:=\langle G,\tau\rangle$, which is isomorphic to $(\ZZ/2\ZZ)^3$.
     \item $\Theta_i$ ($1\leq i\leq 3$) denotes the 1-dimensional part of the fixed locus $S^{\sigma_i}$, and $\Theta:=\sum_{i=1}^3 \Theta_i$ denotes the sum, which are vertical by Proposition~\ref{prop: lifting} (iii).
     \item For a vertical divisor $D$ of $S$ and $b\in B$, $D_b:=D\wedge f^*b$ denotes the part of $D$ on $f^*b$.
 \end{itemize}
Note that the assumptions of Notation~\ref{nota: g5} strengthen those of Proposition~\ref{prop: lifting} in that $K_S$ becomes $f$-ample and the action of $G$ on $S$ preserves each vertical curve.
\end{nota}
 
Now we start to classify the singular fibers of $f\colon S\rightarrow B$ under the set-up of Notation~\ref{nota: g5}.

Recall that for $b\in B$, $\Theta_{b}:=\Theta\wedge f^*b$ is the 1-dimensional part of $\bigcup_{i=1}^{3}(f^*b)^{\sigma_i}$.

\begin{lem}\label{lem: conn fixed curve}
Suppose that $\Theta_b\neq 0$. Then $\Theta_b$ is connected and $\Theta_b^2$ is even.
\end{lem}
\begin{proof}
Let $C$ be a connected component of $\Theta_b$. We first show that $C$ is invariant under the action of $\tau$. Suppose on the contrary that $\tau(C)\neq C$. Then $\tau(C)$ is another connected component of $\Theta_b$. We have $p_a((C+\tau(C))/\tG)=p_a(C/G)\geq 1$ since $C/G$ is still a semi-stable curve by Lemma~\ref{lem: Z22}. This is impossible, because $C/G=(C+\tau(C))/\tG$ is a subcurve of a fiber of the $\PP^1$-fibration $S/\tG \rightarrow B$.

\medskip

Note that $S/\tG$ is smooth around the image of $\Theta_b$ and $2C=\bar\pi^*\overline C$ by Lemma~\ref{lem: Z2n quotient}, where $\bar\pi\colon S\rightarrow S/\tG$ is the quotient map and $\overline{C}=\bar \pi(C)$. It follows that
\[
 C^2 = \frac{1}{4}(2C)^2 = \frac{1}{4}(\bar\pi^*\overline C)^2 = 2\overline C^2.
\]
Therefore, $C^2$ is even. By the adjunction formula $K_S\cdot C$ is also even, and hence is at least $2$. Suppose that $C'$ is a connected component of  $\Theta_b$ other than $C$. Then $K_S\cdot C'\geq 2$ for the same reason. Since $f^*b$ is connected, $f^*b-2C-2C'$ is a nonzero effective divisor. It follows that $K_S\cdot f^*b > K_S\cdot(2\Theta_b)\geq K_S\cdot(2C+2C')\geq 8$, which is absurd.

In conclusion, $\Theta_b = C$ is connected, and $\Theta_b^2$ is even.
\end{proof}

\begin{lem}\label{lem: theta neq 0 g5}
Suppose that $\Theta_b\neq 0$. Then one of the following holds.
\begin{enumerate}
 \item $\Theta_b$ is a smooth curve of genus $3$, and $f^*b = 2\Theta_b$. Up to relabelling, the fixed loci are described as follows:
\[
\Theta_b^{\sigma_1} = \Theta_b,\quad \Theta_b^{\sigma_2} = \Theta_b^{\sigma_3} = \emptyset.
\]
\item $\Theta_b$ is a smooth elliptic curve, and there is another smooth elliptic curve $D$ such that $D$ and $\Theta_b$ intersecting transversely at two points, say $p$ and $p'$, and $f^*b = 2(\Theta_b + D)$. Up to relabelling, the fixed loci are described as follows:
 \[
 \Theta_b^{\sigma_1} =  \Theta_b, \quad \Theta_b^{\sigma_i} =\emptyset \text{ for } i=2,3;\quad \# D^{\sigma_1} = \# D^{\sigma_2} = 4, \quad D^{\sigma_3} = \emptyset. 
 \]
 Moreover, $D$ is fixed by some $\gamma\in \tG\setminus G$.
\end{enumerate}
\end{lem}
\begin{proof}
We have $2K_S\cdot \Theta_b\leq K_S\cdot f^*b=8$. Therefore, $K_S\cdot \Theta_b\leq 4$ and hence 
\[
\Theta_b^2=2p_a(\Theta_b)-2 - K_S\cdot \Theta_b \geq -4.
\]

\begin{claim}\label{claim: not -4}
 $\Theta_b^2=-4$ does not occur.
\end{claim}
\begin{proof}[Proof of the claim]
Suppose on the contrary that $\Theta_b^2=-4$. Then 
\[
K_S\cdot \Theta_b=2p_a(\Theta_b)-2-\Theta_b^2\geq 4.
\]
Therefore, $K_S\cdot(2\Theta_b)\geq 8 = K_S\cdot f^*b$. It follows that $f^*b=2\Theta_b$ and $p_a(\Theta_b)=1$, which is impossible.
\end{proof}

Since $\Theta_b^2$ is even by Lemma~\ref{lem: conn fixed curve}, it follows from the Claim~\ref{claim: not -4} that $\Theta_b^2\in\{0,-2\}$. We proceed according to the value of $\Theta_b^2$.

\medskip

(i) if $\Theta_b^2=0$ then $f^*b = m \Theta_b$ for some positive integer $m$ by Zariski's lemma and $K_S\cdot \Theta_b$ is even the adjunction formula. Then $mK_S\cdot \Theta_b = K_S\cdot f^*b=8$ and hence $m\in \{2,4\}$. 
Note that $S/G$ is smooth around the image of $\pi(\Theta_b)$ by Lemma~\ref{lem: Z2n quotient}, where $\pi\colon S\rightarrow S/G$ is the quotient map, and $h^*b = \frac{m}{2} \pi(\Theta_b)$, where $h\colon S/G\rightarrow B$ is the induced genus two fibration. Since a genus two fibration cannot have double fibers by the adjunction formula, it follows that $m=2$. Now that $K_S\cdot \Theta_b = \frac{1}{2}K_S\cdot f^*b=4$, we have $p_a(\Theta_b) = \frac{1}{2}\Theta_b\cdot(K_S+\Theta_b)=3$.

\begin{claim}\label{claim: Theta_b1}
$\Theta_b$ is a smooth curve of genus $3$.
\end{claim}
\begin{proof}[Proof of the claim.]
Suppose on the contrary that $C$ is an irreducible component of $\Theta_b$ with maximal genus. Without loss of generality, suppose that $G_C=\langle\sigma_1\rangle$. Then 
\[
-C^2=C\cdot(\Theta_b-C) = \#C^{\sigma_2} \geq 
\begin{cases}
2 & \text{if $g(C)=2$} \\
4 & \text{if $g(C)=1$}
\end{cases}. 
\]
where the inequality is by the Riemann--Hurwitz formula. In both cases, $K_S\cdot C = 2g(C)-2-C^2\geq 4$ and hence $K_S\cdot f^*b > 2K_S\cdot C\geq 8$, which is absurd.  If all components of $\Theta_b$ are rational, then they form a cycle (cf.~Lemma~\ref{lem: Z22 chain}), and hence $p_a(\Theta_b)=1$, which is absurd.
\end{proof}

Up to relabelling the $\sigma_i\in G$, we may assume that $G_{\Theta_b} = \sigma_1$. Since the $\sigma_i$ do not fix horizontal curves, $G_p\not\cong (\ZZ/2\ZZ)^2$ for any $p\in \Theta_b$ by Lemma~\ref{lem: Z22}. It follows that $\Theta_b^{\sigma_i}=\emptyset$ for $i=2,3$.

\medskip

(ii) if $\Theta_b^2=-2$ then $f^*b-2\Theta_b>0$ and $p_a(\Theta_b) = 1 + \frac{1}{2}\Theta_b\cdot(K_S+\Theta_b) =\frac{1}{2}\Theta_b\cdot K_S\leq 2$. Suppose on the contrary that $p_a(\Theta_b)=2$. Then $K_S\cdot\Theta_b=4$ by the adjunction formula, and hence $K_S\cdot f^*b>2K_S\cdot \Theta_b=8$, which is absurd.

\begin{claim}\label{claim: Theta_b2}
$\Theta_b$ is a smooth elliptic curve.
\end{claim}
\begin{proof}
Since $\Theta_b$ is simple normal crossing by Lemma~\ref{lem: Z22} and $p_a(\Theta_b)\leq 2$, it suffices to show that $\Theta_b$ is irreducible. Suppose on the contrary that $\Theta_b= C_1+C_2$ has two irreducible components. Since $K_S$ is $f$-ample and $K_S\cdot \Theta_b=2$, we have $K_S\cdot C_i=1$ for $i=1,2$ and the $C_i$ are $(-3)$-curves. Since $C_i^2$ is odd, $C_i$ is not preserved by $\tau$ by Lemma~\ref{lem: Z2n C}. Since $\Theta_b$ is invariant under $\tau$, we infer that $C_1$ and $C_2$ are switched by $\tau$. It follows that $G_{C_1} = G_{C_2}$, which is absurd.
\end{proof}
Up to relabeling, we may assume that $G_{\Theta_b}=\langle \sigma_1\rangle$. Then the multiplicity $m$ of $\Theta_b$ is even by Lemma~\ref{lem: fix vs sing fib}. Since $2m = K_S\cdot (m\Theta_b) < K_S\cdot f^*b =8$, we have $m=2$. We have $\Theta^{\sigma_i}=\emptyset$ for $i=2,3$, since otherwise $\Theta_b$ would have two components through $\Theta_b^{\sigma_i}$ by Lemma~\ref{lem: Z22}, which contradicts Claim~\ref{claim: Theta_b2}. Therefore, $G_p=G_{\Theta_b}=\langle \sigma_1\rangle$ for any $p\in \Theta_b$.

Now let $D$ be an irreducible component of $f^*b-2\Theta_b$ such that $D\cap \Theta_b\neq \emptyset$. Since $G$ acts faithfully on $D$, we have $2\mid D^2=\tau(D)^2$ by Lemma~\ref{lem: Z2n C}. It follows by the adjunction formula that $2\mid K_S\cdot D= K_S\cdot \tau(D)$. 

\begin{claim}\label{cla: D2}
$\Theta_b$ and $D$ intersect transversely at two points.
\end{claim}
\begin{proof}[Proof of the claim.]
Since $\Theta_b$ is fixed by $\sigma_1$ and $D$ is invariant under $\sigma_1$, it suffices to prove that $D$ is smooth at any $p\in \Theta_b\cap D$ by Lemma~\ref{lem: transverse}. Suppose on the contrary that $p\in D$ is a singularity. Note that $[p]_G = \{p, \sigma_2(p)\}\subset \Theta_b\cap D$ and $\sigma_2(p)$ is also a singularity of $D$. Therefore, $p_a(D)\geq 2$ and $\Theta_b\cdot D\geq 4$, and hence
\[
p_a(\Theta_b+D)  \geq p_a(\Theta_b) + p_a(D) + 2(\Theta_b\cdot D)_p \geq 1+2+4-1> 5
\]
which is absurd.

Now suppose on the contrary that $\Theta_b\cap D$ contains more than two points. Since $\Theta_b\cap D \subset D^{\sigma_1}$, $\sigma_1$ fixes more than two points on $D$, which implies that $p_a(D)\geq g(\tilde D)\geq 1$, where $D$ is the normalization of $D$. Since $\Theta_b\cap D$ consists of $G$-orbits, we have 
\[
\#\Theta_b\cap D  \geq \# D^{\sigma_1}\geq 4
\]
and hence $p_a(\Theta_b+D)\geq 5$ as above, which is again impossible by Lemma~\ref{lem: subcurve smaller g}.
\end{proof}

\begin{claim}\label{cla: D3}
$D$ is (pointwise) fixed by some $\gamma\in \tG\setminus G$.
\end{claim}
\begin{proof}[Proof of the claim.]
Suppose on the contrary that $D$ is not fixed by any $\gamma\in \tG\setminus G$. 

If $\tau(D) = D$ then $\tG$ acts faithfully on $D$, and $4\mid D^2$ by Lemma~\ref{lem: Z2n C}. It follows that $p_a(D)\geq g(\tilde D)\geq 1$, where $\tilde D$ is the normalization of $D$, and hence $K_S\cdot D=2p_a(D)-2-D^2 \geq 4$. But then $K_S\cdot (2\Theta_b+D) \geq 8$ and hence $f^*b = 2\Theta_b+D$. This leads to a contradiction: $0=\Theta_b\cdot f^*b =\Theta_b\cdot (2\Theta_b+D) = -4+2<0$.

If $\tau(D)\neq D$, then $K_S\cdot (2\Theta_b+D+\tau(D))\geq 8$. It follows that $f^*b = 2\Theta_b+D + \tau(D)$ and $K_S\cdot D=2$. We have
\[
-4 \leq 2p_a(D)-2 -K_S\cdot D  = D^2  = - D\cdot (f^*b- D) = -D\cdot(2\Theta_b+\tau(D))\leq -D\cdot (2\Theta_b) = -4
\]
where the first inequality is by the facts that $p_a(D)\geq 0$ and $K_S\cdot D=2$. It follows that $D\cdot\tau(D)=0$, and hence $D$ and $\tau(D)$ are two disjoint $(-4)$-curves. Now by the Riemann--Hurwitz formula, we have $\sum_{i=1}^3 D^{\sigma_i}>2$. Therefore, there exists some $p\in D\setminus \Theta_b$ which is fixed by some $\sigma_i$. Since $f^*b$ is smooth at $p$, there is a horizontal $\sigma_i$-fixed curve through $p$, which is a contradiction to Proposition~\ref{prop: lifting} (iii).
\end{proof}

By Claim~\ref{cla: D3}, $D$ is smooth, and the multiplicity $m_D$ of $D$ in $f^*b$ is even. Since $K_S\cdot D\geq 2$, we have $K_S\cdot(2\Theta_b + m_D D)\geq 8$. It follows that $m_D=K_S\cdot D=2$, and $f^*b = 2\Theta_b + 2 D$. By Claim~\ref{cla: D2}, we infer that $D^2= -  D\cdot \Theta_b =-2$ and hence $g(D) = \frac{1}{2}(K_S+D)\cdot D +1 =1$. In other words, $D$ is a smooth elliptic curve.
\end{proof}

Let $\rho\colon\tS\rightarrow S$ be the blow-up of the isolated points of $\bigcup_{i=1}^{3} S^{\sigma_i}$. Then the action of $\tG=\langle G, \tau\rangle$ lift to $\tS$, and the induced map $\mu\colon \tS/G\rightarrow S/G$ is the minimal resolution, resolving the singularities of $S/G$, all of which are of type  $A_1$. For simplicity of notation, we denote $\tT:=\tS/G$ and $T=S/G$.

Let $h\colon T\rightarrow B$ and $\tilde h\colon\tT\rightarrow B$ the induced genus 2 fibrations. There is a relative canonical map $\tilde\lambda\colon\tT\dashrightarrow\Proj(\tilde h_*\omega_{\tT}) = \Proj(h_*\omega_T)$ over $B$. The following commutative diagram includes all the morphisms introduced so far:
\begin{equation}\label{eq: diag g5}
   \begin{tikzcd}
\tS\arrow[rr, "\tilde\pi"]\arrow[d, "\rho"']&  &  \tT \arrow[ld, "\mu"']\arrow[dd, "\tilde h"] \arrow[rd, dashed, "\tilde\lambda"] &\\
S \arrow[rrd, bend right, "f"] \arrow[r, "\pi"] & S/G\arrow[rd, "h"]  & &  \Proj(\tilde h_*\omega_{\tT})\arrow[dl, "r"'] & \\
&   & B 
\end{tikzcd} 
\end{equation}
where $\tilde \pi$ and $\pi$ are the quotient maps. Note that $\tilde \lambda$ factors as $\rho\circ \mu$, where $\lambda\colon S/G\dashrightarrow \Proj(\tilde h_*\omega_{\tT}) = \Proj(h_*\omega_T)$ is the relative canonical map of $T$ over $B$. The automorphism $\tau\in \Aut_B(S)$ descends to $B$-automorphisms of $T$ and $\tT$, which we still denote by $\tau$.
\begin{lem}\label{lem: nef g5}
The canonical divisors $K_{S/G}$ and $K_{\tT}$ are relatively nef over $B$.
\end{lem}
\begin{proof}
Since $S/G$ has at most $A_1$-singularities, we have $K_{\tT}=\mu^* K_{S/G}$. Thus it suffices to prove the relative nefness of $K_{S/G}$. Now $\pi^*K_{S/G} = K_S - \sum_{b\in B} \Theta_b$. Since $K_S$ is $f$-ample, $\pi^*K_{S/G}$ is ample on the fibers with $\Theta_b=0$. For $b\in B$ with $\Theta_b\neq 0$, we have two cases by Lemma~\ref{lem: theta neq 0 g5}:
\begin{enumerate}
    \item $f^*b = 2\Theta_b$, where $\Theta_b$ is a smooth curve of genus $3$;
    \item $f^*b = 2\Theta_b + 2D$, where $\Theta_b$ and $D$ are two smooth elliptic curves intersecting transversely at two points.
\end{enumerate}
In the first case, 
\[
\pi^*K_{S/G}\cdot\Theta_b=K_S\cdot \Theta_b =4.
\]
In the second case,  
\[
\pi^*K_{S/G}\cdot\Theta_b=(K_S-\Theta_b)\cdot \Theta_b=2,\quad \pi^*K_{S/G}\cdot D=(K_S-\Theta_b)\cdot D=0.
\]
 In conclusion, $\pi^*K_{S/G}$ is nef over $B$. Since $\pi$ is finite, we infer that $K_{S/G}$ is also nef over $B$.
\end{proof}

The following result describes the base locus $\Bs|K_{\tT} + \tilde h^*\gm|$, where $\gm$ is a sufficiently ample divisor on $B$; it is exactly the indeterminacy locus of the relative canonical map of $\tilde \lambda\colon \tT\dashrightarrow\Proj(h_*\omega_T)$ over $B$.
\begin{lem}\label{lem: bs g2}
Let $X$ be a smooth projective surface and $h\colon X\rightarrow B$ a relatively minimal fibration of genus $2$. Let $h^*b$ be the fiber over a point $b\in B$. Fix a sufficiently ample divisor $\gm$ on $B$, and let $|K_X+h^*\gm| = |M|+Z$ be the decomposition into the mobile and fixed parts. Then the following holds.
\begin{enumerate}[leftmargin=*]
    \item The 1-dimensional part $Z_b:=Z\wedge h^*b$ is nonempty if and only if the hyperelliptic involution $\tau$ has at least two isolated fixed points. If this is the case, then $\Supp(Z_b)$ is the unique minimal chain of $(-2)$-curves which passes through all the isolated $\tau$-fixed points. 
    \item $\Bs|M|_b:=\Bs|M|\,\cap\, \Supp(h^*b)$ is nonempty if and only if $h^*b$ has odd number of isolated $\tau$-fixed points, that is, $h^*b$ is of type $\I_k, \III_k$ or $\V$ in terms of the notation of \cite[Lemma~8]{Hor77}. If this is the case, then $\Bs|M|_b$ consists of exactly one point, which is the image of the curve corresponding to the "$\bullet$" on the left of the dual graphs in \cite[Lemma~8]{Hor77}.
\end{enumerate}
\end{lem}
\begin{proof}
The assertions follow from the classification of the singular fibers of a genus 2 fibration as well as the description of $|K_X+h^*\gm|$ in \cite[Lemmas~8, 10 and 13]{Hor77}.
\end{proof}

\begin{lem}\label{lem: iso bs}
Let $q$ be an isolated point of $\tT^\tau$. Then the following holds.
\begin{enumerate}
    \item $\tilde \pi^{-1}(q)$ is a free $G$-orbit, consisting of 4 points.
    \item There is some $\sigma\in G$ such that $G_p = \langle \tau\sigma\rangle$ for any $p\in \tilde \pi^{-1}(q)$ and $p$ is an isolated point of $\tS^{\tau\sigma}$. As a consequence, $\tilde \pi^{-1}(q)$ consists of isolated points of $\bigcup_{\gamma\in \tG\setminus G} \tS^{\gamma}$.
\end{enumerate}
\end{lem}
\begin{proof}
Since $q\in \tT^\tau$, $\tilde \pi^{-1}(q)$ is a $G$-orbit invariant under $\tau$. Thus $\tilde \pi^{-1}(q)$ is also a $\tG$-orbit. It follows that, for $p\in \tilde \pi^{-1}(q)$, we have 
\[
[\tG_p: G_p] = [\tG:G]=2.
\]
By Lemma~\ref{lem: l-rank}, $\tG_p\neq \tG$, so $|\tG_p|\in\{2,4\}$, and hence $(|\tG_p|, |G_p|) = (4,2)$ or $(2,1)$.

\medskip

(i) Suppose on the contrary that $G_p$ is nontrivial. Then $\tG_p\cong (\ZZ/2\ZZ)^2$, and hence there are some $\gamma\in \tG\setminus G$ and a $\gamma$-fixed curve $\tilde C\subset \tS$ passing through $p$ by Lemma~\ref{lem: Z22}. But then $\tilde\pi(\tilde C)\subset \tT$ is a $\tau$-fixed curve passing through $p$, which is a contradiction to the assumption that  $q$ is an isolated point of $\tT^\tau$.

\medskip

(ii) By (i) we have $\tG_p=\langle \tau\sigma\rangle$ for some $\sigma\in G$. Since $q\in \tT^\tau$ is isolated, so is $p\in \tS^{\tau\sigma}$.
\end{proof}

\begin{lem}
If $b\in B$ is such that $\Theta_b\neq 0$, then the fiber $\tilde h^*b$ is of type (0) in the sense of \cite{Hor77}. Therefore, $|K_{\tT}+\tilde h^*\gm|$ is base-point-free around $\tilde h^*b$,  or equivalently, $|K_{\tilde h^*b}|$ is base-point-free.
\end{lem}
\begin{proof}
By Lemma~\ref{lem: theta neq 0 g5}, every component of $f^*b$ is fixed by some $\gamma\in \tG\setminus\{\id_{\tS}\}$. Therefore, there are no base points of $|K_{\tT}+\tilde h^*\gm|$ lying on $f^*b$ by Lemma~\ref{lem: iso bs}.
\end{proof}

\begin{lem}\label{lem: iso bs fibre}
If $\tilde h^*b$  contain an isolated point of $\tT^\tau$,  then the fiber $f^*b\subset S$ is as one of the following.
\begin{enumerate}[label=(\roman*)]
 \item\label{double g3} $f^*b = 2C$, where $C$ is a smooth curve of genus $3$. Up to relabelling, we have
\[
\#C^{\sigma_1}=4, \quad C^{\sigma_i}=\emptyset \text{ for }i=2,3.
\]

 \item\label{double ell 2 nodes}  $f^*b = 2C$, where $C$ is an elliptic curve with two nodes. Up to relabelling, we have
\[
\#C^{\sigma_1}=4, \quad C^{\sigma_2}= C_\sing,\quad \#C^{\sigma_3}= \emptyset.
\]
 \item\label{2 ell} $f^*b = C + D$ with $C$ and $D$ two smooth elliptic curves intersecting transversely at $4$ points. Moreover, the group $G$ acts freely on $f^*b$.
 
\end{enumerate}
In all cases, $\bigcup_{\gamma\in\tG\setminus G} S^\gamma$ has exactly $4$ isolated points, forming one orbit of $\tG$ and of $G$.
\end{lem}
\begin{proof}
Recall the notation in the diagram~\eqref{eq: diag g5}. Let $\tilde q\in \tilde h^*b$ be an isolated point of $\tT^\tau$. By Lemma~\ref{lem: iso bs}, the inverse image $\tilde \pi^{-1}(\tilde q)\subset \tS$ is a free orbit of $G$, and hence not lies on the exceptional curves of $\rho\colon \tS\rightarrow S$. It follows that $\rho(\pi^{-1}(\tilde q))$ also consists of 4 points, say $p_i$ with $1\leq i\leq 4$. The image $q=\tilde\mu(\tilde q)\in S/G$ is a smooth point of $S/G$, and $\pi^{-1}(q) = \{p_i\}_{i=1}^{4}$. By Lemma~\ref{lem: iso bs}, we may assume that $\tG_{p_i} =\langle\tau\rangle$. Let $C\subset f^*b$ be an irreducible  component passing through $\{p_i\}_{i=1}^{4}$. 

\begin{claim}
 $C$ is invariant under $\tau$.
\end{claim}
\begin{proof}[Proof of the claim]
Since the $p_i$ are isolated points of $\bigcup_{\sigma\in G}S^{\tau\sigma}$, every tangent direction at $p_i\in S$ is preserved by $\tG_{p_i}=\langle\tau\rangle$. Suppose on the contrary that $\tau(C)\neq C$. Then the intersection number of $C$ and $\tau(C)$ at each $p_i$, $1\leq i\leq 4$, is at least 2. It follows that
\[
p_a(C+\tau(C)) = p_a(C) + p_a(\tau(C)) + C\cdot\tau(C) \geq 8
\]
which is absurd.
\end{proof}

 It follows that $\tG\cong (\ZZ/2\ZZ)^3$ acts faithfully on $C$. Since rational curves does not admit a faithful $(\ZZ/2\ZZ)^3$-action by Lemma~\ref{lem: Z2n C}, we have $g(\tilde C)\geq 1$, where $\tilde C$ denotes the normalization of $C$.

\begin{claim}\label{cla: C smooth}
The curve $C$ is smooth at $p_i$ for $1\leq i\leq 4$.
\end{claim}
\begin{proof}[Proof of the claim]
Suppose on the contrary that $C$ is singular at $p_i$, $1\leq i\leq 4$. Then
\[
5\leq 4+g(\tilde C) \leq  p_a(C)\leq p_a(f^*b) = 5.
\]
 Hence $g(\tilde C)=1$, $p_a(C)=5$, $f^*b=C$ and the singularities $p_i\in C$ are (ordinary) nodes or cusps. If $p_i\in C$ are nodes then $\tau$ fixes $\geq 8$ points on $\tilde C$, contradicting the Riemann--Hurwitz formula for the quotient map $\tilde C\rightarrow \tilde C/\langle\tau\rangle$. If $p_i\in C$ is a cusp then it cannot be an isolated $\tau$-fixed point by Lemma~\ref{lem: cusp}, which is also a contradiction.
\end{proof}
\begin{claim}\label{cla: C and D}
 Suppose there is another component $D$ of $f^*b$ through $p_i$'s. Then $C$ and $D$ are smooth elliptic curves intersecting transversely at the $p_i$'s and $f^*b=C+D$. Moreover, $G$ acts freely on $f^*b$. This gives the singular fiber in \ref{2 ell}. In this case, $\{p_i\mid 1\leq i\leq 4\}$ is exactly the set of isolated points of $\bigcup_{\gamma\in \tG\setminus G} S^\gamma$ by Lemma~\ref{lem: fix vs sing fib}.
\end{claim}
\begin{proof}[Proof of the claim]
As in Claim~\ref{cla: C smooth}, the curve $D$ is smooth at $p_i$'s and has a positive geometric genus. By the inequalities
\[
5\leq p_a(C) + p_a(D) + 4 - 1\leq p_a(C+ D)\leq5,
\]
we get $f^*b= C+D$ and it has nodes at $p_i$'s. Furthermore, $C$ and $D$ are (smooth) elliptic curves. Since $f^*b$ is reduced and $\sigma_i\in G$ does not fix horizontal curves for $1\leq i\leq 3$, the involutions $\sigma_i\in G$ do not fix any points on $f^*b\setminus\{p_i\}_{i=1}^{4}$. On the other hand, $G_{p_i} =\{\id_S\}$ for $1\leq i\leq 4$. In conclusion, $G$ acts freely on both $C$ and $D$.
\end{proof}

From now on, we assume that $C$ is the only component of $f^*b$ passing through $p_i$ ($1\leq i\leq 4$). Since the $p_i$ are isolated points in $S^\tau$ and $\tau\in \Aut_B(S)$, the multiplicity $\mult_C F_b = \mult_{p_i} F_b$ is even by Lemma~\ref{lem: fix vs sing fib}.  On the other hand, since $\tG$ acts faithfully on $C$, $4\mid C^2$ by Lemma~\ref{lem: Z2n C} and hence $K_S\cdot C = 2p_a(C)-2-C^2\geq 4$.  We have thus $8\leq K_S\cdot (mC) \leq K_S\cdot f^*b =8$. It follows that $m=2$, $f^*b=2C$, and $p_a(C)=3$.

If $C$ is singular then it has an even number of singularities due to the action of $G$. Since $p_a(C)= 3$ and the geometric genus of $C$ is positive, there are exactly $2$ singular points of $C$,  say $p_1$ and $p_2$, which form an orbit of $G$. We may assume that $G_{p_i}=\langle \sigma_2\rangle$. Thus the curve $C$ must have geometric genus $1$, and the singular points of $C$ are either nodes or cusps. The latter cannot occur by Lemma~\ref{lem: cusp}. So $C$ is an elliptic curve with 2 nodes. Applying the Riemann--Hurwitz formula to the quotient map $\tilde 
C\rightarrow \tC/\tG$ one sees easily that
\[
 \sum_{1\leq i \leq 3}\#(f^*b)^{\sigma_i} =  6.
\]
Up to relabeling $\sigma_1$ and $\sigma_3$, we have
\[
\tC^{\sigma_2} = \pi^{-1}(\{p_1, p_2\}),\quad \#\tC^{\sigma_1} =4, \quad \tC^{\sigma_3}=\emptyset,
\]
and hence the fiber $f^*b$ is as in (ii).

Now we can assume $C$ is smooth of genus $3$. By the Riemann--Hurwitz formula, applied to the quotient map $C\rightarrow C/\tG\cong\PP^1$, we have $\sum_{\gamma\in \tG\setminus\{\id_C\}}\#C^{\gamma}=20$. Note that, the horizontal part of $\sum_{\gamma\in \tG\setminus\{\id_C\}} S^\gamma$ intersect $C$ tranversally at 12 points (cf.~Lemma~\ref{lem: transverse}). It follows that 
\[
 \sum_{1\leq i \leq 3}\#(f^*b)^{\sigma_i} \leq  \sum_{\gamma\in \tG\setminus\{\id_S\}}\#C^{\gamma} - 12 - 4= 20- 12 - 4 =4.
\]
On the other hand, by the Riemann--Hurwitz formula, applied to $C\rightarrow C/G$, we infer that $\sum_{i=1}^3\#(f^*b)^{\sigma_i}\geq 4$ and $g(C/G)=1$. Up to relabelling the $\sigma_i$, we have $\#C^{\sigma_1}=4$, $C^{\sigma_i}=\emptyset$ for $i=2,3$ as in (i). This finishes the proof of Lemma~\ref{lem: iso bs}.
\end{proof}

\begin{cor}\label{cor: 1 fix pt}
For any $b\in B$ the fiber $\tilde h^*b$ of $\tilde h\colon\tT\rightarrow B$ contains at most one isolated points of $\tT^\tau$. 
\end{cor}
\begin{proof}
If $h^*b$ contains an isolated point of $\tT^\tau$ then $f^*b$ is as in Lemma~\ref{lem: iso bs}. Since $\bigcup_{\gamma\in\tG\setminus G} S^\gamma$ has exactly $4$ isolated points, forming one orbit of $\tG$ and of $G$, the fixed locus $(\tilde h^*b)^\tau$ contains exactly one point.
\end{proof}

\begin{cor}\label{cor: bp KT 2}
For sufficiently ample divisor $\gm$ on $B$, $|K_{\tT} + \tilde h^*\gm|$ has no fixed part and has at most one base point on a singular fiber.
\end{cor}
\begin{proof}
    This follows from Lemma~\ref{lem: bs g2} and Corollary~\ref{cor: 1 fix pt}.
\end{proof}




\begin{lem}\label{lem: no bp fiber}
Let $b\in B$ be such that $\Theta_b=0$ and $f^*b$ contains no isolated points of $\bigcup_{\gamma\in\tG\setminus G}S^\gamma$. Then $G$ acts faithfully on each component of $f^*b$, and the fixed loci $(f^*b)^{\sigma_i}$, $1\leq i\leq 3$, are pairwise disjoint. Moreover, $f^*b$ together with the $G$-action is as one of the following:
\begin{enumerate}
    \item $f^*b = C + \tau(C)$, where $C$ and $\tau(C)$ are two smooth rational curves intersecting each other at 6 points transversally, and 
    \[
    \bigcup_{i=1}^{3}(f^*b)^{\sigma_i}=C\cap \tau(C), \quad \#(f^*b)^{\sigma_i}=2,\text{ for $1\leq i\leq 3$}.
    \]
    \item $f^*b = 2C$, where $C$ is a smooth curve of genus $3$, and  the fixed loci of the $G$-action is as follows:
    \[
    \sum_{i=1}^3\#(f^*b)^{\sigma_i} =12.
    \]
    \item\label{g3 2 nodes} $f^*b=C$, where $C$ is a genus 3 curve with 2 nodes, and the fixed loci of the $G$-action is as follows:
\[
  \bigcup_{i=1}^{3} (f^*b)^{\sigma_i} =C_\sing.
\]
\item\label{g1 2 tacnodes} $f^*b=C$, where $C$ is an elliptic curve with 2 tacnodes, and the fixed loci of the $G$-action is as follows: 
\[
  \bigcup_{i=1}^{3}(f^*b)^{\sigma_i} =C_\sing.
\]
\item\label{ell 4 cusps} $f^*b=C$, where $C$ is an elliptic curve with 4 cusps, and $G$ acts freely on $C$.  
\item\label{ell 4 nodes free} $f^*b=C$, where $C$ is an elliptic curve with 4 nodes, and $G$ acts freely on $C$.  
\item\label{ell 4 nodes non free} $f^*b=C$, where $C$ is an elliptic curve with 4 nodes, and the 4 nodes form two $G$-orbits, say $\{p_1, p_2\}$ and $\{p_3, p_4\}$, such that $G_{p_1}\neq G_{p_3}$, and $p_i$ is an isolated fixed points of $G$. We have 
\[
  \bigcup_{i=1}^{3} (f^*b)^{\sigma_i}  =C_\sing.
\]

\end{enumerate}
\end{lem}
\begin{proof}
Note that $G$ preserves each component of $f^*b$ by Remark~\ref{lem: e S^sigma}. Since $\Theta_b=0$, none of the $\sigma_i$ fixes any component of $f^*b$; in other words, $G$ acts faithfully on each component of $f^*b$. By Lemma~\ref{lem: Z22}, the fixed loci $(f^*b)^{\sigma_i}$, $1\leq i\leq 3$ are pairwise disjoint.

Since $f^*b$ contains no isolated points of $\bigcup_{\gamma\in\tG\setminus G}S^\gamma$, the corresponding fiber $\tilde h^*b\subset\tT$ does not contain isolated points of $\tT^\tau$ either. Thus $\tilde h^*b$ is of type (0) in the sense of \cite{Hor77}, and the relative canonical map $\tilde\lambda\colon\tT\dashrightarrow \PP(\tilde h_*\omega_{\tT})$ over $B$ has no indeterminacy over a neighborhood $U\ni b$. 
Since $K_S$ is ample on $f^*b$ and $K_S|_{f^*b} = \pi^*K_{S/G}|_{f^*b}$, we infer that $K_{S/G}$ is ample on $h^*b$. In other words, $S/G\rightarrow B$ is the relative canonical model of $\tT\rightarrow B$ over a neighborhood $U\ni b$, and we have the following commutative diagram:
\[
\begin{tikzcd}
\tT_U\arrow[rd] \arrow[r] & (S/G)_U \arrow[d] \arrow[r] & \PP(\tilde h_*\omega_{\tT})_U \arrow[ld]\\
& U &
\end{tikzcd}
\]
which is obtained from the corresponding part of \eqref{diag: can map} by restricting everything over $U$.
Since $S/G$ has at most $A_1$-singularities, according the structure theory of genus 2 fibration due to Horikawa \cite{Hor77}, $(S/G)_U\rightarrow \PP(\tilde h_*\omega_{\tT})_U$ is a (finite) double cover whose branch locus is a nodal curve, and the Galois group is obviously generated by $\tau\in\Aut_B(S/G)$. Thus we have $(S/\tG)_U=\PP(\tilde h_*\omega_{\tT})_U$, which is a $\PP^1$-bundle over $U$.
\begin{claim}
There is an irreducible component $C$ of $f^*b$ such that $\Supp(f^*b)=C\cup\tau(C)$.
\end{claim}
\begin{proof}
 Since $G$ preserves each component of the fiber $S_b$, the two fibers $S_b$ and $(S/G)_b$ have the same number of irreducible components. Since $(S/\tG)_b$ is irreducible and $\tG=\langle G, \tau\rangle$, $(S/G)_b$ has at most two irreducible components, and if there are indeed two then they are permuted by $\tau$. The claim follows.
 \end{proof}



\begin{claim}
If $C\neq \tau(C)$, then $f^*b$ is as in (i).
\end{claim}
 \begin{proof}
 Note that $G$ acts faithfully on $C$ and $\tau(C)$. If $C\neq \tau(C)$, then the quotient map $S_b\rightarrow (S/\tG)_b = \PP^1$ is \'etale around a general point of $S_b$. It follows that $S_b = C+\tau(C)$ is a reduced, and hence
 \[
 K_S\cdot C=K_S\cdot \tau(C) =\frac{1}{2}K_S\cdot f^*b = 4.
 \]
 We have then
 \begin{equation}\label{eq: pa C < 3}
 p_a(C)=1+\frac{1}{2}(K_S\cdot C+C^2) < 1 + \frac{1}{2}K_S\cdot C =3.
 \end{equation}

 Suppose on the contrary that $C$ has a singularity $p$. Then $G_p$ has order at most $2$ by Lemma~\ref{lem: Z22}, and the orbit $[p]_G$ contains at least two points, say $\{p,q\}$. Therefore, $p_a(C)\geq 2$, and the equality holds by \eqref{eq: pa C < 3}. It follows that $\{p,q\}=C_\sing$ and $G_p=G_q$ has order 2.  If $p\in C\cap \tau(C)$, then $p, q$ are also singularities of $\tau(C)$, and hence
 \[
 p_a(f^*b) = p_a(C+\tau(C)) \geq p_a(C) + p_a(\tau(C))+C\cdot\tau(C)-1\geq 2+2+4-1>5
 \]
 which is absurd. If $p\in C\setminus \tau(C)$, then $\gamma(p)\in \tau(C)\setminus C$ for any $\gamma\in \tG\setminus G$, and hence cannot be $p$. It follows that $\tG_p=G_p$, and $p$ is an isolated point of $\bigcup_{\gamma\in \tG\setminus \{\id_S\}} S^\gamma$. But then $S/\tG$ is singular at the image of $p$, which is a contradiction to an previous observation. 
 
Since $C\neq \tau(C)$, the natural map $C/G\rightarrow (f^*b)/G\cong \PP^1$ is birational, and, due to the smoothness of the target, must be an isomorphism. By the Riemann--Hurwitz formula for $C\rightarrow C/G$, we have
\[
\sum_{i=1}^3 C^{\sigma_i} = 2g(C) +6
\]
On the other hand, since the points of $C^{\sigma_i}$ are all isolated in $S^{\sigma_i}$, we have by Lemma~\ref{lem: fix vs sing fib}
\[
C^{\sigma_i} \subset \Sing(S_b)  = C\cap \tau(C)
\]
Thus
\[
5=p_a(f^*b) = g(C)+g(\tau(C)) + C\cdot\tau(C)-1 \geq 2g(C) + \sum_{i=1}^3 C^{\sigma_i} - 1 = 4g(C) + 5
\]
It follows that 
\[
g(C)=0,\quad  C\cap \tau(C) = \bigcup_{i=1}^{3} C^{\sigma_i},
\]
and $C$ and $\tau(C)$ intersect transversally. By the Riemann--Hurwitz formula, it is clear that $(f^*b)^{\sigma_i}=C^{\sigma_i}$ consists of 2 points for each $1\leq i\leq 3$, and $\sum_{i=1}^3 (f^*b)^{\sigma_i} = 6$.
\end{proof}

\begin{claim}
If $\tG_C\neq \{\id_S\}$, then $f^*b$ is as in (ii).
\end{claim}
\begin{proof}
Since $\Theta_b=0$, we have $G_C=\{\id_S\}$ and hence $\tG_C=\langle \gamma\rangle$ for some $\gamma\in \tG\setminus G$. Therefore, $\tau(C)=C$, and, being fixed  by a non-trivial automorphism, it is necessarily smooth. Around a general point of $C\subset S$ the fiber $f^*b$ is just the pull-back of the fiber of $\FF_e\rightarrow B$ over $b$ which is reduced. Since a general point of $C$ has a stabilizer of order 2, namely $\langle\gamma\rangle$, the multiplicity of $C$ in $f^*b$ is exactly 2. Thus we have $f^*b =2C$. Then $K_S\cdot C=\frac{1}{2}K_S\cdot f^*b =4$, and hence $g(C)=3$ by the adjunction formula. We have $g(C/G)=g(C/\tG) = 0$. Applying the Riemann--Hurwitz formula to the quotient map $C\rightarrow C/G$ we have $\sum_{i=1}^3\#(f^*b)^{\sigma_i} = 2g(C) +6 =12$.
\end{proof}
 
In the remaining part of the proof, we may assume that $C=\tau(C)$ and $\tG_C=\{\id_S\}$. Then $f^*b=mC$ for some $m$, and $C$ admits a faithful action of $\tG\cong (\ZZ/2\ZZ)^3$. It follows that $g(\tilde C)\geq 1$ by Lemma~\ref{lem: Z2n C}, where $\tilde C$ is the normalization of $C$. Around a general point of $C\subset S$ the fiber $f^*b$ is just the pull-back of the fiber of $\PP(\tilde h_*\omega_{\tT})$ over $b$, which is reduced. Since a general point of $C$ has a trivial stabilizer, we have $m=1$, that is, $f^*b=C$. 

Obviously the singular set $C_\sing$ of $C$ is a union of $G$-orbits. Since $\Theta_b=0$, the stabilizer $G_p$ has order at most 2 for any $p\in C$ by Lemma~\ref{lem: Z22}. Thus every $G$-orbit has an even number of points. In particular, $\#C_\sing$ is even. On the other hand, $p_a(C)\leq 5$ implies that $\#C_\sing \leq 5$. It follows that $\#C_\sing =2$ or $4$.

We proceed according to the number $\#C_\sing$.

\begin{claim}
 If $f^*b=C$ and $\#C_\sing = 2$, then $f^*b$ is as in one of the cases (iii) and (iv).
\end{claim}
\begin{proof}[Proof of the claim]
Let $p,\,q$ be the two singular points of $C$. Then $C_\sing$ forms one orbit for the actions of both $G$ and $\tG$, and the stabilizers $\tG_p$ and $G_p$ have order $4$ and $2$ respectively. The two elements $\{\gamma_1, \gamma_2\}:=\tG_p\setminus G_p$ generate $\tG_p$, and $\sigma:=\gamma_1\gamma_2$ generates $G_p$. Since $p$ and $q$ are isolated in $S^\sigma$, the automrophisms $\gamma_1, \gamma_2$ fix curves by Lemma~\ref{lem: Z22}, which are necessarily horizontal with respect to $f$.

We have
\begin{equation}\label{eq: paC C_sing=2}
5 = p_a(C) = g(\tC) + 2\delta_p(C)
\end{equation}
where $\nu\colon\tC\rightarrow C$ is the normalization and $\delta_p(C):=(\nu_*\sO_{\tC}/\sO_C)_p$. It follows that $\delta_p(C)\leq 2$. 

Suppose that $\delta_p(C)=1$. Then $g(\tC)=3$ by \eqref{eq: paC C_sing=2}. Since $\tG_p=(\ZZ/2\ZZ)^2$, $p\in C$ cannot be a cusp by Lemma~\ref{lem: cusp} again. Therefore, the two singularities of $C$ are nodes. Since $C^{\sigma_i}\subset C_\sing$, we have 
\[
\bigcup_{i=1}^{3}(f^*b)^{\sigma_i} = C_\sing.
\]
This gives the case (iii).

Suppose that $\delta_p(C)=2$.  Then $g(\tilde C) = 1$ by \eqref{eq: paC C_sing=2}. Let $S_2\xrightarrow{\rho_2}S_1\xrightarrow{\rho_1} S$ be the two blow-ups resolving the singularity $p\in C$. Then the action of $\tG$ lifts to $S_i$. Let $C_i\subset S_i$ be the strict transform of $C$ and $E_i\subset S_i$ the exceptional divisor of $\rho_i$. Then $E_1$ and $C_1$ intersect at a unique point, say $\{p_1\} = E_1\cap C_1$, which is necessarily fixed by $\tG_p$, $\delta_{p_1}(C_1)=\delta_p(C)-1=1$, and $(C_1\cdot E_1)_{p_1}=2$. By Lemma~\ref{lem: cusp}, $p_1\in C_1$ is not a cusp and thus $p_1\in C_1$ is node. Therefore, $p\in C$ is a tacnode. Since $C^{\sigma_i}$ is isolated in $S^{\sigma_i}$, it lies in $C_\sing$ and we have
\[
\bigcup_{i=1}^{3}(f^*b)^{\sigma_i} = C_\sing.
\]
\end{proof}


\begin{claim}
 If $f^*b=C$ and $\#C_\sing = 4$, then $f^*b$ is as in one of the cases (v) and (vi).
\end{claim}
\begin{proof}[Proof of the claim.]
If $\#C_\sing = 4$ then $g(\tC)=1$, and $\delta_p(C)=1$ for each $p\in C_\sing$. Thus $C_\sing$ consists of nodes or cusps of $C$. Since $C_\sing$ is the union of $\tG$-orbits and $\#C_\sing =4$ is less than $|\tG|=8$, the stabilizer $\tG_p$ is nontrivial for each $p\in C_\sing$. Thus $|\tG_p|\in \{2,4\}$.

Suppose that $p\in C$ is a cusp. By Lemma~\ref{lem: cusp}, $G_p$ is trivial, since there is no $\sigma_i$-fixed curve for $1\leq i\leq 3$. Therefore, 
\[
\#[p]_G=4 = \# C_\sing,
\]
and hence $[p]_G=C_\sing$. Since $\bigcup_{i=1}^{3}C^{\sigma_i}\subset C_\sing$, we infer that 
\[
\bigcup_{i=1}^{3}C^{\sigma_i} =\bigcup_{i=1}^{3} C_\sing^{\sigma_i} =\emptyset.
\]
In other words, $G$ acts freely on $C$. This gives the case \ref{ell 4 cusps}. 

Now we can assume that $C_\sing$ consists solely of nodes of $C$. 

If $|\tG_p|=2$ for some $p\in C_\sing$, then $C_\sing$ consists of one $\tG$-orbit. Note that $G_p$ must be trivial, since otherwise $G_p$ would fix the 8 points $\nu^{-1}(C_\sing)$ on the elliptic curve $\tC$, which is impossible by the Riemann--Hurwitz formula. Since $\sigma_i$ does not fix any curve, it cannot have fixed points on the smooth locus of $C$ either. Therefore, $G$ acts freely on $C$, giving the case \ref{ell 4 nodes free}. 

Finally, suppose that $|\tG_p|=4$ and $C_\sing$ is the union of two $\tG$-orbits, say $\{p_1, p_2\}$ and $\{p_3, p_4\}$. Since $\sigma_i\in G$ does not have horizontal fixed curves for $1\leq i\leq 3$, $G_{p_j}$ has order two for $1\leq j\leq 4$, and $p_j$ is an isolated fixed point of $G_{p_j}$. Then the local branches of $C$ at $p_j$ is preserved by $G_{p_j}$, and hence the two points $\nu^{-1}(p_j)\subset \tC$ is fixed by $G_{p_j}$. By the Riemann--Hurwitz formula, applied to $\tilde C\rightarrow \tilde C/G$, we infer that $G_{p_1}\neq G_{p_3}$ and the involution $\sigma\in G\setminus (G_{p_1}\cup G_{p_3})$ acts freely on $\tC$ and $C$. This gives the case \ref{ell 4 nodes non free}.
\end{proof}
\end{proof}

Combining Lemmas~\ref{lem: theta neq 0 g5}, \ref{lem: iso bs}, and \ref{lem: no bp fiber}, we obtain a classification of the singular fibers.
\begin{cor}\label{g5 ample sing fib}
Let $f\colon S\rightarrow B\cong \PP^1$ be as in Set-up~\ref{nota: g5}. Then any singular fiber $f^*b$, $b\in B$ is as one of the following:
\begin{enumerate}[label=(\arabic*)]
  \item  \label{item: g=3 2 nodes}
 $f^*b=C$, where $C$ is a genus 3 curve with two nodes, say $p$ and $p'$. We have 
 \[
 C^{\sigma_1}=\{p, p'\},\quad C^{\sigma_i}=\emptyset, \quad i\in\{2,3\}.
 \]
 \item\label{item: elliptic 4 nodes} $f^*b=C$, where $C$ is an elliptic curve with four nodes, denoted by $p_i$, $1\leq i\leq 4$.
    Up to relabelling of the involutions $\sigma_i\in G$, there are two possible actions of $G$:
    \begin{enumerate}
        \item $G$ acts freely on $C$.
        \item  $C^{\sigma_1}=\{p_1,p_2\}$, $C^{\sigma_2}=\{p_3, p_4\}$, and $C^{\sigma_3}=\emptyset$.
    \end{enumerate}
     \item \label{item: elliptic 4 cusps} 
 $f^*b=C$, where $C$ is an elliptic curve with four cusps, denoted by $p_i$, $1\leq i\leq 4$. $G$ acts freely on $C$. 
\item\label{item: g=1 2 tacnodes} $f^*b=C$, where $C$ is an elliptic curve with two tacnodes, say $p$ and $p'$. We have
\[
\bigcup_{i=1}^{3}(f^*b)^{\sigma_i} = C^{\sigma_1} = \{p,p'\}.
\]
\item\label{item: double g=3} $f^*b=2C$, where $C$ is a smooth curve of genus $3$. Up to relabelling the $\sigma_i\in G$, there are the following possible actions of $G$ on $C$:
    \begin{enumerate}
        \item $C^{\sigma_1}=C$, $C^{\sigma_2}=C^{\sigma_3}=\emptyset$.
     \item $\# C^{\sigma_1} =4$, and $C^{\sigma_2} = C^{\sigma_3}=\emptyset$.
   \item  $G$ acts faithfully on $C$, $g(C/G)=0$, so $\sum_{i=1}^3\# C^{\sigma_i}=12$;
    \end{enumerate}
      \item\label{item: double elliptic 2 nodes}
    $f^*b=2C$, where $C$ is an elliptic curve with two nodes, say $p$ and $p'$. $C^{\sigma_1}=\{p,p'\}$, $C^{\sigma_2}$ consists of $4$ points lying on the smooth locus of $C$, and $C^{\sigma_3}=\emptyset$. 
    \item\label{item: 2 elliptic at 4 pts} $f^*b=C+D$, where $C$ and $D$ are smooth elliptic curves intersecting each other transversely at 4 points. $G$ acts freely on $C$ and $D$. 
        \item\label{item: 2 rational at 6 pts} $f^*b=C+D$, where $C$ and $D$ are two smooth rational curves intersecting transversely at six points. We have 
        \[
        \#C^{\sigma_i} =\# D^{\sigma_i}=2 \text{ for $1\leq i\leq 3$,}\quad C\cap D = \bigcup_{i=1}^{3}C^{\sigma_i} = \bigcup_{i=1}^{3}D^{\sigma_i}.
        \]
        \item\label{item: double 2 elliptic at 2 pts} $f^*b=2(C+D)$, where $C$ and $D$ are two smooth elliptic curves intersecting transversely at two points $p$ and $p'$. Up to relabelling the involutions $\sigma_i\in G$ and the components of $f^*b$, the action of $G$ on $f^*b$ is described as follows: 
    \[
    C^{\sigma_1}=C,\quad C^{\sigma_i}=\emptyset \text{ for $i=2,3$},\quad \#D^{\sigma_1}=\#D^{\sigma_2}=4,\quad \#D^{\sigma_3}=0.
    \]
   \end{enumerate}
\end{cor}

\begin{lem}\label{lem: equality g5}
Let $f\colon S\rightarrow B$ and $G$ be as in Notation~\ref{nota: g5}. Assume that $G$ is numerically trivial. Then we have the following equalities:
\begin{enumerate}
    \item $\sum_{b\in B} \epsilon_b(S) -\epsilon_b(S^{\sigma_i})  = 8e(B)$ for each $1\leq i\leq 3$.
   \item  $\sum_{b\in B} \epsilon_b(S^{\sigma}) - \frac{1}{3}\left(\epsilon_b(S^{\sigma_1})+\epsilon_b(S^{\sigma_2})+\epsilon_b(S^{\sigma_3}) \right)  =  8e(B)$.
\end{enumerate}
\end{lem}
\begin{proof}
The proof of (i) is similar to that of Lemma~\ref{lem: equality g3}, and hence is left to the reader. (ii) is the consequence of (i).
\end{proof}

\begin{thm}\label{thm: g5}
Let $f\colon S\rightarrow B$ and $G\subset \Aut_B(S)$ as in Notation~\ref{nota: g5}. Assume that $G$ is numerically trivial. Then $g(B)\leq 1$, and $f$ is a quasi-bundle and $S$ is isogenous to a product of curves of unmixed type.
\end{thm}
\begin{proof}
Let $n_{1}, n_{2a}, n_{2b},\dots, n_9$ be the number of singular fibers as listed in (1), (2a), (2b), \dots, (9) of Lemma~\ref{g5 ample sing fib}. 

For singular fibers of type (5b), (6), and (7), the corresponding singular fiber of $\tilde h\colon \tT\rightarrow B$ is of type $\III_1$, $\III_1$, and $\I_1$ respectively in the sense of Horikawa \cite{Hor77}. By \cite[Theorem~3]{Hor77}, we have 
\[
K_{S/G}^2=K_{\tT}^2 = 2\chi(\sO_{S/G})-3e(B)+n_{5b}+n_6+n_7.
\]
As before, denote $M:=\pi^*K_{S/G}$. Then $K_S=M+\Theta$, and hence
\begin{equation}\label{eq: Horikawa}
\begin{split}
    K_S^2 & = (M+\Theta)^2\\
    & = M^2 + M\cdot\Theta + K_S\cdot\Theta  \\
    & = 4K_{S/G}^2 + (2K_S-\Theta)\cdot\Theta\\
    & = 4(2\chi(\sO_{S/G})-3e(B)+n_{5b}+n_6+n_7) + 8n_{5a} + 6n_9 \\
    & = 8\chi(\sO_S)-12e(B) + 4n_{5b}+4n_6+4n_7+ 8n_{5a} + 6n_9 
    \end{split}
\end{equation}
On the other hand, by Lemma~\ref{lem: sign inv}, we have $K_S^2 = 8\chi(\sO_S) + \Theta_i^2$ for $1\leq i\leq 3$. It follows that 
\begin{equation}\label{eq: sign inv g5}
    K_S^2 = 8\chi(\sO_S) + \frac{1}{3}\sum_{i=1}^3\Theta_i^2 = 8\chi(\sO_S)  -\frac{2}{3}n_9
\end{equation}
Combining \eqref{eq: Horikawa} and \eqref{eq: sign inv g5}, we obtain
\begin{equation}\label{eq: hol equ}
     6n_{5a} + 3n_{5b}+3n_6+3n_7+ 5n_9  = 9e(B)
\end{equation}
For all the other type of singular fibers $f^*b$, we have 
\[
s_b+ (K_S+M)\cdot Z_b - \frac{1}{3}\sum_{i=1}^3 \Theta_{i,b}^2 =0.
\]

It is also straightforward to compute the numbers $\epsilon_b(S)$, $\epsilon_b(S^{\sigma_i})$ for $1\leq i\leq 3$ and the type of $\tilde h^*b$ for each type of singular fibers. We list them in the following table,
 where  $e_1, e_2, e_3$ are nonnegative
integers with $e_1+ e_2+ e_3=12$,   0 in the last row, $\mathrm{I}_1$, $\mathrm{III}_1$ are type of fibers of genus $2$ in the sense of \cite{Hor77}.

\begin{center}
\begin{tabular}{|l|l|l|l|l|l|l|l|l|l|l|l|l|}
\hline  $f^*b$  & (1) & (2a) & (2b) & (3) & (4) & (5a) & (5b) & (5c) & (6) & (7) & (8) & (9)\\
\hline $\tilde{h}^*b$&0&0&0&0&0&0&$\mathrm{III}_1$&0&$\mathrm{III}_1$&$\mathrm{I}_1$&0&0\\
\hline $\epsilon_b(S)$ &2&4&4&8&6&4&4&4&6&4&6&6\\
\hline $\epsilon_b(S^{\sigma_1})$ &2&0&2&0&2&$-4$&4&$e_1$&2&0&2&2\\
\hline $\epsilon_b(S^{\sigma_2})$ &0&0&2&0&0&0&0&$e_2$&4&0&2&4\\
\hline $\epsilon_b(S^{\sigma_3})$ &0&0&0&0&0&0&0&$e_3$&0&0&2&0\\
\hline
\end{tabular}
\

\tablename  \ 2. $\epsilon_b(S)$ and $\epsilon_b(S^{\sigma_i})$ for a singular fiber in Corollary~\ref{g5 ample sing fib}
\end{center}
\

By Lemmas~\ref{lem: equality g5} we have 
\begin{equation}\label{eq: top equ} 
\frac{4}{3} n_{1} + 4 n_{2a} + \frac{8}{3}n_{2b} + 8 n_3 + \frac{16}{3} n_4 + \frac{16}{3} n_{5a} + \frac{8}{3}n_{5b}+ 4 n_{6} + 4 n_{7} + 4n_8 + 4n_9=8e(B)
\end{equation}

It follows from \eqref{eq: hol equ} and \eqref{eq: top equ} that $g(B)\leq 1$. If $g(B)=1$, then all numbers $n_1, n_{2a}, ...$ appearing on the left side of \eqref{eq: hol equ} and \eqref{eq: top equ} are zero. Thus the singular fibers of $f$ can only be of type (5c). It follows that, as in the proof of Theorem~\ref{thm: g3}, $f$ is quasi-bundle and $S$ is isogenous to a product.

In the remaining part of the proof, we assume that $g(B)=0$, so $e(B)=2$. One checks directly that, in this case, the only nonnegative integer solutions to \eqref{eq: hol equ} and \eqref{eq: top equ} satisfy one of the following:
\begin{enumerate}
    \item $2n_{5a}+n_{5b}=6$, $n_{5c}$ arbitrary, and all the other numbers are zero; 
    \item $n_9=3, n_{6}=1$, $n_{5c}$ arbitrary, and all the other numbers are zero;
    \item $n_9=3, n_{7}=1$, $n_{5c}$ arbitrary, and all the other numbers are zero.
\end{enumerate}
For the case (i) the singular fibers are all of the form $2C$, where $C$ is a smooth curve of genus 2, and we have $\Theta_i^2=0$ for each $1\leq i\leq 3$. It follows that $K_S^2=8\chi(\sO_S)$ by Lemma~\ref{lem: sign inv}. In turn, $f$ is a quasi-bundle and $S$ is isogenous to a product of unmixed type by \cite{Ser95}.

To complete the proof, we need to exclude the remaining cases (ii) and (iii).

\begin{claim}\label{cla: n6=0}
The case (ii) does not occur.
\end{claim}
\begin{proof}[Proof of the claim.]
Suppose on the contrary that case (ii) holds. Then there are $n_{5c}$ singular fibers of type (5c), $3$ fibers of type (9), and $1$ fiber of type (6). By Lemma~\ref{lem: sign inv}, applied to the involutions $\sigma_i$, we have $K_S^2= 8\chi(\sO_S)-2$.  By the Noether formula, we obtain
\begin{equation}\label{eq: g5 e 1}
    e(S) = 12\chi(\sO_S) - K_S^2 = 4\chi(\sO_S) +2.
\end{equation}
On the other hand, we have by Lemma~\ref{lem: e S^sigma} 
\begin{equation}\label{eq: g5 e 2}
   e(S) = -16 + 4n_{5c} + 6n_6 + 6n_9 =  4n_{5c} + 24
\end{equation}
We obtain a contradiction: The right hand side of \eqref{eq: g5 e 1} is not divided by $4$ while that of \eqref{eq: g5 e 2} is.
\end{proof}

\begin{claim}\label{cla: n7=0}
The case (iii) does not occur.
\end{claim}
\begin{proof}[Proof of the claim.]
Suppose on the contrary that case (iii) holds. Then there are $1$ fiber $f^*b_1$ of type (7), $3$ fibers $f^*b_j$ ($2\leq j\leq 4$) of type (9), and $n:=n_{5c}$ singular fibers $F_j:=f^*b_j$ ($5\leq j\leq n+4$) of type (5c).

Look at the fiber $f^*b_1= C+D$, where $C$ and $D$ are two smooth elliptic curves intersecting transversely at four points $p_1, p_2, p_3, p_4$, and $G$ acts freely on $C$ and $D$. The stabilizer $\tG_{p_i}$ are the same for $1\leq i\leq 3$, generated by some $\gamma_0\in \tG\setminus G$, and $\{p_i\}$ are isolated fixed points of $\gamma_0$. Let us denote $\gamma_i:=\gamma_0\circ\sigma_i$ for $1\leq i\leq 3$. Since $\gamma_0$ fixes the points $p_j\in C\cap D$ for $1\leq j \leq 4$ and $\sigma_i$ acts freely on $C$ and $D$ for $1\leq i\leq 3$, we infer that, by applying the Riemann--Hurwitz formula to $C\rightarrow C/\langle \gamma_0, \sigma_i\rangle$ (resp.~$D\rightarrow D/\langle \gamma_0, \sigma_i\rangle$), the fixed locus $C^{\gamma_i}$ (resp.~$D^{\gamma_i}$) consists of $4$ points lying on $C\setminus D$ (resp.~$D\setminus C$). Also, one has $(f^*b_1)^{\gamma_0} = C\cap D$, so $\gamma_0$ does not fix any horizontal curve. For $1\leq i\leq 3$, since $S_b^{\gamma_i}$ lies on the smooth locus of $S_b$, its elements are not isolated in $S^{\gamma_i}$ by Lemma~\ref{lem: fix vs sing fib} and hence there is a curve $H_i\subset S^{\gamma_i}$ such that $S_b^{\gamma_i} = S_b\cap H_i$.

Looking at a general fiber $F$ of $f$, we have
\[
\#F^{\gamma_0} = 0, \quad \#F^{\gamma_i} = H_i\cdot F = H_i\cdot S_b = 8\, \text{ for $1\leq i\leq 3$}
\]
and by Riemann--Hurwitz formula,
\[
g(F/\langle\gamma_0\rangle)=3,\quad g(F/\langle\gamma_i\rangle)=1\,\text{ for $1\leq i\leq 3$.}
\]

Let $\check\rho\colon \cS \rightarrow S$ be the blow-up of singularities $\{p_i\}_{i=1}^{4}$ of $f^*b_1$. Then the action of $\tG$ lifts to $\cS$. The induced fibration $\cS/G\rightarrow B$ has genus 2. Let $  P:=\PP(h_*\omega_{S/G})\rightarrow B$ be the projectivized relative cannoincal bundle. Since $B=\PP^1$, the $\PP^1$-bundle $P$ over $B$ is isomorphic to a Hirzebruch surface $\FF_e$ for some $e\geq 0$. Then we have morphisms $\cS/G \rightarrow \cS/\tG \rightarrow \FF_e$, where the first morphisms is the quotient map by the hyperelliptic involution of $\check h$ and the composition $\cS/G\rightarrow \FF_e$ is the relative canonical map (\cite{Hor77}):
\[
\begin{tikzcd}
    \cS \arrow[r]\arrow[d, "\check\rho"] & \cS/G \arrow[r] \arrow[d]& \cS/\tG\arrow[d]\\
    S\arrow[r]\arrow[rd, "f"'] & S/G\arrow[r, dashed] \arrow[d, "h"] & P=\PP(h_*\omega_{S/G})=\FF_e\arrow[ld]\\
    & B=\PP^1 &
\end{tikzcd}
\]

For $1\leq i\leq 3$, we take the horizontal part $\cH_i$ of $ \cS^{\gamma_i}=\check\rho^{-1}\left(S^{\gamma_i}\right).$
Let $\oH_i$ be the image of $\cH_i$ in $P$. Then $\oH_{i} \in |2\Delta + u_{i} P_b|$ for $1\leq i\leq 3$,  where $\Delta$ (resp.~$P_b$) is a section (resp.~a fiber) of $P\rightarrow B$, with $\Delta^2=-e$, and the $u_i$ are some non-negative integers. Since the $\oH_i$ $(1\leq i\leq 3)$ are reduced and have no common irreducible components, at most one of them, say $\oH_1$, contains $\Delta$ and $u_{1}\geq e$, and we have $u_{i}\geq 2e$ for $i=2,3$.

It is clear that the induced morphisms $\oH_i\rightarrow B$ is finite of degree 2 for each $1\leq i\leq 3$. By the description of the fibers $S_{b_i}$ for $1\leq i\leq 4$, we have 
\begin{itemize}[leftmargin=*]
\item the singularities of the union curve $\oH:=\bigcup_{i=1}^{3}\oH_i$ consists of two ordinary triple points over $b_1$ and one ordinary triple points over $b_i$ for each $2\leq i\leq 4$;
\item $P_b\cup\oH$ still has ordinary singularities for each fiber $P_b$ of $P\rightarrow B$;
\item $\oH_i$ is smooth over $b_1$ for each $1\leq i\leq 3$.
\end{itemize}
It follows that $\tH\rightarrow B$ is \'etale, where $\tH\rightarrow \oH$ is the normalization. Since $B=\PP^1$, we infer that $\tH$ is the disconnected union of 6 smooth rational curves, and hence $\oH$ is union of 6 sections of $P\rightarrow B$. Each $\oH_i$ is the union of two sections of $P\rightarrow B$, say $\oH_{i1}$ and $\oH_{i2}$. 
For $1\leq i\leq 3$, $1\leq j\leq 2$, there are nonnegative integers $u_{ij}$ such that
\begin{equation}\label{eq: nij}
    u_{i1}+u_{i2}=u_i,\quad \oH_{ij}\in|\Delta+u_{ij}P_b|,
\end{equation}
and $u_{ij}\geq e$ unless possibly $e>0$ and there is exactly one pair of indices $(i,j)$ such that $u_{ij}=0$. 
For $0\leq i,\,i'\leq 3$ and $1\leq j,\,j'\leq 2$, we have thus
\begin{equation}\label{eq: Hij intersect}
 \oH_{ij}\cdot \oH_{i'j'} = (\Delta+u_{ij} P_b)\cdot  (\Delta+u_{i'j'} P_b)= u_{ij}+u_{i'j'}-e,
\end{equation}
which is at least $e$, unless $e>0$ and one of $u_{ij}, \, u_{i'j'}$ is $0$. Therefore, there are at least 5 out of the 6 sections $\{\oH_{ij}\}_{ij}$, which pairwise intersect at $\geq e$ points. Since $\oH=\bigcup_{i=1}^{3} \oH_i$ possesses $5$ ordinary triple points, we infer by \eqref{eq: Hij intersect} that 
\[
5\sum_{(i,j)} u_{ij} - 15e=\sum_{(i,j)\neq(i',j')}\left(u_{ij}+u_{i'j'}-e\right) =\sum_{(i,j)\neq(i',j')}\oH_{ij}\cdot \oH_{i'j'} = 5\cdot C_3^2 = 15.
\]
It follows that
\begin{equation}\label{eq: sum nij}
\sum_{(i,j)} u_{ij} = 3(e+1).    
\end{equation}
Since $u_{ij}\geq e$ except for possibly one $(i,j)$, we infer that 
\begin{equation}\label{eq: sum nij geq 5e}
\sum_{(i,j)} u_{ij} \geq 5 e    
\end{equation}
Combining \eqref{eq: sum nij} and \eqref{eq: sum nij geq 5e}, we obtain $e\leq 1$.

Suppose that $e=1$. Then $\sum_{(i,j)}u_{ij}=6$ by \eqref{eq: sum nij}, and at most one of $u_{ij}$ can be $0$. Thus, up to reordering, 
\[
(u_{ij})=(0,1,1,1,1,2), \text{  or } (1,1,1,1,1,1).
\]
In the case $(u_{ij})=(0,1,1,1,1,2)$, since the section $\oH_{i_1j_1}$ with $u_{i_1j_1}=0$ only intersects $\oH_{i_2j_2}$ with $u_{i_2j_2}=2$, so does not pass through any triple points of $\oH$, contradicting the fact that each $\oH_{ij}$ passes through one of the triple point of $\oH$ over $b_1$. In the case $(u_{ij})=(1,1,1,1,1,1)$, none of $\oH_{ij}$ intersects the negative section $\Delta$ of $P\rightarrow B$, and hence each of them corresponds to a line on $\PP^2$, which is obtained by contracting $\Delta\subset P$. But, as one can visibly verify, there are no configration of $6$ lines with only ordinary triple points as singularities.

Now suppose that $e=0$. Then $\sum_{(i,j)}u_{ij}=3$ by \eqref{eq: sum nij}, and hence there are at least three of the $u_{ij}$ is $0$. It follows that, for some $1\leq i\leq 3$, both $\oH_{i1}$ and $\oH_{i2}$ has self-intersection $0$; moreover, we may take a pair of indices $(i',j')$ such that $i\neq i'$ and $\oH_{i'j'}^2=0$. Then $\oH_{i'j'}$ does not intersect $\oH_{i} = \oH_{i1}\cup \oH_{i2}$, which contradicts our knowledge that $\oH_{i'j'}$ intersect $\oH_i$ at a point over $b_1$.
\end{proof}
The proof of Theorem~\ref{thm: g5} is now complete.
\end{proof}
\section{Fibrations of genus 3 with \texorpdfstring{$(\ZZ/4\ZZ)$}{}-action}\label{sec: fib g3c4}
In this section we analyze fiber-preserving actions of $\ZZ/4\ZZ$ on
fibrations of genus $3$, with the following result.
\begin{thm}\label{thm: g3c4}
Let $S$ be a smooth projective regular surface, and $f\colon
S\rightarrow \mathbb{P}^1$ a  fibration of genus $3$ such that $K_S$
is $f$-ample (i.e., $S$ has no $(-1)$-curves or $(-2)$-curves contained
in a fiber of $f$). If there exists a numerically trivial
automorphism $\sigma$ of order $4$ of $S$ with $f\circ\sigma=f$,
then $\chi(\mc O_S)\leq2$.
\end{thm}
 We begin by fixing notations.
 \begin{nota}\label{5.2}Let $f\colon
S\rightarrow B:=\mathbb{P}^1$ be as in Theorem~\ref{thm: g3c4} with $p_g(S)>0$. Let $F$ be a general fiber of $f$.

Let $M$ and  $Z$ be the moving part and  the fixed part of the
linear system $|K_S|$, respectively. (We have $M=0$ if $p_g(S)=1$.)
Let  $H$ and $V$ the horizontal part and the vertical part of $Z$
with respect to $f$, respectively.

 Let  $\sigma$ be  a  numerically trivial automorphism
of order $4$ of $S$ with $f\circ\sigma=f$, and $\alpha=\sigma^2$.
For $a=1, 2, 3$, let
$k_a$ be the number of isolated $\sigma$-fixed points  of weight
$\frac{1}{4}(1, a)$. 
\end{nota}

\begin{lem}\label{H-3} Let the notation be as above.  Then the following holds.
 \begin{enumerate}[leftmargin=*]
 \item Either $H$ is irreducible with
$H\cdot F=4$ or  $H$  has two irreducible components
 $\Gamma_1, \Gamma_2$ with
$\Gamma_1 \cdot F=\Gamma_2 \cdot F=2$;
\item $M\equiv (\chi(\mo_S)-2)F$;
\item
 $H$ is $\alpha$-fixed and  there are no  $\alpha$-fixed
curves apart from $H$;
\item There are no
$\sigma$-fixed curves.
\end{enumerate}
\end{lem}
\begin{pf} (i)
Note that $g(F/\langle\sigma\rangle)=1$ since $p_g(S/\langle\sigma\rangle)=p_g(S)>0$.  Applying the Hurwitz formula to the quotient map $\pi\rest F \colon F\to F/\langle\sigma\rangle$
we have that $\pi\rest F$ has exactly two branch points and the
ramification indices are both two.  Clearly
  the stabilizer  of $\left<\sigma\right>$ at every ramification point of $\pi\rest F$ is
  $\left<\alpha\right>$.

  Since $F$ is a general fiber, there is exactly an
   $\alpha$-fixed (possibly reducible) horizontal curve $H'$ with $H'\cdot F=4$.
 Note that $H'$ is contained in the fixed part of $|K_S|$ by Lemma~\ref{lem: fix vs base} and $H\cdot F\leq K_S\cdot F=4$.
  We must have $H=H'$. Note that any section of $f$ must be $\sigma$-fixed. So there are
  two possibilities for $H$ as given in (i) of the lemma.

\medskip

(ii) We have  $M\cdot F=(K_S-H-V)\cdot F=0$ by (i).  This yields (ii).

\medskip

(iii)-(iv) We have showed that $H$ is a reduced $\alpha$-fixed curve
and any component of $H$ is not $\sigma$-fixed. So (iii) and (iv) would follow
if there are no $\alpha$-fixed curves apart from $H$.

 Otherwise, let $\Theta$ be any irreducible $\alpha$-fixed
curve not contained in $H$. Then $\Theta<V$ (cf.~Lemma~\ref{lem: fix vs base}).  So
$\Theta<F'$ for some fiber   $F'$ of $f$. Let $m=\op{mult}_\Theta
F'$.  By  (ii) of Lemma~\ref{lem: fix vs sing fib}, we have $2 \mid m $. Since
$H\cdot F'=4$ and $H\cdot \Theta=0$, we have $F'\not=m\Theta$. This implies
$K_S\cdot \Theta=1$ and $ m=2$, since $K_S\cdot F'=4$.
So we may write $F'=2\Theta+A$ for some effective divisor $A$. Then
we have $K_S\cdot A=2$ and $H\cdot A=H\cdot F'=4$.

Let $C$ be an irreducible component of $A$. Since $H$ and $\Theta$ are $\alpha$-fixed, either $C$ is $\alpha$-fixed and $(H+\Theta)\cdot C =0$ or $C$ is not $\alpha$-fixed and by Lemma~\ref{t1}
\begin{align}
\label{leq2}(H+\Theta)\cdot C\leq 2g(C)+2.
\end{align}

First we assume that $A$ is not reduced. Then $A=2C$ with
 $K_S\cdot C=1$.
By the adjunction formula, we have $C^2=-1$ or $-3$. From
$0=C\cdot F'=2C\cdot \Theta+2C^2$, follows that $C\cdot \Theta=-C^2$.

If  $C$ is a $(-3)$-curve, then $C\cdot \Theta=3$, which contradicts \eqref{leq2}.

If  $C^2=-1$,   then $p:=\Theta\cap C$ is a $\sigma$-fixed point of
weight $\frac{1}{4}(1, 2)$. Take suitable local coordinates $(x, y)$ of $S$ around $p$, and a local coordinate $t$ of $B$ around $f(p)$, such that
\[
p=(0,0), \quad C=(x=0),\quad \Theta=(y=0),\quad  \sigma(x,y) = (-x, \xi y),\quad f^*t=x^2y^2.
\]
where $\xi$ is a primitive $4$-th root of $1$. But then $\sigma^*(f^*t) = -f^*t \neq f^*t$, a contradiction to the fact that $\sigma$ induces the trivial action on $B$.

Now we assume that $A$ is reducible. Then $A=C+C'$ with
$K_S\cdot C=K_S\cdot C'=1$. Since $F'$ is connected, we have $A\cdot \Theta>0$ and we may assume $C
\Theta>0$. By the adjunction formula, we have $C^2=-1$ or $-3$.
 From
$0=C\cdot F'=2C\cdot \Theta+C^2+C\cdot C'$, we get $C^2\leq-2$. So $C$ is a
$(-3)$-curve, $C\cdot \Theta=1$ and $C\cdot C'=1$. From $0=\Theta\cdot F'=2\Theta^2+\Theta\cdot C+\Theta\cdot C'$,
 we get $\Theta\cdot C'> 0$. So $C'$ is also a $(-3)$-curve.

 By (\ref{leq2}), we have $(H+\Theta)\cdot C\leq2 $ and $(H+\Theta)\cdot C'\leq2 $.
  So $(H+\Theta)\cdot(C+C')\leq4 $.
 Since
$H\cdot (C+C')=4$, we have $\Theta\cdot (C+C')\leq0$, which is absurd.

Finally we assume that $A$ is irreducible. From $A=F'-2\Theta$,  we
get $A^2=4\Theta^2\leq-4$.
  So $A$ is a $(-4)$-curve by the adjunction formula.  By (\ref{leq2}), we have
   $H\cdot A\leq2 $, a
contradiction since $H\cdot A=4$. \qed\end{pf}
 \begin{cor}\label{equ6.1}We have
\begin{enumerate}
\item $K_S^2=8\chi(\mc O_S)+H^2$;
\item
$k_3-k_1=K_S^2-8\chi(\mc O_S)$;
\item
 $k_1+k_2+k_3=12\chi(\mc O_S)-K_S^2$;
\item
 $(H+K_S)\cdot V=16$.
\end{enumerate}
\end{cor}

 \begin{pf} (i) follows from Lemma~\ref{H-3} (ii) and Lemma~\ref{lem: sign inv}.

\medskip

(ii) and (iii) follow from Lemma~\ref{H-3} (i), (\ref{sign1}) and
(\ref{top}).

\medskip

(iv)   Since $K_S\equiv M+H+V$ and $M\equiv(\chi(\mc
O_S)-2)F$,  we have
 \begin{align*}
 K_S^2=&K_S\cdot M+ (M+H+V)\cdot H+K_S \cdot V\\
 =&(K_S+H)\cdot M+(H+V)\cdot H+K_S \cdot V\\
 =&
8\chi(\mc O_S)-16+ H^2+(H+K_S)\cdot V.
 \end{align*}
 Combining this  equality with (i), we obtain (iv).
 \qed\end{pf}

\begin{nota}
We have  a commutative diagram
\[\begin{tikzcd}
   \oS  \arrow{r}{\bar \mu}\arrow{d}{\eta'}  &\oY:=\oS/\bar \alpha\arrow{d} \\
S\arrow{r}{\mu}  & Y:=S/\alpha
\end{tikzcd}\]
where  $\eta$ is the blowup of all
 isolated fixed points of $\alpha$,
  $\bar \alpha $  is the lift of $\alpha$
 on $\oS$. Let  $j: \oY\to B$ be the fibration induced by
 $f$, so that $f\circ\eta=j\circ \bar\mu$.\end{nota}
\begin{lem}\label{min} The elliptic fibration $j$ is relatively minimal.\end{lem}
\begin{proof}
Otherwise, let $\oC$ be a $(-1)$-curve contained in $j^*b$ for some
$b\in B$,  $\tilde C=\bar\mu^*\oC$, and $C=\eta_*\tC$.

\begin{claim} $C$ is an irreducible curve with $C^2\geq-2$
 and $\op{mult}_C f^*b\geq2$.\end{claim}
\begin{proof}[Proof of the claim.]Since each irreducible component of
$(f\circ\eta)^*b$ is $\ti\alpha$-invariant, we have that  $\tilde C$
is an irreducible curve with $\tilde C^2=-2$. Hence $ C$ is an
irreducible curve with $C^2\geq-2$.

We show that  $\op{mult}_{\oC}j^*b\geq2$: otherwise, since $\bar
C^2=-1$, we have that $\oC$ and $j^*b-\oC$  meets exactly at
one point, say $q$,  and $j^*b$ is smooth along $\oC\setminus\{q\}$.

Let $\bar \sigma$  be the involution of $\oY$ induced by
$\sigma$.

   Since $g(F/\langle\alpha\rangle)=1$,
    we have that $\bar \sigma$ is a translation of  $F/\langle\alpha\rangle$, a general fiber of $j$.
     Hence  $\bar \sigma$ has no horizontal fixed curves.

  Since $C$ is $ \sigma$-invariant, we have
$\oC$ is $\bar \sigma$-invariant. By (ii) of Lemma~\ref{lem: fix vs sing fib}, $\oC$ is not $\bar \sigma$-fixed.
     Since $\oC$ is a smooth rational curve and $\bar \sigma$ has no horizontal
      fixed curves,   there is an isolated $\bar \sigma$-fixed point  on $\oC\setminus\{q\}$.
      By (i) of Lemma~\ref{lem: fix vs sing fib}, such a point is a singular point of $j^*b$,
       a contradiction since $j^*b$ is smooth along $\oC\setminus\{q\}$.

From
  $f\circ\eta=j\circ\bar{\mu}$,
we have $\op{mult}_Cf^*b=\op{mult}_{\oC}j^*b\geq2$.
   \end{proof}
By the claim above, we have $\op{mult}_C f^*b\geq2$. So $H\cdot C\leq2$
since $H\cdot f^*b=4$.

By Lemma~\ref{H-3}, we have $K_S-H=\mu^*K_Y$.  So \[
K_S\cdot C-H\cdot C=K_Y\cdot\mu_*C =K_{\oY}\cdot\lambda^*(\mu_*C)=2K_{\bar
Y}\oC=-2.\]

  From $H\cdot C\leq2$ and  $K_S\cdot C-H\cdot C=-2$
we have
 $H\cdot C=2$ and $K_S\cdot C=0$, a contradiction since $K_S$ is $f$-ample by
 assumption.
\end{proof}
\begin{lem}\label{k-1'}Let $b_1,\cdots, b_l\in B$ be the image of the support of $V$ under
$f$. For $i=1,\cdots, l$, write $f^*b_i=m_iC_i$, where $m_i$ is the
multiplicity of $f^*b_i$ and $C_i$ is a $1$-connected  effective
divisor on $S$.  We have  \begin{enumerate}
\item  $l=4$ and $m_i=2$ for all $i$;
 \item$V=\sum_{i=1}^4 C_i$;
\item  $j^*b_i$ ($i=1, \cdots, 4$) is  a multiple
fiber of $j$, and the multiplicity of $j^*b_i$ is $2$;
\item if $q\in S$ is an isolated $\alpha$-fixed point and $q\in C_i$
for some $i$, then  $\op{mult}_q C_i$ is divisible by $2$.
 \end{enumerate}
 \end{lem}
\begin{pf}
(i)-(ii) By (ii) of Lemma~\ref{H-3}, we have $K_S=\mu^*K_Y+H$. By
Lemma~\ref{min}, we have $K_{\bar{Y}}^2=0$ and hence $K_Y^2=0$. So
we have $V^2=(M+V)^2=(K_S-H)^2=0$. This implies  $V=\sum_{i=1}^l
u_iC_i$ for some integers $u_i\geq1$. Now since $C_i$ is contained
in the fixed part of $|K_S|$ and $B$ is a rational curve, (note that
$m_i=1$ or $2$ since $m_iC_i\cdot K_S=4$,) we must have $m_i=2$ and
$u_i=1$.

Since $V=\sum_{i=1}^l C_i$, by (iv) of Corollary~\ref{equ6.1}, we
have $l=4$.

\medskip

(iii)   Write $j^*b_i=v_i R_i$, where $v_i$ is the multiplicity of
$j^*b_i$ and $R_i$ is a $1$-connected  effective divisor on $\bar
Y$. By the Hurwitz formula, $K_{\oS}-\bar{\mu}^*K_{\bar{Y}}$ is an $\bar\alpha$-fixed
divisor, we have
$\sum_{i=1}^l\oC_i=\sum_{i=1}^l(v_i-1)\bar{\mu}^*R_i$,
 where $\oC_i$ is the strict transform of $C_i$.  This
implies $v_i=2$.

(iv)  Let $E=\eta^*q$. Then $E$ is $\bar\alpha$-fixed and
$\op{mult}_E (f\circ\eta)^*b_i=2n$, where $n=\op{mult}_q C_i$. So
from $(f\circ\eta)^*b_i=\bar\mu^*(j^*b_i)$, we have $\op{mult}_{\bar
E} j^*b_i=n$, where $\oE=\bar\mu(E)$.
  Since $j^*b_i=2R_i$ by (iii),  we have that $n$ is divisible by $2$.\qed\end{pf}

 To classify singular fibers of $f$, we will use frequently  the following two simple
 observations.
\begin{lem}\label{412}
Let $C'$ and $C''$ be two irreducible components of a fiber $f^*b$, and $p\in
C'\cap C''$ is a $\sigma$-fixed point.  Then either
$(C'.C'')_p\geq2$
 or  $p\not \in
H$.
\end{lem}
\begin{pf} Suppose on the contrary that $(C'.C'')_p=1$
 and  $p \in H$. Then $C'$ and $C''$ are smooth at $p$, and they meet
transversely in  $p$. Since $H$ is $\alpha$-fixed
(cf.~Lemma~\ref{H-3}), $p$ is an isolated
$\sigma$-fixed point
 of
weight $\frac{1}{4}(1, 2)$  and $H$ meets
  $C'$ (resp.~$C''$) transversely in  $p$.  This implies that there  are at least three
$\sigma$-invariant  one dimensional linear subspaces of $T_{p}S$,
and hence
 the projective transformation of $\mathbb{P}(T_{p}S)\simeq\mathbb{P}^1$ induced by $\sigma$ must be the identity map. It follows that $p$ is an isolated
$\sigma$-fixed point
 of
weight $\frac{1}{4}(1, 1)$, a contradiction. \qed\end{pf}

\begin{lem}\label{rational}Let $C$ be an irreducible component of a fiber $f^*b$.
 Let $\tC$ be the normalization of $C$ if $C$ is singular,
and $\ti\alpha, \, \tilde{\sigma}$ the induced action of $\alpha,\,
\sigma$ on $\tC$, respectively.

(i) If $H$ and $C$ meet transversely in two points, which are not
$\sigma$-fixed, then $C\not\simeq \mathbb{P}^1$.

(ii)  If $\ti C$ has two $\ti\alpha$-fixed points, which are not
$\ti\sigma$-fixed, then $\ti C\not\simeq \mathbb{P}^1$.
\end{lem}

\begin{pf}   Note that $C$ is $\sigma$-invariant (\cite[Lemma 3.1]{Cai04}).
 Since $H$ is $\alpha$-fixed (cf.~Lemma~\ref{H-3}), the two points in $H\cap C$ are $\alpha$-fixed under the assumption of (i).

Now (i) and (ii) follow from the simple fact that for each
automorphism $\sigma$ of order $4$ of a smooth rational curve, the
$\sigma^2$-fixed points are $\sigma$-fixed. \qed\end{pf}
\begin{lem}\label{c-i} For each $b\in B$, if $f^*b=2C$ is a multiple fiber
  then the possible configurations for   $C$ and the action of $\sigma$ on $C$ are
   as  follows:
  \begin{enumerate}
\item
  $C$ is a smooth curve of genus $2$, and  $H$ meets $C$ transversely in two
points of weight $\frac{1}{4}(1, 2)$.
\item  $C$ is   an irreducible curve
with a single  node $p$
 of
weight $\frac{1}{4}(1, 1)$. $C$ and $H$  meet transversely in two
points, which are not $\sigma$-fixed.
\item
  $C$ is an irreducible curve with a single node $p$
 of
weight $\frac{1}{4}(1, 3)$, and  $H$ meets $C$ transversely in two
points
 of
weight $\frac{1}{4}(1, 2)$.
\item  $C$ is an irreducible curve with one tacnode $p$
 of
weight $\frac{1}{4}(1, 2)$.
\item
 $C=C'+C''$, where    $C'$ and $C''$ are $(-3)$-curves
 meeting in two points,  a node of
weight $\frac{1}{4}(1, a)$, where $a=1$ or $3$,  and a tacnode  of
$C$
 of
weight $\frac{1}{4}(1, 2)$.
\item
 $C=C'+C''$, where    $C'$ and $C''$ are elliptic curves
 meeting transversely  in one point
 of
weight $\frac{1}{4}(1, 3)$.  $C'$ (resp.~$C''$) and $H$ meet
transversely in one point
 of
weight $\frac{1}{4}(1, 2)$.
\end{enumerate}
\end{lem}
\begin{pf}
   Since  $K_S\cdot C=2$ and $K_S$ is $f$-ample, $C$ has at most two irreducible components.

\medskip

 {\bf Case 1.  $C=C'+C''$ is reducible.}
  Since $K_S\cdot C'=K_S\cdot C''=1$,
by the adjunction formula we have that  either
  $C'$
 and $C''$ are  $(-3)$-curves with $C'\cdot C''=3$ or they are curves of arithmetic genus one
  with $C'\cdot C''=1$.
In the former case, by  Lemma~\ref{412}, we have that   $C$ is as in
(v) of Lemma~\ref{c-i}.

In the latter case, if $C'$ is singular, then the singular point,
say  $q$, of $C'$ must be $\sigma$-fixed. If $q$ is a node of $C'$,
there are at least three $\ti{\alpha}$-fixed points on $\ti{C'}$,
the normalization of $C'$, a contradiction since $\ti{C'}$ is a
smooth rational curve; if $q$ is a cusp of $C'$, by Lemma~\ref{lem: cusp}
and (iii) of Lemma~\ref{H-3}, we have that $H$ passes through $q$
and hence $H\cdot C'\geq2$. Thus $H\cdot C''=0$, and so $C''$ is a $(-2)$-curve, a
contradiction since $K_S$ is ample.

So $C'$ and $C''$ are elliptic curves
 meeting transversely  in one point, which  is an isolated $\sigma$-fixed point
 of
weight $\frac{1}{4}(1, a)$, where $a=1$ or $3$. The case $a=1$ does
not occur: otherwise, by Lemma~\ref{lem: fix vs base}, we have $C<V$. Clearly
$H$ meets $C'$ transversely in one point $q'$, which is isolated
$\sigma$-fixed  of weight
 $\frac{1}{4}(1, 2)$.  Consider the action of $\sigma$ on the local
 neighborhood
  of $q'$,  by an argument as in Lemma~\ref{lem: fix vs base}, there
exists a holomorphic $2$-form $\omega\in H^0(\omega_S)$ such that
$\sigma^*\omega=-\omega$,  a contradiction.  Thus  $C$ is as in (vi)
of Lemma~\ref{c-i}.

\medskip

{\bf Case 2.
 $C$ is irreducible.} If $C$ is smooth,  then $C$ is as in (i) of
Lemma~\ref{c-i}.

If $C$ is singular, let $\ti C$ be the normalization of $C$, and
$\ti \alpha$ the induced involution  of $\alpha$ on $\ti C$.

If $g(\ti C)=1$, let $p\in C$ be the unique singular point of $C$.
Then $p$ is $\sigma$-fixed. If $p$ is a node of $C$,  then $C$ is as
in (ii) or (iii) of Lemma~\ref{c-i}.

 If $p$ is a cusp of $C$, then
 $p$ is an isolated $\sigma$-fixed point
 of
weight $\frac{1}{4}(1, 2)$ by Lemma~\ref{lem: cusp}, and   $H\cap C=\{p\}$
since $H\cdot C=2$. Then  $\sigma$ has a fixed point $q$ on
$C\setminus\{p\}$ (since  $\ti\sigma$ has two fixed points on $\ti
C)$). $q$ is an isolated $\sigma$-fixed point
 of
weight $\frac{1}{4}(1, a)$, where $a=1$ or $3$, a contradiction by
Lemma~\ref{lem: singular'}.

Now we consider the case $g(\ti C)=0$. Since $p_a(C)=2$, $C$ has at
most two singular points.

First we assume $C$ has two singular points, say $p, q$. We have $p$
and $q$ are $\alpha$-fixed.

If  $p$ and $q$ are cusps, then $H$ passes through both $p$ and $q$
by Lemma~\ref{lem: cusp}. This implies $H\cdot C\geq4$, a contradiction since
$H\cdot C=2$.

If  $p$ is a node  and $q$ is a cusp, then $H$ passes through  $q$
and $p$ is an isolated $\sigma$-fixed point. Clearly $\alpha$ fixes
the two local branches of $C$ at $p$. This implies that $\ti \alpha$
has at least three fixed points on $\ti C$, which is impossible
since $\ti C\simeq\mathbb{P}^1$.

Now we may assume that  $p$ and $q$ are nodes. If both $p$ and $q$
are $\sigma$-fixed, then $\alpha$ fixes the two local branches of
$C$ at both $p$ and $q$. This implies that $\ti \alpha$ has at least
four fixed points on $\ti C$, which is impossible since $\ti
C\simeq\mathbb{P}^1$.

So we may assume  $\sigma(p)=q$.
  If $H$
 passes through $p$ then $H$ passes through  $q$ since $\sigma(p)=q$ and
 $H$ is $\sigma$-invariant. This implies $H\cdot C\geq4$, a contradiction since $H\cdot C=2$.

So we have that $p, q$ are isolated $\alpha$-fixed. This implies
$\alpha$ fixes the two local branches of $C$ at both $p$ and $q$,
and hence $\ti \alpha$ has at least four fixed points on $\ti C$,
which is impossible since $\ti C\simeq\mathbb{P}^1$.

 Second we assume that  $C$ has only one singular point, say $p$. Then $p$
is $\sigma$-fixed.

If $p$ is a cusp of $C$, then $\sigma$ has exactly one fixed point
$q$ on $C\setminus\{p\}$ (since $\ti \sigma$ has two fixed points on
$\ti C$ and the inverse image of $p$ in $\ti C$ is one point), and
there are no $\alpha$-fixed points on $C\setminus\{p, q\}$. By
Lemma~\ref{lem: singular'}, $H$ passes through $q$. So $\{
q\}\subseteq H\cap C\subseteq \{p, q\}$. This implies either $H\cdot C=1$
or $H\cdot C\geq3$ since $(H.C)_q=1$ and $(H.C)_p\geq2$, a contradiction
since $H\cdot C=2$.

If  $p$ is a tacnode of $C$, then $\alpha$ fixes the two local
branches of $C$ at $p$.  This implies that $\alpha$ has no fixed
points on $C\setminus\{p\}$. So $C$ as in (iv) of Lemma~\ref{c-i}.
\qed\end{pf}

\begin{cor}\label{c-i'} For $i=1,\cdots, 4$, let $f^*b_i=2C_i$ be as in Lemma~\ref{k-1'}.
  Then  the possible configurations for  $C_i$ and the action of $\sigma$ on $C_i$ are
   as in (ii),   (iv) or (v)  of Lemma~\ref{c-i}.
\end{cor}
\begin{pf}
If $C_i$ is as in (i)  of Lemma~\ref{c-i}, denote $C_i$ by $C$,
applying the Hurwitz formula
  to the quotient map $\varpi:C\to C/\sigma$,  we have  that  either $g(C/\sigma)=0$
  and   $\alpha$  has $6$ fixed points on $C$, or $g(C/\sigma)=1$
  and $\varpi$ has only one branch point. In the former case,  since $H\cdot C=2$,
   we get a contradiction by (iv) of Lemma~\ref{k-1'}.  The latter case does not occur
   since any non-\'etale Abelian cover over an elliptic curve has at least two branch points.

If $C_i$ is as in (iii) or (vi)  of Lemma~\ref{c-i}, then $H$ meets
$C_i$ transversely in two points, which are  isolated
$\sigma$-fixed  of weight
 $\frac{1}{4}(1, 2)$.  Consider the action of $\sigma$ on the local
 neighborhood
  of one of such points,  by an argument as in Lemma~\ref{lem: fix vs base},  there
exists a holomorphic $2$-form $\omega\in H^0(\omega_S)$ such that
$\sigma^*\omega=-\omega$,  a contradiction.\qed\end{pf}
\begin{cor} \label{c-i''}$k_1\leq4$.\end{cor}
\begin{pf}For each $p\in S$, if $p$ is an isolated
$\sigma$-fixed point
 of
weight $\frac{1}{4}(1, 1)$, by Lemma~\ref{lem: fix vs base}, we have $p\in C_i$
for some $i$, where $C_i$ is as in Lemma~\ref{k-1'}. By
Corollary~\ref{c-i'}, there are at most one isolated $\sigma$-fixed
point
 of
weight $\frac{1}{4}(1, 1)$ on $C_i$. So by (ii) of Lemma~\ref{k-1'},
we have $k_1\leq4$. \qed\end{pf}

\begin{lem}\label{k3'}For each $b\in B$, if the fiber $F':=f^*b$ is singular
and irreducible,
 then the possible configurations for $F'$ and the action of $\sigma$ on $F'$
  are as follows.
 \begin{enumerate}
\item
 $F'$ has a single node
  of weight $\frac{1}{4}(1, 3)$, and $\sigma$ fixes the two local branches of $F'$ at
  $p$. $H$ meets $F'$   transversely in four points,
  and   $\sigma$ maps  two of them to the others.
\item
 $F'$ has  exactly two nodes,  $H$ passes though
 them,  $\sigma$ maps one the the other, and the lift of $\sigma$ on the normalization of
 $ F'$ is
  an translation of order four.
  \item
 $F'$ has  three nodes, say $p_1, p_2, p_3$.  $H$ passes though
 $p_1$ and  $p_2$,  $\sigma(p_1)=p_2$,
 $p_3$
 is of weight $\frac{1}{4}(1, 3)$, and $\sigma$ fixes the two local branches
 of $F'$ at $p_3$.
\item
$F'$ has one tacnode $p$
 of
weight $\frac{1}{4}(1, 2)$.  The normalization of $F'$ is an elliptic
curve. $ H$ meets $F'\setminus\{p\}$ transversely  in two points,
which are not $\sigma$-fixed.
\item
$F'$ has a single  singular point $p$
 of
weight $\frac{1}{4}(1, 2)$, $\op{mult}_pF'=3$, and   $F'$ has exactly
two local branches  at $p$.
\item
$F'$ has two cusps
 of
weight $\frac{1}{4}(1, 2)$.   One of them is not an ordinary cusp.
\item
$F'$ has one tacnode $p$,
  which is  an isolated $\sigma$-fixed point
 of
weight $\frac{1}{4}(1, 2)$.  The normalization of $F'$ is a rational
curve.
\end{enumerate}
 \end{lem}
\begin{pf}Let $\ti F'$ be the
normalization of $F'$, and $\ti \sigma$, $\ti \alpha$  the induced
automorphism of $\sigma$, $\alpha$ on $\ti F'$, respectively. Let
$p_1, \cdots, p_u$ be all points in $F'\cap H$,  where $u$ is  the
cardinal number of
 $F'\cap H$. We have $u\leq4$ since $H\cdot F'=4$. We distinguish between
 four cases according to the value of $u$.

\medskip

\noindent{ \bf Case 1.  $u=4$.} In this case $p_1, \cdots, p_4$ are
smooth points of $F'$ since $H\cdot F'=4$. Since  $\ti \alpha$ has at
least four fixed points on $\ti F'$ (the inverse image of $p_1,
\cdots, p_4$), we have $g(\ti F')=1$ or $2$.

If $g(\ti F')=2$, then   the singular point  $p$ of $ F'$
  is $\sigma$-fixed. Since $p\not\in H$ and
$\op{mult}_pF'=2$, by Lemma~\ref{lem: singular'}, $p$ is an isolated
$\sigma$-fixed
 point of weight $\frac{1}{4}(1, 3)$ and  $p$ is a node of $F'$.
Applying the Hurwitz formula to the quotient map $\ti F'\to \ti
F'/\ti \sigma$, we have that $\ti \sigma$ has two fixed points. So
$\sigma$ fixes the two local branches of $F'$ at $p$. Thus $F'$ is
as in (i) of Lemma~\ref{k3'}.

 If $g(\ti F')=1$,  note that the inverse image $\ti p_i$ of  $p_i$
are four $\ti \alpha$-fixed points on $\ti F'$, $\ti \alpha$ has no
fixed points
 on $\ti F'\setminus\{\ti p_1, \cdots, \ti p_4\}$.
 This implies that there are exactly   two nodes on $F'$ and $\sigma$ maps
 one to the other.
 Thus $\ti \sigma$ has no fixed points on $\ti F'$. Applying the Hurwitz formula
to the quotient map $\ti F'\to \ti F'/\ti \sigma$, we have that it
has eight ramification points of ramification index two. This
implies that $\ti \alpha$ has eight fixed points on $\ti F'$, a
contradiction.

\medskip

\noindent{\bf Case 2. $u=3$.} Note that since $H$ is $\alpha$-fixed,
we have that $(H.F')_{p_i}>1$ if and only if  $p_i$ is a singular
point of $F'$. Since $H\cdot F'=4$, we may assume that $p_1$ and $p_2$ are
smooth points of $F'$ and $\op{mult}_{p_3}F'=2$. Then $\sigma$ maps
$p_1$ to $p_2$ by Lemma~\ref{lem: fix vs sing fib},  and $p_3$ is an
isolated $\sigma$-fixed point
  of weight $\frac{1}{4}(1, 2)$.
By (ii) of Lemma~\ref{rational}, we have   $g(\ti F')\geq1$.

If $g(\ti F')=2$, then $p_3$ is a cusp of $F'$ by Lemma~\ref{lem:
singular'} and $\alpha$ has no fixed points on $F'\setminus\{p_1,
p_2, p_3\}$ by Lemma~\ref{lem: fix vs sing fib}. Thus $\ti \alpha$ has
exactly three fixed points (the inverse image of $p_1, p_2, p_3$) on
$\ti F'$, and we get a contradiction by applying the Hurwitz formula
to the quotient map $\ti F'\to \ti F'/\ti \alpha$.

  Now we assume  $g(\ti F')=1$. If $F'\setminus\{p_1,
p_2, p_3\}$ is smooth, then $F'$ is as in  (iv) of Lemma~\ref{k3'}.

 If $F'\setminus\{p_1,
p_2, p_3\}$ is singular, let $p\in F'\setminus\{p_1, p_2, p_3\}$ be
the (unique) singular point. Since $\op{mult}_pF'=2$ and $p\not\in
H$, by Lemma~\ref{lem: singular'}, $p$ is an isolated $\sigma$-fixed
 point of weight $\frac{1}{4}(1, 3)$ and  $p$ is a node of $F'$.
Then $\ti \alpha$   has exactly at least five fixed points (the
inverse image of $p_1, p_2, p_3$ and $p$) on $\ti F'$, which is
absurd.

\medskip

\noindent{\bf Case 3. $u=2$.}
 If  $\sigma(p_i)=p_i$ for $i=1$, $2$, then  $p_1, p_2$ are
isolated $\sigma$-fixed of weight
 $\frac{1}{4}(1, 2)$, and they are cusps of $F'$ by
Lemma~\ref{lem:
  singular'}. So $g(\ti F')\leq1$.

If $g(\ti F')=1$, then $\alpha$ has two isolated fixed points
 on $F'\setminus\{p_1, p_2\}$, which is smooth,   a contradiction by Lemma~\ref{lem: fix vs sing fib}.

If  $g(\ti F')=0$, then  $F'\setminus\{p_1, p_2\}$ is smooth
(otherwise, $\ti \alpha$
  has at least three fixed points on $\ti F'$, a contradiction), and
$F'$ is as in (vi) of Lemma~\ref{k3'}.

 Now we assume $\sigma(p_1)=p_2$. Since $H\cdot F'=4$,  $p_1, p_2$ are
singular points of $F'$. So $g(\ti F')\leq1$.

 If $g(\ti F')=1$,  then $F'\setminus\{p_1, p_2\}$ is smooth, and either  $F'$ is as in (ii)  of Lemma~\ref{k3'} or $p_1,
p_2$ are cusps of $F'$. In the latter case, besides  the inverse
image of $p_1$ and $p_2$,  $\ti \alpha$ has two fixed points on $\ti
F'$ and hence on  $F'\setminus\{p_1, p_2\}$. Clearly these two
points are isolated $\alpha$-fixed,   a contradiction by
Lemma~\ref{lem: fix vs sing fib} since  $F'\setminus\{p_1, p_2\}$ is
smooth.

If $g(\ti F')=0$,  then either  $F'$ is as in (iii) of
Lemma~\ref{k3'} or  $p_1, p_2$ are cusps of $F'$.  The latter case
is excluded  by (ii) of Lemma~\ref{rational}.

\medskip

 \noindent{\bf Case 4. $u=1$.}
Since $H\cdot F'=4$ and $p_a(F')=3$,  we have $\op{mult}_pF'=2$ or $3$. If
$\op{mult}_pF'=3$, then $F'\setminus\{p\}$ is smooth, and $F'$ as in
(v) of Lemma~\ref{k3'}.

If  $\op{mult}_pF'=2$, then $p$ is a cusp or a tacnode of $F'$ by
Lemma~\ref{lem: singular'}. First we assume $p$ is a cusp. Then the
inverse image $\ti p\in \ti F'$ of $p$ is $\ti \sigma$-fixed. If
$g(\ti F')=2$, then $\alpha$ has isolated fixed points
 on $F'\setminus\{p_1\}$, which is smooth,   a contradiction by
 Lemma~\ref{lem: fix vs sing fib}.

Since  $g(\ti F')\leq1$, there are exactly two $\ti \sigma$-fixed
points, $\ti p$ and  $\ti p'$, on $\ti F'$. Let $ p'\in F'$ be the
image of $\ti p'$. Then either $p'$ is a smooth point of $F'$ or
$p'$ is a cusp of $F'$.  Since $\op{mult}_{p'}F'\leq2$, by
Lemma~\ref{lem: singular'}, $p'$ is an isolated $\sigma$-fixed point
of weight $\frac{1}{4}(1, 2)$. This is absurd since $p'\not \in H$.

 Now we assume $p$ is a tacnode. Then $g(\ti F')\leq1$ and  the
inverse image $\ti p_1, \ti p_2\in \ti F'$ of $p$ are $\ti
\alpha$-fixed. If $g(\ti F')=1$, then $ F'\setminus\{ p\}$ is smooth
and there are two $\ti \alpha$-fixed points on $\ti F'\setminus\{\ti
p_1, \ti p_2\}$. This implies that there are two $ \alpha$-fixed
points on $ F'\setminus\{ p\}$, which is impossible by
Lemma~\ref{lem: fix vs sing fib}.

If  $g(\ti F')=0$, then $\ti \alpha$ has no fixed points on  $\ti
F'\setminus\{\ti p_1, \ti p_2\}$. So $ F'\setminus\{ p\}$ must be
smooth, and  $F'$ as in (vii) of Lemma~\ref{k3'}. \qed\end{pf}

\begin{lem}\label{k3''}For each   $b\in B$,  if the fiber $F':=f^*b$ is reducible but
 not a multiple fiber,
 then the possible configurations for $F'$ and the action of $\sigma$ on $F'$
  are as follows.
\begin{enumerate}
\item[(viii)]
 $F'=C_1+C_2$, where $C_1, C_2$ are elliptic curves meeting transversally
  in two points of weight $\frac{1}{4}(1, 3)$.
\item[(ix)]
 $F'=C_1+C_2$, where $C_1$ is a $(-3)$-curve and  $C_2$ is an elliptic curve.
 They  meet  tangentially in one point of weight $\frac{1}{4}(1, 2)$  and
  transversely in one  point
  of weight $\frac{1}{4}(1, 3)$.
\item[(x)]
$F'=C_1+C_2$, where $C_1$ and $C_2$ are $(-4)$-curves meeting in two
points
 of
weight $\frac{1}{4}(1, 2)$.
\item[(xi)]
$F'=C_1+C_2+C_3$, where $C_1$ and $C_2$ are $(-3)$-curves, $C_3$  is a
$(-4)$-curve. They meet in three points, say $p_1, p_2, p_3$: $C_1$
and $C_2$ meet transversely in $p_3$
 of
weight $\frac{1}{4}(1, 3)$, $C_i$  and $C_3$ meet tangentially in
$p_i$  of weight $\frac{1}{4}(1, 2)$ for $i=1$, $2$.
\item[(xii)]
$F'=C_1+C_2+C_3$, where $C_1$ and $C_2$ are $(-3)$-curves, $C_3$  is
an elliptic curve. They meet in three points, say $p_1, p_2, p_3$:
$C_1$ and $C_2$ meet tangentially  in $p_3$
 of
weight $\frac{1}{4}(1, 2)$,  $C_i$  and $C_3$ meet transversely in
$p_i$ of weight $\frac{1}{4}(1, 3)$ for $i=1$, $2$.
\item[(xiii)]
$F'=C_1+\cdots+C_4$, where  $C_1, \cdots, C_4$  are $(-3)$-curves
meeting in four points, say $p_1, \cdots, p_4$: $C_1$ and $C_2$,
$C_3$ and $C_4$ meet tangentially in  $p_1$, $p_2$ of weight
$\frac{1}{4}(1, 2)$, respectively;
 $C_1$ and $C_3$, $C_2$ and $C_4$ meet transversely  in
$p_3$, $p_4$ of weight $\frac{1}{4}(1, 3)$, respectively.
\end{enumerate}
 \end{lem}
\begin{pf}Write $F'=m_1C_1+m_2C_2+\cdots$ with $m_1\geq m_2\geq\cdots$, where $C_i$ are irreducible components of
$F'$ and $m_i=\op{mult}_{C_i}F'$.
 Note that for each   $C_i$, $H\cdot C_i>0$  since $K_S$ is ample
 by assumption. We have
 \[m_1H\cdot C_1+m_2H\cdot C_2+\cdots=4.\]

 First  we show that $m_1=1$, that is, $F'$ is reduced.

If $F'=3C_1+C_2$, then $K_S\cdot C_i=1$ for  $i=1$, $2$, and
$C_2^2=(-3C_1)^2=9C_1^2\leq-9$, which is absurd by the adjunction
formula.

If $F'=2C_1+C_2$, then  $K_S\cdot C_1=H\cdot C_1=1$,
 $K_S\cdot C_2=H\cdot C_2=2$, and  $C_2^2=(-2C_1)^2=4C_1^2\leq-4$. So $C_2$ is a
$(-4)$-curve.

 Since $C_2$ is smooth, $H$ meets
$C_2$ transversely in two points, say $p_1$, $p_2$. By (i) of
Lemma~\ref{rational}, $p_1$ and $p_2$ are isolated $\sigma$-fixed
points (of weight
 $\frac{1}{4}(1, 2)$). This implies  $F'$ is singular at $p_1$  and
 $p_2$ by Lemma~\ref{lem: fix vs sing fib}. This yields
 $p_1, p_2\in C_1$ (since $C_2$ is smooth and $m_2=1$),  a contradiction
 since $H\cdot C_1=1$.

If $F'=2C_1+C_2+C_3$, then  $K_S\cdot C_i=H\cdot C_i=1$,  $C_i^2=-1$  or $-3$
for all $i$. We may assume $C_2^2\geq C_3^2$.
    From $(-2C_1)^2=(C_2+C_3)^2= C_2^2+C_3^2+2C_2C_3$, we have $C_1^2=-1$, and
    either
\begin{enumerate}
\item[(a)]
     $C_2C_3=1$ and $C_2$ and $C_3$   are
$(-3)$-curves, or
\item[(b)]  $C_2C_3=0$,  $C_2^2=-1$ and $C_3$   is a $(-3)$-curve.
\end{enumerate}
The latter case does not occur: otherwise, from $F'C_3=0$, we get
$2C_1\cdot C_3-3=0$, which is absurd.

In case (a), we have $C_1\cdot C_2=C_1\cdot C_3=1$. Since $H\cdot C_i=1$,  $H$ meets
$C_i$ transversely in one smooth point, say $p_i$, which is
isolated $\sigma$-fixed.
  By Lemma~\ref{lem: fix vs sing fib}, $p_i$ is  a singular point of  $F'$.
 This implies that
  $p_3\in C_1$
or $p_3\in C_2$ since $C_3$ is smooth and $m_3=1$,  and $p_2\in C_1$
or $p_2\in C_3$ since $C_2$ is smooth and $m_2=1$.

If either $p_3\in C_2$ or $p_2\in C_3$ then $\{p_2\}=\{p_3\}=C_2\cap
C_3$, a contradiction by Lemma~\ref{412}.

Now we assume  $p_2, p_3\in C_1$.  Then  $p_2=p_1=p_3$ since
$H\cdot C_1=1$. In this case $(C_2.C_3)_{p_1}=1$, we get a contradiction
again by Lemma~\ref{412}.

From now on we assume $F'$ is reduced. Write $F'=C_1+\cdots+C_t$
with $K_S\cdot C_1\leq \cdots\leq K_S\cdot C_t$, where $t$ is the number of
irreducible components of $F'$. Let $p_1, \cdots, p_u$ be all points
in $F'\cap H$, where $u$ is the cardinal number of
 $F'\cap H$.  We distinguish several
cases according to the value of $u$:

\medskip

 \noindent{\bf
Case 1. $u=4$.}

 In this case $p_1, \cdots, p_4$ are smooth points of $F'$ since $H\cdot F'=4$,
 and  $\sigma$ maps two of them to the others since $p_i$ is not isolated
 $\sigma$-fixed by Lemma~\ref{lem: fix vs sing fib}.

 For $i=1, 2,\cdots$,  since $C_i$ is $\sigma$-invariant,
 if $p_j\in C_i$ for some $j$ then
 $\sigma(p_j)\in C_i$.
 This implies  $H\cdot C_i\geq2$.   So
 $t=2$ and
 $K_S\cdot C_i=H\cdot C_i=2$ for $i=1$, $2$. By the adjunction formula we have that  either
  $C_1$
 and $C_2$ are  $(-4)$-curves with $C_1\cdot C_2=4$ or they are curves of arithmetic genus one
  with $C_1\cdot C_2=2$.
In the former case, we get a contradiction by (i) of
Lemma~\ref{rational}.

In the latter case, if $C_i$ is singular, we get a contradiction by
(ii) of Lemma~\ref{rational}.  So $C_1$ and $C_2$ are elliptic
curves and $F'$ is as in (viii) of Lemma~\ref{k3''}.

\medskip

\noindent{\bf Case 2. $u=3$.}

By the argument as in Case 2 of Lemma~\ref{k3'},  we may assume that
$p_1$ and $p_2$ are smooth points of $F'$,  $\op{mult}_{p_3}F'=2$,
$\sigma$ maps $p_1$ to $p_2$,  and $p_3$ is an isolated
$\sigma$-fixed point
  of weight $\frac{1}{4}(1, 2)$.

Assume that $p_1\in C_j$ for some $j$.  Then $H\cdot C_j\geq2$
 since $p_2=\sigma(p_1)\in C_j$,  and so $t\leq3$.
  Note that for any $i\not=j$, $H\cdot C_i\leq4-H\cdot C_j\leq2$.
 Reindexing if necessary, we may assume $j=t$.

If $t=3$ then $K_S\cdot C_i=H\cdot C_i=1$ for $i=1$, $2$, and $K_S\cdot C_3=H\cdot C_3=2$.
 By the adjunction formula and
Lemma~\ref{rational}, we have  $C_3$ is an elliptic curve with
$C_3^2=-2$.

We have $\{p_3\}=H\cap C_i$ for $i=1$, $2$. By Lemma~\ref{412},
$(C_1.C_2)_{p_3}\geq2$. So $C_1$
 and $C_2$ are  $(-3)$-curves with $C_1\cdot C_2=2$, and  $C_i$  and $C_3$ meet
 transversely in
one point for $i=1$, $2$, which is an  isolated $\sigma$-fixed point
 of
weight $\frac{1}{4}(1, 3)$. So $F'$ is as in (xii) of
Lemma~\ref{k3''}.

Now we assume $t=2$. If $p_3\in C_2$ then $K_S\cdot C_2=H\cdot C_2=3$ (and hence
$K_S\cdot C_1=H\cdot C_1=1$).
 By the adjunction formula and Lemma~\ref{rational},
we have  $C_2$ is an elliptic curve with $C_2^2=-3$.

 Note that for $i=1, 2$,  $p_i\not\in C_1$ since $p_i$ is a smooth point of
 $F'$. So $p_3\in C_1$ since $H\cdot C_1=1$.
By Lemma~\ref{412}, $(C_1.C_2)_{p_3}\geq2$. So $C_1$
  is a  $(-3)$-curve meeting    $C_2$ tangentially in $p_3$ and
 transversely in
one other point, which is an  isolated $\sigma$-fixed point
 of
weight $\frac{1}{4}(1, 3)$. In this case $F'$ is as in (x) of
Lemma~\ref{k3''}.

If $p_3\not\in C_2$ then $p_3$ is a singular point of $C_1$ (since
it is a singular point of $F'$),  and hence $H\cdot C_1=2$.
 Since
$K_S\cdot C_1=H\cdot C_1=2$  and  $C_1$ is singular, we have $C_1^2=-2$,
$C_1\cdot C_2=2$.   If
 $C_1$  and  $C_2$ meet transversely in two points, say $p',p''$,
 then there are three $\ti \alpha$-fixed points, the inverse image of
$p',p''$ and $p_3$ (note that $p_3$ is a cusp of $C_1$ by
Lemma~\ref{lem: singular'}),  on $\ti C_1\simeq\mathbb{P}^1$, a
contradiction.

If
 $C_1$ and $C_2$ meet tangentially  in one point, say $p$, then $p$ is an
  isolated $\sigma$-fixed point
 of
weight $\frac{1}{4}(1, 2)$ by Lemma~\ref{lem: singular'}. This is a
contradiction since $p\not=p_i$ for any $i$.

\medskip

 \noindent{\bf Case 3.
 $u=2$ and $\sigma(p_1)=p_2$.}

In this case $F'\cap H$ consists of  a single $\sigma$-orbit. This
implies that $p_1, p_2\in C_i$ for each irreducible component $C_i$
of $F'$, and hence   $t=2$ and  $H\cdot C_1=H\cdot C_2=2$.  So
   $K_S\cdot C_1=K_S\cdot C_2=2$,
    $C_1^2=C_2^2=-2$ and $C_1\cdot C_2=2$. (The case $C_1^2=C_2^2=-4$ is excluded by (i) of Lemma~\ref{rational}.)

Since both $p_1$ and $p_2$ are smooth points of $C_1$ (otherwise,
  $H\cdot C_1\geq(H.C_1)_{p_1}+(H.C_1)_{p_2}\geq4$,
a contradiction), and $p_1, p_2$ are singular points of $F'$.(since
$H\cdot F'=4$),  we have $\{p_1, p_2\}=C_1\cap C_2$.

   If $C_1$ is an elliptic curve, then $\alpha$ has two isolated fixed points on $C_1$,
    which are smooth points
   of $F'$, a contradiction by Lemma~\ref{lem: fix vs sing fib}.
   If  $C_1$ is singular, by (ii) of Lemma~\ref{rational} we get a
   a contradiction.

\medskip

\noindent{\bf Case 4. $u=2$ and $\sigma(p_i)=p_i$ for $i=1$, $2$.}
Since $p_1, p_2$ are isolated $\sigma$-fixed, they are singular
points of $F'$.

\medskip

\noindent{\bf Case 4.1.
 $t=4$.} In this case $K_S\cdot C_i=H\cdot C_i=1$ and $C_i^2=-1$ or $-3$
for all $i$. Reindexing if necessary, we may assume that $C_1$ and
$C_2$ pass though $p_1$ but not $p_2$, and  $C_3$ and $C_4$ pass
though $p_2$ but not $p_1$.

By Lemma~\ref{412}, we have $(C_1.C_2)_{p_1}\geq2$ and
$(C_3.C_4)_{p_2}\geq2$. This implies   $C_i^2$ is a $(-3)$-curve for
all $i$.
 Since $F'$ is $1$-connected, we have that $C_1+C_2$ and
$C_3+C_4$ meets transversely in two points. So $F'$ is as in (xiii)
of Lemma~\ref{k3''}.

\medskip

\noindent{\bf Case 4.2.
 $t=3$.}  In this case $H\cdot C_i=K_S\cdot C_i=1$ for $i=1$, $2$,  and
$H\cdot C_3=K_S\cdot C_3=2$. Thus $C_1$ passes though $p_1$ but not $p_2$, and
$C_2$ passes though $p_2$ but not $p_1$.  For $i=1$, $2$, $p_i$ is a
smooth point of $C_i$ since $H\cdot C_i=1$.

 For $i=1$, $2$, since $p_i$ is a singular  point
of $F'$,  we have that  $C_3$ passes though both $p_1$ and $p_2$.
Then $(C_i.C_3)_{p_i}\geq2$ for $i=1$, $2$ by Lemma~\ref{412}. This
implies that $C_1$  and $C_2$ are $(-3)$-curves and $C_3$ is a
$(-4)$-curves, and so  $F'$ is as in (xi) of Lemma~\ref{k3''}.

\medskip

\noindent{\bf Case 4.3.  $t=2$ and $K_S\cdot C_1=K_S\cdot C_2=2$.} We may assume
$p_i\in C_i$ for $i=1$, $2$.

Since $H\cdot C_1=H\cdot C_2=2$, we have that either $p_i$ is a singular point
of $C_i$ for $i=1$, $2$,  or $p_1\in C_2$ and $p_2\in C_1$. In the
latter case, by Lemma~\ref{412}, $(C_1.C_2)_{p_i}\geq2$ for $i=1$,
$2$. Since $K_S\cdot C_1=K_S\cdot C_2=2$, we have $C_i^2=-4$, and so  $F'$ is as
in (ix) of Lemma~\ref{k3''}.

In the former case, note that $p_a(C_i)=1$,  $p_i$ is the unique
singular point of $C_i$.  If
 $C_1$ and  $C_2$ meet transversely in two points, say $p'$ and
 $p''$,
 then there are at least three $\ti \alpha$-fixed points, the inverse image
 of $p'$
$p''$ and $p_1$,  on $\ti C_1\simeq\mathbb{P}^1$, a contradiction.

 If
 $C_1$ and  $C_2$ meet tangentially  in one point, say $q$, then    $q$ is
 isolated $\sigma$-fixed of weight $\frac{1}{4}(1, a)$, where $a=1$ or
 $3$ (since $q\not\in H$), a contradiction by
Lemma~\ref{lem: singular'}.

\medskip

\noindent{\bf Case 4.4.  $t=2$, $K_S\cdot C_1=1$ and $K_S\cdot C_2=3$.} In
this case $H\cdot C_1=1$ and $H\cdot C_2=3$. If  $p_i\in C_1$ for some $i$, say
$i=1$, then $p_1$ is a smooth point of $C_1$ (since $H\cdot C_1=1$). Since
$p_1$ is a singular point of $F'$, we have $p_1\in C_2$. Then
$(C_1.C_2)_{p_1}\geq2$ by Lemma~\ref{412}. Since $K_S\cdot C_1=1$, this
implies that $C_1$ is a $(-3)$-curve, and hence $C_1\setminus\{p_1\}$
and  $C_2\setminus\{p_1\}$ meet transversely in one point, say $q$,
which is $\sigma$-fixed.

Since  $K_S\cdot C_2=3$ and $C_2^2=-3$, we have $p_a(C_2)=1$. This implies
that $p_2$ is the unique singular point of $C_2$. By Lemma~\ref{lem:
singular'}, $p_2$ is a cusp of $C_2$. Thus there are at least three
$\ti \sigma$-fixed points (the inverse image of $p_1$, $p_2$ and
$q$) on $\ti C_2\simeq\mathbb{P}^1$, a contradiction.

\medskip

\noindent{\bf Case 5. $u=1$.}

Since $H\cdot F'=4$,  we have $2\leq\op{mult}_pF'\leq4$.

If $t\geq3$   then $K_S\cdot C_1=1$. By Lemma~\ref{412}, we have
$(C_1.C_j)_p\geq2$ for $1<j\leq t$. This implies  $C_1^2\leq2-2t$, a
contradiction by the adjunction formula.

Now we assume $t=2$. If $\op{mult}_pC_1=1$ and $\op{mult}_pC_2=3$,
then $p_a(C_2)\geq3$, which is impossible.

 If $\op{mult}_pC_1=\op{mult}_pC_2=1$, then $H\cdot C_i=1$ for $i=1$,
$2$, and hence $H\cdot F'=2$, a contradiction since $H\cdot F'=4$.

 If
$\op{mult}_pC_1=\op{mult}_pC_2=2$, then $K_S\cdot C_i=H\cdot C_i=2$,
$p_a(C_i)\geq1$ for $i=1$, $2$,   and $(C_1.C_2)_p\geq4$. So
$C_i^2=-C_1\cdot C_2\leq-4$, a contradiction by the adjunction formula.
\qed\end{pf}

\begin{lem}\label{H irred}$H$ is irreducible.\end{lem}
\begin{pf}Otherwise, by (i) of Lemma~\ref{H-3},
 we have  $H=\Gamma_1+\Gamma_2$, where $\Gamma_i$ ($i=1, 2$) is
an irreducible curve with $\Gamma_iF=2$.

 If $\sigma(\Gamma_1)=\Gamma_2$ then $k_2=0$. Combining this
equality with Corollary~\ref{equ6.1}, by an elementary calculation
we have
\begin{align*} k_1=10\chi(\mc
O_S)-K_S^2.
\end{align*} Since $k_1\leq4$
(cf.~Corollary~\ref{c-i''}), we have $K_S^2\geq10\chi(\mc O_S)-4$, a
contradiction.

  Now we may assume that
$\Gamma_i$ is $\sigma$-invariant.

 For $i=1,2$, let $f_i:
\Gamma_i\to B$ be the restriction of $f$ to $\Gamma_i$, and $R_i$
the ramification divisor of $f_i$. By the Hurwitz formula, we have
\[\deg R_i=K_S\Gamma_i+\Gamma_i^2+4.\]
So \begin{align*} k_2&=\deg R_1+\deg R_2 =K_SH+H^2+8\\&=4\chi(\mc
O_S)+2H^2+8.
\end{align*}
 Combining this equality with Corollary~\ref{equ6.1}, by an elementary calculation we have
\begin{align*} k_1=-2H^2-4, \  \ \
k_3=-H^2-4.
\end{align*}

The two equalities above and the inequality  $k_1\leq4$
(cf.~Corollary~\ref{c-i''}) yield $H^2=-4$ and $k_1=4$.

Let $f^*b_i=2C_i$ ($i=1, \cdots ,4$) be multiple fibers of $f$  with
$C_i<V$ (cf.~Lemma~\ref{k-1'}). Since $k_1=4$, by
Corollary~\ref{c-i'}, note that $H\cap C_i$ consists of two points,
we have that $C_i$ is as in (ii) of Lemma~\ref{c-i} for all $i$.
Then we have that there are at least four $\ti \sigma$-fixed points
on the normalization $\tC_i$ of $C_i$: the inverse image of
$\Gamma_1\cap C_i$, $\Gamma_2\cap C_i$  and the node of $C_i$. This
is absurd since $\tC_i$ is an elliptic curve and $\ti\sigma$
is of order $4$. \qed\end{pf}

\subsection{Proof of Theorem~\ref{thm: g3c4}} Let $f: S\to B$, $\sigma$, $H$, $k_a$ be as in Notation~\ref{5.2}. Assume that $\chi(\mc O_S)>2$. We will get a contradiction. 
By Lemma~\ref{H irred}, $H$ is irreducible. Let $\oH=H/\sigma$
and $\pi \colon H\to \oH$ be the quotient map. Let $\varpi\colon \oH \to
B$ be the map induced by $f\rest H$, the restriction of $f$ to $H$.
 We have $f\rest H=\varpi\circ\pi$.  Let $R_\pi$ be  the ramification
divisor of $\pi$. By the Hurwitz formula, we have
\begin{equation}\label{k2}k_2=\deg R_\pi=K_SH+H^2+4-4\oH=4\chi(\mc
O_S)+2H^2+4-4\oH,\end{equation} where $\oH=g(\oH)$.
Combining this with Corollary~\ref{equ6.1}, we have
\begin{align} \label{k1}k_1=&-2H^2-2+2\oH,\\
k_3=&-H^2-2+2\oH.
\end{align}

For each $b\in B$, we set
\[
r(b)=
\begin{cases}
1 & \text{if $\varpi$ is branched at $b$;} \\
r(b)=0, & \text{otherwise.}
\end{cases}
\]
If $f^*b$ is a smooth fiber, we have $r(b)=0$.
By definition, $$\sum_{b\in B}r(b)=\deg R_\varpi=2\oH+2.$$

For each  $b\in B$, let $ e_f(b)=e(f^*b)-e(F)$, where  $F$ is
 a general fiber of
$f$.
 Recall
that $e_f(b)\geq0$ and the equality holds if and only if $f^*b$ is
smooth, and that $e_f(b)$ satisfies the following
equality(\cite[11.4, 11.5, p.~118]{BHPV04}):
\begin{align}\label{es16} e(S)=4(g-1)(g(B)-1)+ \sum_{b\in B}e_f(b).\end{align}

For each $b\in B$, $a=1, 2, 3$,  we denote by  $k_a(b)$    the
number of isolated $\sigma$-fixed
 points of weight $\frac{1}{4}(1, a)$ contained in $f^*b$. If $f^*b$ is a smooth fiber,
we have $k_a(b)=0$. By definition, $$\sum_{b\in B}k_a(b)=k_a.$$

For a singular fiber  $f^*b$, by Lemmas~\ref{c-i}, \ref{k3'} and
\ref{k3''}, we have that $e_f(b)$, $k_a(b)$ and $r(b)$ are given in
the following tables, where v(1) means  a fiber $f^*b$ being as in
(v) of Lemma~\ref{c-i} with $a=1$.
\begin{center}
\begin{tabular}{|c|c|c|c|c|c|c|c|c|}
\hline         $f^*b$    &  i  & ii & iii & iv
      & v(1)  & v(3)   &vi     \\
\hline $e_f(b)$   &  $2$ & $3$& 3 & $5$  &$6$ & $6$ & 3 \\
\hline $k_1(b)$ &  $0$ & $1$& 0 & $0$  &$1$ & $0$ &  0 \\
\hline $k_2(b)$ &  $2$ & $0$& 2 & $1$  &$1$ & $1$ &  2 \\
\hline $k_3(b)$ &  $0$ & $0$& 1 & $0$  &$0$ & $1$ &  1 \\
\hline $r(b)$  &  $0$ & $1$& 0 & $1$  &$1$ & $1$ &  0 \\
 \hline
\end{tabular}

\

\tablename  \ 1. $f^*b$    is as in Lemma~\ref{c-i}
 \end{center}
\

\

\begin{center}
\begin{tabular}{|c|c|c|c|c|c|c|c|c|c|c|c|c|c|}
\hline         $f^*b$       &  i  & ii & iii & iv
      & v  & vi  & vii &viii& ix & x &xi &xii &xiii    \\
\hline $e_f(b)$ &  $1$ & $2$& 3 & $3$  &$5$ & $6$ & 5&2& 4&6& 7&5& 8 \\
\hline $k_1(b)$ &  $0$ & $0$& 0 & $0$  &$0$ & $0$ & 0&0& 0&0 &0&0&0\\
\hline $k_2(b)$ &  $0$ & $0$& 0 & $1$  &$1$ & $2$ & 1&0& 1&2 &2&1&2\\
\hline $k_3(b)$ &  $1$ & $0$& 1 & $0$  &$0$ & $0$ & 0&2& 1&0 &1&2&2\\
\hline $r(b)$   &  $0$ & $1$& 1 & $0$  &$1$ & $0$ & 1&0& 0&0&0&0& 0\\
 \hline
\end{tabular}

\

\tablename  \ 2. $f^*b$    is as in Lemmas~\ref{k3'} and
\ref{k3''}
 \end{center}

Let $x_1$ (resp.~$x_2, \dots, x_7$)  be the number of fibers as in
(i) (resp.~(ii), (iii), (iv), (v(1)), (v(3)),  (vi)) of
Lemma~\ref{c-i}. Let $y_1$ (resp.~$y_2, \dots, y_{13}$)  be the
number of fibers as in (i) (resp.~(ii), $\dots$, (xiii)) of
Lemmas~\ref{k3'} and \ref{k3''}. Table 1 and Table 2 yield the
following five equations.
\begin{equation}\label{equef}\begin{split}
&2x_1+3x_2+3x_3+5x_4+6x_5+6x_6+3x_7+y_1+2y_2+3y_3\\+&3y_4+5y_5+6y_6+5y_7
2y_8+ 4y_9+6y_{10}+7y_{11}+5y_{12}+8y_{13}
=e(S)+8.\end{split}\end{equation}
\begin{equation}\label{equk1} x_2+x_5=k_1.\end{equation}
\begin{equation}\label{equk2}\begin{split}
&2x_1+2x_3+x_4+x_5+x_6+2x_7+y_4+y_5\\
+&2y_6+y_7+y_9+2y_{10}+2y_{11}+y_{12}+2y_{13}=k_2.\end{split}\end{equation}
\begin{equation}\label{equk3}
x_3+x_6+x_7+y_1+y_3+2y_8+y_9+y_{11}+2y_{12}+2y_{13}=k_3.\end{equation}
\begin{equation}\label{equbarh}
x_2+x_4+x_5+x_6+y_2+y_3+y_5+y_7=\sum_{b\in B}r(b)=2\bar
h+2.\end{equation} By Corollary~\ref{c-i'}, we have
\begin{equation}\label{equv} x_2+x_4+x_5+x_6\geq4.\end{equation}

  Subtracting (\ref{equk2}) and $2\times$(\ref{equbarh})  from (\ref{equef}), we
get
\begin{equation}\label{equef''}\begin{split}
&x_2+x_3+2x_4+3x_5+3x_6+x_7+y_1+y_3+2y_4+2y_5\\
+&4y_6+2y_7+ 2y_8
+3y_9+4y_{10}+5y_{11}+4y_{12}+6y_{13}=-3H^2.\end{split}\end{equation}

By  (\ref{equv}) and (\ref{equef''}) we have  $H^2\leq-2$.

By (\ref{equbarh}) and (\ref{equv}), we have $\oH\geq1$, and
$\oH=1$ if and only if $x_2+x_4+x_5+x_6=4$ and
$y_2=y_3=y_5=y_7=0$.

Combining the two inequalities above with (\ref{k1}),
 note that
 $k_1\leq4$
(cf.~Corollary~\ref{c-i''}), we obtain $H^2=-2$, $\oH=1$,  and
hence  $k_1=4$, $k_3=2$.

Now subtracting  (\ref{equk1}) and (\ref{equk3}) from
(\ref{equef''}), we have
\[2x_4+2x_5+2x_6+2y_4+2y_5+4y_6+ +2y_7+2y_9+4y_{10}+4y_{11}+2y_{12}+4y_{13}=-3H^2-k_1-k_3=0.\] This implies
\begin{equation}\label{equef''2}
x_4=x_5=x_6=y_4=y_5=y_6=y_7=y_9=y_{10}=y_{11}=y_{12}=y_{13}=0.
\end{equation}
Combining (\ref{equef''2}) with (\ref{equk1}) and (\ref{equk2}), we
have
\begin{equation}\label{equk21}
 x_2=4,\ \ \ \  x_1+x_3+x_7=2\chi(\mc O_S)-2.   \end{equation}

Now we distinguish two cases to get a contradiction.

\medskip

\noindent{\bf Case 1.  $f$ is hyperelliptic.} In this case, since $x_2>0$, $f$ has
a singular fiber $2C$ as in Lemma~\ref{c-i} (ii).  Let $\tau$ be the
hyperelliptic involution of $f$, and $G$ the group generated by
$\tau$ and $\sigma$. Then $G$ acts faithfully on $C$ (since $C$ is
singular). Let $F$ be a general fiber of $f$. Since
$\alpha(=\sigma^2)$ is a bi-elliptic involution of $f$, and
$\gamma:=\tau\circ\alpha$ induces identity on
$H^0(\omega_F)_\alpha^{-1}$, we have $g(F/\gamma)=2$. So $\gamma$
has no horizontal fixed curves. This implies the node $p$ of $C$ is
an isolated $\gamma$-fixed point, and so $\gamma$ fixes the local
branches of $C$ at $p$. Then $G$ fixes the local branches of $C$ at
$p$ (since $p$ is an isolated $\sigma$-fixed
 point of weight $\frac{1}{4}(1, 1)$ and $G$ is generated by $\gamma,
 \sigma$). This implies $G$  is cyclic,
a contradiction since $G$ is not cyclic.

\medskip

\noindent{\bf Case 2.  $f$ is  nonhyperelliptic.}  In this case, the
natural morphism of sheaves
 $$r\colon S^2(f_*\omega_{S/B})\to f_*\omega_{S/B}^2 $$
is generically surjective, and
 $\textrm{Ker}\  r=0$.
 Let $\mathcal M=\textrm{Coker}\  r$.
 Then $\mathcal M=\bigoplus_{b\in B}\mathcal{M}_b$,
 where $\mathcal{M}_b$ is the stalk of $\mathcal{M}$
at $b\in B$, which is  an $\mathcal{O}_{B,b}$-module of finite
length. Let $H(S/B, b)=\textrm{length} \ \mathcal{M}_b$. Using the
Riemann-Roch theorem on $S$ and the Leray spectral sequence, we
obtain: (Reid, cf.~\cite{Reid})
\begin{equation}\label{reid}K_{S/B}^2=3\chi_f+\sum_{b\in B}H(S/B, b).\end{equation}

Note that if $F':=f^*b$ is a multiple fiber, then the canonical ring
$R(F', K_{F'})$ needs $3$ generators in degree $2$(cf.~\cite[Theorem
III 5.1]{ML88}), so that $H(S/B, b)\geq3$. Hence we have
\begin{align*}K_S^2&\geq3\chi(\mc O_S)-10+3\sum_{i=1}^7 x_i\\ &=
3\chi(\mc O_S)-10+3(2\chi(\mc O_S)-2+4) \ \ \textrm{( by
(\ref{equk21}))}\\&=9\chi(\mc O_S)-4.\end{align*} Combining this
with $K_S^2=8\chi(\mc O_S)-2$ (cf.~(i) of Corollary~\ref{equ6.1}),
we get $\chi(\mc O_S)\leq2$, a contradiction.

This finishes the proof of  Theorem~\ref{thm: g3c4}.

\section{Regular surfaces of general type with \texorpdfstring{$|\Aut_\QQ(S)|=4$}{}}\label{sec: G4}

Let $S$ be a  minimal  smooth projective regular surface of general type with $|\Aut_\QQ(S)|=4$ and $p_g(S)\geq 2$. For simplicity of notation we write $G=\au$. Then quotient surface $S/G$ has at most quotient singulaities with weights $\frac{1}{4}(1,a)$ with $1\leq a\leq 3$ and $\frac{1}{2}(1,1)$. Note that both weights $\frac{1}{4}(1,a)$ and $\frac{1}{2}(1,1)$ are $A_1$-singularity. Let $\mu\colon\tT\rightarrow S/G$ be the minimal resolution, and $\eta\colon\tT\rightarrow T$ the birational morphism onto the minimal model. Let $\tS$ be the minimal resolution of the singularities of $S\times_{S/G} \tT$, and $\rho\colon\tS\rightarrow S$ the induced birational map. Then the action of $G$ lifts to $\tS$. Since $G$ acts trivially on $H^0(S, \omega_S)$, we have natural identification of the vector spaces $H^0(\tS, K_{\tS})$, $H^0(S, K_S)$, $H^0(\tT, K_{\tT})$, and $H^0(T, K_{T})$ via appropriate pull-backs. We have the following commutative diagram:
\begin{equation}\label{diag: can map}
 \begin{tikzcd}
\tS \ar[rr, "\tilde\pi"]\ar[d, "\rho"'] & &\tT\ar[d, "\mu"] \ar[rrd, dashed, "\phi_{\tT}"]\arrow[rr, "\eta"] && T\arrow[d, dashed, "\phi_T"]\\
  S \ar[rr,"\pi"] \arrow[rrrr, bend right,dashed, "\phi_S"]& & S/G\arrow [rr, dashed, "\phi_{S/G}"] && \PP^{p_g-1}
 \end{tikzcd} 
\end{equation}
where $p_g$ is the common geometric genus of the surfaces $\tS, S, \tT$ and $T$. Since $p_g(T)=p_g(S)\geq 2$, we have $\kappa(T)\in \{1,2\}$. We proceed according to the value of $\kappa(T)$.

 \subsection{The case \texorpdfstring{$\kappa(T)=1$}{}}
In this subsection, we assume that $\kappa(T)=1$. Let $f_{\tT}\colon \tT\rightarrow B$ and $f_T\colon T\rightarrow B$ be the Iitaka fibrations, defined by the pluri-canonical system $|mK_{\tT}|$ and $|mK_T|$ for $m$ sufficiently large and divisible. Then $f_{\tT}$ and $f_T$ are elliptic fibrations such that $f_{\tT} = f_T\circ\eta$. Since $q(\tT) = q(S) = 0$, we have $B\cong \PP^1$. Suppose that $p_g\geq 4$. Then $K_{\tT}\cdot \tilde C\geq K_{T}\cdot \eta_*\tilde C\geq 4$ for any horizontal curve $C$. Since $S/G$ has at most $\frac{1}{4}(1,n)$-singularities, the exceptional locus $\Exc(\eta)$ of $\eta$ has at most $(-2)$ or $(-4)$-curves. It follows that $\Exc(\eta)$ is vertical with respect to $\eta$, and hence $f_{\tT}$ induces a fibration $f_{S/G}\colon S/G\rightarrow B$. 
\begin{lem}\label{lem: fib k=1}
Let $f=f_{S/G}\circ\pi\colon S\rightarrow B$. If $p_g\geq 4$, then the following holds.
\begin{enumerate}[leftmargin=*]
 \item
$f$ has connected fibers.
\item 
The canonical map $\phi_S = \varphi_\gb\circ f$, where $\varphi_\gb\colon B\rightarrow \PP^{p_g-1}$ is the embedding induced by a complete linear system $|\gb|$ of degree $p_g-1$. In particular, the rational map $\phi_S$ is a morphism.
\item The fibers of $f$  are invariant under the action of $G$. In other words, $f\circ\sigma=f$  for each $\sigma\in G$.
\end{enumerate}
\end{lem}
\begin{proof}
 (i) Let $f\colon S\xrightarrow{f_1} \tilde B \xrightarrow{j} B$ be the Stein factorization of $f$, where $j$  is a finite morphism between two smooth 
 curves. We have $g(\tilde B)=0$ since $q(\tS) = q(S) = 0$.  
  By the canonical bundle formula for elliptic fibrations (\cite[V, Thm 12.1]{BHPV04}) we have
\begin{equation}\label{eq: KT}
 |K_{T}| = f_T^* |\gb| +Z_T
\end{equation}
where $\gb=\omega_B\otimes(h_*\mc O_T)^\vee$ is a divisor on $B$ of degree $p_g-1$, and $Z_T$ is a vertical effective divisor. It follows that $|K_{\tT}|=f_{\tT}^* |\gb| +Z_{\tT}$ and, in turn, $|K_{S/G}|=f_{S/G}^* |\gb| +Z_{S/G}$, where $Z_{\tT}$ and $Z_{S/G}$ are vertical divisors with respect to the fibrations $f_{\tT}$ and $f_{S/G}$ respectively.

Suppose on the contrary that $\deg j\geq 2$. Then $\deg j^*\gb = 2 \deg \gb>\deg \gb>0$, and hence
\[
h^0(T, K_T) = h^0(B, \gb) < h^0(\tilde B, j^*\gb) \leq h^0(S, K_S)
\]
where the last inequality is because $f_1^*j^*\gb = \pi^*h^*\gb$ can be identified with a sub-divisor of $K_{S}$. This contradicts the fact that $\pi^*\colon H^0(T, K_T) \rightarrow H^0(S, K_S)$ is an isomorphism. Hence $\deg j = 1$ and $j$ is an isomorphism. In turn, $f$ has connected fibers.

(ii) We have isomorphisms of linear sysmtems 
\[
f^*|\gb| = \pi^*h^*|\gb|\cong  \pi^* |K_T| \cong |K_S|
\]
It follows that $\phi_S=\varphi_\gb\circ f$.

(iii) is clear, since $\phi_S$ is equivariant with respect to the action of $G$ and $G$ (pointwise) fixes $\varphi_\gb(B)\subset\PP^{p_g-1}$.
\end{proof}

We will denote $g=g(F)$, where $F$ is a general fiber of $f$.
\begin{prop}[{\cite[Proposition~2.1]{Be79}}]
If $\chi\geq 21$, then $2\leq g \leq 5$.
\end{prop}

By \eqref{eq: KT} we know
\[
 |K_S| = \pi^*|K_{{S/G}} | + R = f^*|\gb| + R + \pi^*Z_{S/G}
\]
where $R$ is the ramification divisor of $\pi$. It follows that $M=f^*\gb$ and $Z=R + \pi^*Z_{S/G}$ are the moving part and the fixed part of  $|K_S|$ respectively. We write $Z=H+V$, where $H$ (resp.~$V$) is the horizontal part (resp.~vertical part) with respect to $f$. Then $H\leq R$ and $\pi^*Z_{S/G}\leq V$. Let $H=\sum_{1\leq i\leq t}n_iH_i$ be the irreducible decomposition of $H$. Then $n_i+1 = |G_{H_i}|\in \{2,4\}$; thus $n_i=1$ or $3$. We may assume that $n_1\geq n_2\geq \dots \geq n_t$.

We have $M\equiv (p_g-1)F = (\chi-2)F$, where $\chi=\chi(\sO_S)$. For a general fiber $F$ of $f\colon S\rightarrow B$, one has $M\cdot F =0$, and
\[
F\cdot H = F\cdot R = F\cdot K_S = 2g-2
\]
where $g=g(F)$ is the genus of $F$. Therefore,
\begin{equation}\label{eq: mh} 
H\cdot M=K_S\cdot M= (\chi-2) K_S\cdot F =(2g-2)(\chi-2),
\end{equation}
and hence
\begin{equation}\label{eq: K^2}
K_S^2=K_S\cdot (M+Z)=(2g-2)(\chi-2)+K_S\cdot Z.
\end{equation}
\begin{lem}\label{lem: bound KZ}
 We have 
 \[
 K_S\cdot Z\geq K_S\cdot H \geq \frac{2g-2}{n_1+1}(\chi(\sO_S)-2) + \sum_{i=1}^t \frac{n_i^2}{n_1+1}(2g(H_i)-2)\]
\end{lem}
\begin{proof}
Since $Z\geq H$ and $K_S$ is nef, the first inequality is clear. 

Note that 
$(n_1K_S+H+V)\cdot H_i \geq (n_1K_S+n_iH_i)\cdot H_i\geq n_i(2g(H_i)-2)$ for $1\leq i\leq t$. So we have 
\begin{equation}\label{eq: bound KH}
(n_1+1)K_S\cdot H - M\cdot H=(n_1K_S+H+V)\cdot H \geq \sum_i^t n_i^2(2g(H_i)-2).
\end{equation}
Now the second inequality of the lemma follows from \eqref{eq: bound KH} and \eqref{eq: mh}.
\end{proof}

\begin{lem}\label{lem: g neq 5}
$g=5$ does not occur.
\end{lem}
\begin{proof}
Suppose that $g=5$. Since $n_1+1\leq 4$, we have $\frac{2g-2}{n_1+1} =\frac{8}{n_1+1}\geq 2$ and $\frac{n_i^2}{n_1+1}\geq\frac{n_i}{4}$.  By Lemma~\ref{lem: bound KZ}, we have
\[
K_S\cdot Z \geq 2(\chi-2) + \sum_{1\leq i\leq t}\frac{n_i}{4}(2g(H_i)-2)  \geq 2(\chi-2) + \frac{1}{2}\sum_{1\leq i\leq t}n_i = 2\chi
\]
where the last equality is implied by the Riemann--Hurwitz formula for $F\rightarrow F/G$.
By \eqref{eq: K^2}, we obtain
\[
K_S^2 = 8(\chi-2) + K_S\cdot Z\geq 8(\chi-2) + 2\chi > 9\chi
\]
provided that $\chi\geq 21$, which is a contradiction to the Bogomolov--Miyaoka--Yau inequality.
\end{proof}

\begin{lem}\label{lem: g neq 4}
$g=4$ does not occur. 
\end{lem}
\begin{proof}
Suppose on the contrary that $g=4$. Then 
\begin{equation}\label{eq: g4 H}
\sum_{1\leq i\leq t} n_i H_i\cdot F = H\cdot F = 2g-2 =6.
\end{equation}
Since $n_i\in \{1,3\}$, there are the following possibilities by \eqref{eq: g4 H}:
\begin{enumerate}
    \item[(i)] $n_i=1$ for each $1\leq i\leq t$, equivalently, $H$ is reduced, each irreducible component of $H$ is $\sigma$-fixed for some involution $\sigma\in G$. 
    \item[(ii)] $n_1=3$, and $H=3H_\red$, where $H_\red$ is the reduced divisor of $H$, $G\cong \ZZ/4\ZZ$, and $H_\red$ is $G$-fixed.
\end{enumerate}
We first exclude the case (i). In this case, we have 
\begin{equation}\label{eq: K2 vs 8chi, g=4}
K_S^2\leq 8\chi + u \sum_{i=1}^t \Gamma_i^2
\end{equation}
where $u=\frac{1}{3}$ (resp.~$u=1$) if $G\cong(\ZZ/2\ZZ)^2$ (resp.~$G\cong \ZZ/4\ZZ$). Indeed, if $G\cong \ZZ/4\ZZ$, let $\sigma\in G$ be the unique element of order $2$, then $H$ is $\sigma$-fixed. By Lemma~\ref{lem: sign inv}, we obtain \eqref{eq: K2 vs 8chi, g=4} with $u=1$.

   If $G\cong(\ZZ/2\ZZ)^2$, write $G=\left<\op{id}_S, \sigma_1, \sigma_2, \sigma_3 \right>$.
 For $j=1, 2, 3$, let $\Theta_j$ be the divisorial part of the fixed
locus of $\sigma_j$, and $\Theta_{j, h}$ the horizonal part of $\Theta_j$ with respect to $f$.
 Since the stabilizer of each curve under the
action of $G$ is a cyclic group, we have that for $i\not=j$,
$\Theta_i$ and $\Theta_j$ have no common components.

 Since each irreducible component of $H$ is $\sigma_i$-fixed   for some
$i$,
 we have $\Theta_{1, h}+\Theta_{2, h}+\Theta_{3, h}=H$.
By
Lemma~\ref{lem: sign inv} we have $K_S^2 \leq 8\chi(\mc O_S)+ \Theta_{j,h}^2$, and hence we get \eqref{eq: K2 vs 8chi, g=4} with $u=\frac{1}{3}$.

On the other hand,
since $H$ is reduced in case (i),
by Lemma~\ref{lem: bound KZ}, we have
\[K_S H \geq 3(\chi(\mc O_S)-2)+
\sum_{i=1}^t (g(\Gamma_i)-1).\]
So by \eqref{eq: mh} we have
\begin{align}\label{eq: bound K2, g=4 case 1}
K_S^2\geq K_S\cdot M+K_SH\geq9(\chi(\mc O_S)-2)+
\sum_{i=1}^t (g(\Gamma_i)-1).\end{align}
Combining \eqref{eq: bound K2, g=4 case 1} and \eqref{eq: K2 vs 8chi, g=4}, and noting that $K_S\Gamma_i+\Gamma_i^2=2g(\Gamma_i)-2$ by the adjunction formula  and  $(K_S-\Gamma_i)\Gamma_i\geq2$ since $K_S$ is $2$-connected,
we get
\begin{equation*}
\begin{split}
\chi(\mc O_S)&\leq\begin{cases}&18 +\frac{1}{3}\sum_{i=1}^t(1-g(\Gamma_i) -K_S\Gamma_i),\ \ \textrm{if $u=\frac{1}{3}$}\\& 18 +\frac{1}{2}\sum_{i=1}^t(\Gamma_i^2 -K_S\Gamma_i), \ \ \ \ \    \textrm{if $u=1$}\end{cases}\\ &<18,
\end{split}
\end{equation*}
which
contradicts the assumption $\chi(\mc O_S)\geq21$.

Now we may assume the pair ($H$, $G$) is as the case (ii). Then $G\cong\ZZ/4\ZZ$, and
\[
 K_SH=3(M+3H_\red+V)H_\red\geq 3MH_\red+9H_\red^2 = 6(\chi(\mo_S)-2)+9H_\red^2.\]
So by \eqref{eq: K^2} we obtain
\begin{align}\label{eq: g4'}
K_S^2\geq 6(\chi(\mo_S)-2)+K_SH
\geq 12(\chi(\mo_S)-2)+9H_\red^2.
\end{align}

Combining \eqref{eq: g4'} with \eqref{eq: K2 vs 8chi, g=4}, we get
\begin{align}\label{eq: g4''}K_S^2\leq\frac{15}{2}\chi (\mo_S)+3.\end{align}

On the other hand, let $X$ be the minimal model of $S/\alpha$, and $f_X\colon X\rightarrow B$ the induced fibration. There are exactly two $\alpha$-fixed points, namely $H_1\cap F$, and the Riemann--Hurwitz formula implies that $g(F/\langle\alpha\rangle)=2$. Thus $f_{X*}\omega_X$ is a rank two vector bundle of degree $\chi$ over $B$, and the associated projective bundle $\PP:=\PP(f_{X*}\omega_X)$ is a ruled surface over $B$. Let $e$ be the largest number such that there is a section $\Gamma$ of $\PP\rightarrow B$ with $\Gamma^2=-e$. We have (cf.~\cite[page 7]{X85a}):
\begin{equation}\label{eq: e}
e = \max\{2\deg \sL - \deg f_{X*}\omega_X \mid \sL\subset\deg f_{X*} \text{ is a sub-line-bundle}\}.
\end{equation}
Since $f_X\colon X\rightarrow B$ is the composition of the induced rational map $X\dashrightarrow T$ with $f_T\colon T\rightarrow B$, the rank two vector bundle $f_{X*}\omega_X$ contains the line bundle $f_{T*}\omega_T$. Note that $\deg f_{T*}\omega_T =\chi$, so we have by \eqref{eq: e}
\[
d\geq 2 \deg f_{T*}\omega_T - f_{X*}\omega_X = \chi -1
\]
Therefore, by \cite[Theoremm 2.2, (ii)]{X85a}, we have
\[
K_X^2 \geq \chi + 3e-7 \geq 4\chi -10.
\]
By Lemma~\ref{lem: bound KZ}, we have $K_S\cdot H\geq \frac{3}{2}(\chi-2) + \frac{9}{4}(2g(H_1)-2) \geq \frac{3}{2}\chi - \frac{15}{2}$. We obtain
\[
K_S^2 \geq 2 K_X^2 + K_S\cdot H \geq 8\chi -20 + \frac{3}{2}\chi - \frac{15}{2} = \frac{19}{2}\chi -\frac{55}{2}
\]
which is a contradiction to \eqref{eq: g4''}, provided that $\chi\geq 21$.
\end{proof}

\begin{thm}\label{thm: can pencil g3}
Suppose that $\chi\geq 21$ and $\kappa(\tT)=1$. Then the canonical map is composed with a fibration $f\colon S\rightarrow B$ of genus $3$.
\end{thm}
\begin{proof}
 Lemmas~\ref{lem: g neq 5} and \ref{lem: g neq 4}, the Iitaka fibration $f_T\colon T\rightarrow B$ pulls back to a fibration $f\colon S\rightarrow B$ of genus $g=2$ or 3. The case $g=2$ is impossible by \cite{Cai06}. 
\end{proof}

\begin{thm}\label{thm: k=1}
Let $S$ be a regular surface of general type with $|\Aut_\QQ(S)|=4$ and $\chi(\sO_S)\geq 13$. If $K_S$ is ample and $\kappa(\tT)=1$ then $\Aut_\QQ(S)\cong (\ZZ/2\ZZ)^2$ and  $S$ is isogenous to a product of curves.
\end{thm}
\begin{proof}
By Lemma~\ref{lem: fib k=1} and Theorem~\ref{thm: can pencil g3}, there is a genus 3 fibration $f\colon S\rightarrow B$ such that $\Aut_\QQ(S)$ preserves each fiber. By Theorem~\ref{thm: g3c4}, $\Aut_\QQ(S)\cong (\ZZ/2\ZZ)^2$. By Theorem~\ref{thm: g3}, $S$ is isogenous to a product of curves.
\end{proof}

\subsection{The case \texorpdfstring{$\kappa(T)=2$}{}}
In this subsection, we assume that $\kappa(\tT)=2$, where we recall that $\tT$ is the minimal resolution of $S/G$.

As before, we use the canonical map to describe the surfaces. Let $\Sigma:=\op{Im}(\phi_S) = \op{Im}(\phi_T)\subset\PP^{p_g-1}$ be the canonical image, and $\varphi_S\colon S\dashrightarrow\Sigma$ and $\varphi_{S/G}\colon S/G\dashrightarrow \Sigma$ the dominant rational maps induced by the canonical maps. We have $\varphi_S = \varphi_{S/G}\circ\pi$; see \eqref{diag: can map}. 

\begin{prop}\label{prop: deg can map}
If  $\chi\geq14$, then $\Sigma$ is a rational surface, and $\deg\varphi_S =4\deg {\varphi_{S/G}}=8$.
\end{prop}
\begin{pf}
If $\Sigma$ is a curve, then $\vol(K_T)\geq3\chi-6$ (cf.~\cite{Hor78}), and hence
\[
 K_S^2 \geq 4 \vol(K_T) \geq 4(3\chi(\mc O_T) - 6) = 12\chi - 24
 >9\chi
 \]
 provided $\chi\geq14$. This contradicts
the Bogomolov--Miyaoka--Yau
 inequality. Therefore, $\Sigma$ is a surface, and $\varphi_T$ and $\varphi$ are generically finite. We have $\deg\varphi_S= (\deg \pi)(\deg \varphi_T) = 4\deg \varphi_T$.

If $p_g(\Sigma)=p_g(S)$, then $\Sigma$ is a canonical surface, which
satisfied the Castelnuovo's inequality $\deg \Sigma\geq3\chi(\mc
O_S)-10$ (cf.~\cite[5.6]{Be79}). We have
$$K_S^2\geq|G|\deg\varphi_T\deg \Sigma\geq4(3\chi-10 )
>9\chi$$ provided $\chi\geq14$, contradicting the 
Bogomolov--Miyaoka--Yau
 inequality again. So we can assume
$p_g(\Sigma)=0$. Then $\deg \varphi_T\geq2$.

If $\deg \varphi_T\geq3$, then  $K_S^2\geq|G|\deg\varphi_T\deg
\Sigma\geq12(\chi-3 ) >9\chi$ provided $\chi(\mc
O_S)>12$, a contradiction. So we have  $\deg \varphi_T=2$ and hence
$\deg \varphi_T=8$.

It follows from the Bogomolov--Miyaoka--Yau inequality that 
\[
\deg\Sigma\leq \frac{1}{\deg\varphi_S} K_S^2 = \frac{1}{8} K_S^2\leq \frac{9}{8}\chi
\]
Thus $\Sigma$ is a rational surface by \cite[Lemme~1.4]{Be79}, noting that the irregularity $q(\Sigma) = q(S)=0$.
\qed
\end{pf}

\begin{prop}\label{prop: K^2 leq 8X}
 If  $\chi\geq22$ then  $K_S^2\leq 8\chi$.
\end{prop}
\begin{pf} Let  $C_1,\ \cdots,\ C_m$ ($m\geq0$) be the $\sigma$-fixed curves. By Lemma~\ref{lem: sign inv}, we have 
\begin{equation}\label{eq: bound K^2 above}
 K_S^2=8\chi+\sum_{i=1}^m C_i^2.
\end{equation}
Thus it suffices prove that $C_i^2\leq 0$ for any $1\leq i\leq m$. Suppose on the contrary that one of the $C_i$'s, say $C_1$, satisfies $a:=C_1^2>0$. Note that the $C_i$'s are pairwise disjoint and contained in the fixed part $Z$ of $|K_S|$. By the Hodge Index Theorem, we have $C_i^2<0$ for any $i\neq 1$, and
\begin{equation}\label{eq: bound K^2 below}
(M\cdot C_1)^2\geq M^2 C_1^2= aM^2\geq 8a\deg\Sigma \geq 8a(\chi -3)
\end{equation}
where $M=K_S-Z$ is the mobile part of $|K_S|$, and we used the general fact that $\deg\Sigma\geq p_g-2=\chi -3$ (\cite{EH87}). Since $K_S$ and $M+C_1$ are both nef and $K_S-(M+C_1)= Z-C_1\geq 0$, we have
\begin{align*}K_S^2 \geq (M+C_1)^2 = M^2+2M\cdot C_1+C_1^2 &\geq 8(\chi-3)+2\sqrt{a(8\chi-24)} +a.\end{align*}
On the other hand, 
  It follows from \eqref{eq: bound K^2 below} and \eqref{eq: bound K^2 above} that $-24+2\sqrt{a(8\chi-24)} \leq 0$ and a simple computation gives $\chi\leq 21$, which contradicts the assumption on $\chi$.
\qed\end{pf}


Now we apply Xiao's results \cite{X86, X87} to obtain fibration structures on $T$ and $S$. 

 \begin{prop}\label{prop: g5} If  $\chi\geq56$ then there is a fibration $f\colon S\rightarrow \PP^1$ of genus $ g=5$ such that the fibers of
are preserved under the action of $G$.
\end{prop}
\begin{pf} 
We will first find a genus two fibration on the quotient surface $T$: By Proposition~\ref{prop: deg can map}, the map $\varphi_T\colon T\rightarrow \Sigma$ is generically finite of degree 2. By pulling back a pencil of rational curves on the rational surface $\Sigma$, we infer the existence of a pencil of hyperelliptic curves on $T$.  By Proposition~\ref{prop: K^2 leq 8X}, we have $\vol(K_T)\leq \frac{1}{4}K_S^2 \leq 2\chi$. Now, under the assumption that $\chi\geq 56$, we can apply \cite[Theorem~1]{X87} to obtain a unique genus 2 fibration $h\colon T\rightarrow B\cong \PP^1$. Since $\varphi_T$ is generically finite of degree 2, it necessarily separate the fibers of $h$ and sends them to lines in $\Sigma\subset\PP^{p_g-1}$.

Let $f=h\circ\pi$, and let $f\colon S\xrightarrow{f_1}\tilde B\xrightarrow{j} B$ be the Stein factorization. We will show that $j$ is an isomorphism, and hence $f$ has connected fibers: Suppose on the contrary that $\deg j\geq 2$. Then, since $\varphi_S$ factors through $\pi\colon S\rightarrow T$, it does not separate the elements of the linear system $|f^*\sO_B(1)|$, and this contradicts \cite[Lemma 2, ii)]{X86}.

Since the fibers of $f$ are the inverse image of curves on the quotient surface $T=S/G$, they are preserved by $G$. 

It remains to shows that $g(F)=5$. Note that $g(F/G)=2$. By applying the Riemann--Hurwitz formula to $F\rightarrow F/G$, one sees easily that $g(F)\geq 5$. On the other hand, setting $g=g(F)$, we have
\[
\frac{2g-2}{2g-5}(g\chi - 7g + 20) \leq K_S^2 \leq 8\chi
\]
where the first inequality is by \cite[Lemma~3]{X86} and the second one is by Proposition~\ref{prop: K^2 leq 8X}. Under the assumption that $\chi \geq 56$, one obtains $g  = 5$ by a straightforward computation.\qed
\end{pf}

\begin{thm}\label{thm: k=2}
Let $S$ be a regular surface of general type with $|\Aut_\QQ(S)|=4$
and $\chi(\sO_S)\geq 56$. If $K_S$ is ample  and $\kappa(\tT)=2$ then $\Aut_\QQ(S)\cong
(\ZZ/2\ZZ)^2$ and  $S$ is isogenous to a product of curves.
\end{thm}
\begin{proof}
By Propositions~\ref{prop: g5}, there is a genus 5 fibration
$f\colon S\rightarrow B$ such that $\Aut_\QQ(S)$ preserves each
fiber. If $\Aut_\QQ(S)\cong (\ZZ/2\ZZ)^2$, by Theorem~\ref{thm: g5},
$S$ is isogenous to a product of curves.

Form now on we assume that $\Aut_\QQ(S)\cong \ZZ/4\ZZ$. We will
show that this case  does not occur.

Let  $M$ and  $Z$ be the moving  part and   the fixed part of
$\vert K_S\vert$,  respectively.   We have
$\phi_S^*\mathcal{O}_{\PP^m}(1)=M$, where $m=p_g(S)-1$. Since
$\phi_S$ separates fibers of $f$ and $\phi_S(F)\subset\PP^m$ is a
line (by the proof of Proposition~\ref{prop: g5}), we have  $\deg
\phi_S=\deg\phi_S\rest F=M\cdot F$. So $ZF=K_S\cdot F-M\cdot F=8-\deg\phi_S=0$ by
Proposition~\ref{prop: deg can map}. Thus $Z$ is vertical with
respective to  $f$. Consequently, we have
\begin{enumerate}
 \item[(a)] $Z^2\leq0$;
\item[(b)] $\Theta^2\leq0$ for each irreducible component $\Theta<Z$.
\end{enumerate}

Let $\rho:\tilde{S}\to S$ be the blowing-ups of a number of
(ordinary or infinitely near) base points $x_1,\ \cdots,\ x_t $ (
$t\geq0$) of $|M|$, and $r_i$ the multiplicity of $|M|$ at $x_i$.
Then the moving part $|\tilde M|$ of $|\rho^*M|$ is base-point-free,
and $\varphi_{\tilde M}=\varphi_M\circ\rho$. Since $\deg
\varphi_M=8$ by Proposition~\ref{prop: K^2 leq 8X} and $\deg \Sigma\geq p_g(S)-2$, we have
\begin{equation}\label{1}\begin{split}
    K_S^2&=
MZ+K_SZ +\sum_{i=1}^t r_i^2+\tilde M^2\\
 &\geq MZ+K_SZ +\sum_{i=1}^t r_i^2+8(\chi (\mc O_S)-3).
 \end{split}\end{equation}

Let $\sigma$ be a generator of $\Aut_\QQ(S)$, and $\alpha=\sigma^2$.
Now we show that
\begin{align}\label{k1+k2}
k_1+k_2\leq 24,
\end{align}
 where
$k_a$ is the number of isolated $\sigma$-fixed points  of weight
$\frac{1}{4}(1, a)$.

Indeed, for each irreducible component $\Theta<Z$, we have
$\Theta^2\leq0$ by (b), and hence there are at most $2K_S\cdot \Theta$
isolated $\sigma$-fixed points on $\Theta$ by Lemma~\ref{k1k2}. So
 by Lemma~\ref{lem: fix vs base}, we have
\begin{equation}\label{equ:1}
    k_1+k_2\leq  2K_SZ+t.\end{equation}
    
 By (a), we have $ K_SZ\leq MZ$. So we have
\begin{align*}2K_SZ+t&\leq MZ+ K_SZ +\sum_{i=1}^t r_i^2\\
&\leq K_S^2
 - 8\chi (\mc O_S)+24 \ \ \ \ \ \ \textrm{(by (\ref{1}))}\\
&\leq24. \ \ \ \ \ \ \textrm{(by  Prop.~\ref{prop: K^2 leq 8X})}
\end{align*}
Combining this inequality with (\ref{equ:1}), we obtain (\ref{k1+k2}).

 Let $D_1,\ \dots,\ D_m$ ($m\geq 0$) be
$\sigma$-fixed curves, and $D_{m+1},\ \dots,\ D_n$ ($n\geq m$) be
$\alpha$-fixed but not $\sigma$-fixed curves.
 We have
 \begin{align*} 2\sum_{i=1}^n D_i^2-4\chi(\mc O_S)&= 2K_S^2-20\chi
(\mc O_S) \ \ \ \  \textrm{(by Lemma~\ref{lem: sign inv})}\\ &\geq3\sum_{i=1}^m D_i^2-2(k_1+k_2)\ \ \ \ \ \textrm{(by (\ref{eq2}))}\\
&\geq 3\sum_{i=1}^n D_i^2-2(k_1+k_2)\ \ \  \ \textrm{ (by  (b) and Lemma~\ref{lem: fix vs base})}.
  \end{align*}
So \begin{align*}4\chi(\mc O_S)&\leq2(k_1+k_2)-\sum_{i=1}^n D_i^2\\
&=2(k_1+k_2)-(K_S^2-8\chi (\mc O_S)) \ \ \ \
   \ \textrm{ (by Lemma~\ref{lem: sign inv})}\\
&\leq2\times24+24 \ \ \  \textrm{ (by (\ref{k1+k2}) and  (\ref{1}))
}\\
&=72.\end{align*} This contradicts the assumption on $\chi(\mc O_S)$.
\end{proof}

\section{Examples}\label{sec: ex}
In \cite{Liu10} the second author constructed explicitly regular surfaces  $S$ with $\au\cong(\ZZ/2\ZZ)^2$ and $p_g(S)$ arbitrarily large. In this section,
by extending the spherical systems of the surfaces in \cite{BC04} we  construct many families of regular surfaces isogenous to a product of unmixed type with $\au=(\ZZ/2\ZZ)^2$ and $\ZZ/3\ZZ$ respectively and $p_g(S)$ arbitrarily large. Note that the surfaces in \cite[Example~2.8]{Liu10} can also be obtained in this way, see Example~\ref{ex: Z23} below.

\begin{defn}(\cite[Definition~3.1]{Cat00})\label{isogenous} Let $C,\,D$ be curves of genus at least two, and $G$ is a finite group acting (faithfully) on $C$ and $D$. If the diagonal action of $G$ on $C\times D$ is free then the smooth quotient surface $S:=(C\times D)/G$ is called a surface isogenous to a product of unmixed type.
\end{defn}

Given a surface isogenous to a product of unmixed type $S=(C\times D)/G$ its invariants satisfy
\[
 K_S^2 = 8\chi(\mc O_S) = \frac{8}{|G|}(g(C)-1)(g(D)-1)\text{ and } q(S)=g(C/G)+g(D/G).
\]
If $q(S)=0$ then $C/G\cong D/G\cong\PP^1$ and the surface is determined by two \emph{disjoint spherical systems} (\cite[Definition~1.2]{BC04}).

Here is an expression for the second cohomology of a surface isogenous to a product of unmixed type, reducing the computation of cohomology on surfaces to that on curves.
\begin{lem}[{\cite[4.5]{CLZ13}}]\label{lem: 2nd cohomology}
Let $S=(C\times D)/G$ be a surface isogenous to a product of unmixed
type. If $G$ is abelian then the second cohomology of $S$ is
\begin{equation}\label{eq: cohomology2}
 H^2(S,\mathbb{C}) = W \bigoplus  \left(\bigoplus_{\chi\in \widehat G}
  H^1(C,\mathbb{C})^{\chi}\otimes_\mathbb{C}
  H^1(D,\mathbb{C})^{\bar\chi}\right),
\end{equation}
where
$W=H^2(C,\mathbb{C})\otimes_\mathbb{C}H^0(D,\mathbb{C})\bigoplus
 H^0(C,\mathbb{C})\otimes_\mathbb{C}H^2(D,\mathbb{C})$.
\end{lem}
 The group $\Aut_\QQ(S)$ acts trivially on $W$, so it preserves the two natural fibrations $S\rightarrow C/G$ and $S\rightarrow D/G$.

Let $\Delta_G$ be the diagonal subgroup of $G\times G$. Assume $G$ is abelian. Then the quotient group $(G\times G)/\Delta_G$ has an action on $S=(C\times D)/G$, induced by the product action of $G\times G$ on $C\times D$. Write $\overline{(\sigma,\tau)}$ for the image of $(\sigma,\tau)\in G\times G$ in the quotient $G\times G/\Delta_G$. Note that any element of $(G\times G)/\Delta_G$ can be written as $\overline{(\sigma,1)}$ with $\sigma\in G$, and in this way $(G\times G)/\Delta_G$ is identified with $G$.
\begin{cor}\label{cor: element in aut0}
An element $\overline{(\sigma,1)}\in(G\times G)/\Delta_G$ acts trivially on $H^*(S,\Q)$ if and only if for every $\chi\in\widehat G$ such that $H^1(C,\CC)^\chi \neq 0, H^1(D,\CC)^{\bar\chi}\neq 0$ we have $\chi(\sigma)=1$.
\end{cor}

The following result on the cohomology representation of the
group of automorphisms of an irreducible curve will be used in calculation  of the examples below.
\begin{thm}[{\cite[Proposition~2]{Br87} or
\cite[p.~244]{Br91}}]\label{thm: cohomology representation curves}
 Let $C$ be a smooth curve of genus $g(C) \geq 2$
and
 $G<\Aut(C)$ a finite abelian group of automorphisms.
 Let $\pi\colon C\rightarrow C/G$ be the quotient map and
 $r$ the number of branch points on $C/G$. Then, for $1\neq\chi\in \widehat G$,
\begin{equation}
 \dim_\CC H^1(C,\mathbb{C})^\chi=(2g(C/G)-2+r)-\sum_{j=1}^{r}l_{\sigma_j}(\chi)
\end{equation}
where $\langle\sigma_1\rangle,\dots,\langle\sigma_r\rangle$ are the
stabilizers of the points lying over the $r$ branch points of
$C\rightarrow C/G$,  and for any $\sigma\in G$,
\[l_{\sigma}(\chi)=\left\{
\begin{array}{ll}
 1  &  \text{if } \chi(\sigma)=1;\\
 0  &  \text{if } \chi(\sigma)\neq 1.
\end{array}
\right.\]
\end{thm}
\begin{cor}\label{cor: non0 es}
Let the notation be as in Theorem~\ref{thm: cohomology representation curves}.
If $g(C/G)=0$ then for any $\chi\in\widehat G$, $H^1(C,
 \CC)^\chi \neq 0$ if and only if there are distinct $j_1,j_2,j_3$ such that $\chi(\sigma_{j_i})\neq 1$ for $1\leq i\leq 3$.
\end{cor}

\begin{ex}[Surfaces $S$  with  $\au=(\ZZ/2\ZZ)^2$ and $\kappa(S/\au)=1$] (cf.~\cite[Example~2.8]{Liu10})\label{ex: Z23}
Let $G=\langle e_1,e_2,e_3\rangle\cong(\ZZ/2\ZZ)^3$. We have
\[
 |e_1|=|e_2|=|e_3|=2 \text{ and } \langle e_i\rangle \cap \langle e_j \rangle =1, \text{ for all } 1\leq i<j \leq 3.
\]
Let $C',\,D'$ be two smooth rational curves. Given a positive integer $r$ and $2r+6$ distinct points $p_1,\dots,p_{2r+6}$ of $C'$. Then the fundamental group of the punctured curve $ C'\setminus\{p_1,\dots,p_{2r+6}\}$ is
\[
 \pi_1(C'\setminus\{p_1,\dots,p_{2r+6}\})=\langle \gamma_1,\cdots,\gamma_{2r+6} \mid
  \gamma_1\dots\gamma_{2r+6}=1  \rangle
\]
and we have a surjective homomorphism of groups
\[
 \theta\colon \pi_1(C'\setminus \{p_1,\dots, p_{2r+6}\})\rightarrow G
\]
such that $\theta(\gamma_i)=e_1 + e_2$ (resp.~$e_1+e_3$, $e_1+e_2+e_3$) for any
$1\leq i\leq 2r + 2$ (resp.~$i=2r+3$ or $2r+4$, $i=2r+5$ or $2r+6$). By Riemann's existence theorem, we obtain a
$G$-covering $C\rightarrow C'$ such that $p_1,\dots,p_{2r+6}$ are
the branch points and $\langle e_1 + e_2 \rangle$ (resp.~$\langle e_1+e_3 \rangle$, $\langle e_1+e_2+e_3 \rangle$) is the stabilizer of the points lying over $p_i$ for $1\leq i\leq 2r$ (resp.~$i=2r+3$ or $2r+4$, $i=2r+5$ or $2r+6$).

Let $q_1,\dots,q_5$ be five distinct points on $D'$. Then
\[
 \pi_1(D'\setminus\{q_1,\dots,q_5\})=\langle \delta_1,\dots,\delta_5\mid\ \delta_1\cdots\delta_5=1 \rangle.
\]
So we can define a surjective homomorphism of groups
\[
 \vartheta\colon \pi_1(D'\setminus\{q_1,\dots,q_5\})\rightarrow G
\]
such that
\[
 \vartheta(\delta_1)=\vartheta(\delta_2)=e_1,\quad
\vartheta(\delta_3)= e_2,\quad
\vartheta(\delta_4)=e_3, \quad
\vartheta(\delta_5)=e_2+e_3.
\]
Again by Riemann's existence theorem there is a $G$-covering
$D\rightarrow D'$ such that $q_1,\dots,q_5$ are the branch points
and the stabilizers of the points lying over $q_1$ and $q_2$ (resp.~$q_3$, $q_4$, $q_5$) are both $\langle e_1\rangle$ (resp.~$\langle e_2\rangle$, $\langle e_3\rangle$, $\langle e_2+e_3\rangle$).

Let $G$ act on $C\times D$ diagonally. It is easy to see that the two spherical systems  of $C\rightarrow C'$ and $D\rightarrow D'$ are disjoint, so $G$ acts freely on $C\times D$. Therefore $S:=(C\times D)/G$ is a surface isogenous to a product of unmixed type. By Riemann--Hurwitz's formula we have $g(C)=4r+5$ and $g(D) = 3$. So
\[
 K_S^2= 8(r+1),\quad  \chi(\mc O_S)  = r +1\  \text{ and } \  q(S)=g(C')+g(D')=0.
\]

 Consider the character $\chi\in \widehat{G}$ such that
$\chi(e_1)=\chi(e_3)=-1$ and $\chi(e_2) = 1$. By Corollary~\ref{cor: non0 es} one checks easily that $\chi$ is the only character
 satisfying the following conditions:
\[
H^1(C,\mathbb C)^\chi\neq 0 \text{ and }H^1(D,\mathbb
C)^{\bar\chi}\neq 0.
\]
Then, by the expression of $H^2(S,\CC)$ in Corollary~\ref{cor: element in aut0}, $\ker\chi$ acts trivially on $H^2(S,\CC)$ and hence also on $H^*(S,\Q)$.
Since $\ker\chi=\langle e_2,e_1+e_3\rangle$, we have that  $\au$ contains a subgroup isomorphic to $(\ZZ/2\ZZ)^2$. On the other hand, we have $|\au|\leq 4$ if $\chi(\mc O_S) \geq 189$ by \cite{Cai04}. Hence we have $\au=\ker\chi\cong (\ZZ/2\ZZ)^2$ if $r\geq 188$. Our surfaces $S$ form an infinite series as $r$ varies.

The projection $C\times D\to C$ induces a fibration of genus $3$ on $S$. One sees easily that $g(D/\ker\chi)=1$ and $\kappa(S/\au)=1$. The canonical map of $S$ is then composed with this fibration of genus $3$.
\end{ex}

\begin{rmk}Example~\ref{ex: Z23} gives a new example of  surfaces whose canonical maps are composed with a fibration  of genus $3$.   See \cite[Exemple 2]{Be79}, \cite[Exemple]{X85b}, \cite[Remark, p.~600 ]{X87}, \cite[Examples 4.4, 4.6]{CL18}  for earlier examples. \end{rmk}

In Example~\ref{ex: Z23} we extended one of the spherical systems of \cite[3.1]{BC04} by adding $2r$ copies of $e_1+e_2$. A new series of surfaces with $\au = (\ZZ/2\ZZ)^2$ can be obtained if the other spherical system is extended.
\begin{ex}[Surfaces $S$ with  $\au=(\ZZ/2\ZZ)^2$ and $\kappa(S/\au)=2$]\label{ex: Z23'}
Let $G$ be a finite group isomorphic to $(\ZZ)^3$, with generators $e_1,\,e_2,\,e_3$. Then the 6-tuple $\mc A=(e_1+e_2, e_1+e_2, e_1+e_3, e_1+e_3, e_1+e_2+e_3, e_1+e_2+e_3)$ is a spherical system of $G$ and we can construct a $G$-covering $C\rightarrow \PP^1$ such that there are 6 branch points and the stabilizers over the branch points are exactly generated by the entries of $\mc A$. By the Riemann--Hurwitz formula we have $g(C) = 5$.

Given a positive integer $r$, consider another spherical system of generators of length $2r+5$:
\[
 \mc B=(e_1,\,\dots,\,e_1,\, e_2,\, e_3,\, e_2+e_3)
\]
where the first $2r+2$ entries are all $e_1$. Corresponding to $\mc B$ we obtain another $G$-covering $D\rightarrow \PP^1$ such that there are exactly $2r+5$ branch points and the stabilizers over the branch points are exactly generated by the entries of $\mc B$.  By the Riemann--Hurwitz formula we have $g(D) = 4r+3$.

Let $G$ act on $C\times D$ diagonally. Since the two spherical systems $\mc A$ and $\mc B$ are disjoint,  $G$ acts freely on $C\times D$. Therefore $S:=(C\times D)/G$ is a surface isogenous to a product of unmixed type and
\[
 K_S^2= 16r+8, \  \chi(\mc O_S)  = 2r +1, \   q(S)=0.
\]

Let  $\chi\in \widehat G$ such that
$\chi(e_1)=-1$ and $\chi(e_2) = \chi(e_3)=1$.
By Corollary~\ref{cor: non0 es} one checks easily that $\chi$ is the only character
 satisfying the following conditions:
\[
H^1(C,\mathbb C)^\chi\neq 0 \text{ and }H^1(D,\mathbb
C)^{\bar\chi}\neq 0.
\]
By the argument as in Example~\ref{ex: Z23}, we have that
$\au\cong\ker\chi=\langle e_2,e_3\rangle\cong (\ZZ/2\ZZ)^2$  and $\kappa(S/\au)=2$ if $r\geq 15$.
\end{ex}

Now we  extend the spherical systems of \cite[3.2, 3.3]{BC04} to get Example~\ref{ex: Z24} below.  Note that the two spherical systems of  both \cite[3.2]{BC04} and \cite[3.3]{BC04}  are in symmetric positions, so we obtain isomorphic surfaces if we extend either of them.
\begin{ex}[Surfaces $S$ with $\au=(\ZZ/2\ZZ)^2$ and $\kappa(S/\au)=2$]\label{ex: Z24}
Let $G$ be a finite group isomorphic to $(\ZZ/2\ZZ)^4$, with generators $e_1,\,e_2,\,e_3,\,e_4$. Then the following 5-tuple is a spherical system of $G$
\[
\mc A=(e_1,\, e_2,\, e_3,\, e_4,\, e:=e_1+e_2+ e_3 +e_4)
\]
and we can construct a $G$-covering $C\rightarrow \PP^1$ such that there are 5 branch points and the stabilizers over the branch points are exactly generated by the entries of $\mc A$. By the Riemann--Hurwitz formula we have $g(C) = 5$.

Given a positive integer $r$, consider another spherical system of generators of length $2r+5$:
\[
 \mc B=(e+e_1,\dots,e+e_1,e+e_2, e_1+ e_3,e_2+e_4, e_3+e_4)
\]
where the first $2r+1$ entries are all $e+e_1$. Corresponding to $\mc B$ we obtain another $G$-covering $D\rightarrow \PP^1$ such that there are exactly $2r+5$ branch points and the stabilizers over the branch points are exactly generated by the entries of $\mc B$.  By the Riemann--Hurwitz formula we have $g(D) = 8r+5$.

Let $G$ act on $C\times D$ diagonally. Since the two spherical systems $\mc A$ and $\mc B$ are disjoint,  $G$ acts freely on $C\times D$. Therefore $S:=(C\times D)/G$ is a surface isogenous to a product of unmixed type and
\[
 K_S^2= 16r+8,\ \chi(\mc O_S)  = 2r +1, \ q(S)=0.
\]

Consider the two characters $\chi_i$, $i=1,2$ such that
$\chi_i(e_2)=\chi_i(e_3) = \chi_i(e_4)= - 1$ and $\chi_i(e_1) = (-1)^i$. By Corollary~\ref{cor: non0 es} these are the only characters
$\chi$ satisfying the following conditions:
\[
H^1(C,\mathbb C)^\chi\neq 0 \text{ and }H^1(D,\mathbb
C)^{\bar\chi}\neq 0.
\]
 So  $\ker\chi_1\cap \ker\chi_2$ acts trivially on $H^2(S,\CC)$ by the expression of $H^2(S,\CC)$ in Corollary~\ref{cor: element in aut0}.
By a simple  calculation we have  \[\ker\chi_1\cap\ker\chi_2=\langle e_2+e_3,e_2+e_4\rangle\cong (\ZZ/2\ZZ)^2.\]
Then by the argument as in Example~\ref{ex: Z23},  we have $\au\cong\ker\chi_1\cap\ker\chi_2\cong (\ZZ/2\ZZ)^2$  and $\kappa(S/\au)=2$ if $r\geq 15$.
\end{ex}

Finally in this section we give an example of surfaces $S$ with $p_g(S)=0$ and $|\au|>4$.
\begin{ex}\label{ex: isog pg0}
Surfaces isogenous to a product of unmixed type with $p_g(S)=q(S)=0$ are classified in \cite{BCG08}. Let $S=(C\times D)/G$ be such a surface with $G$ abelian (see \cite{BC04}).
It is easy to see by Lemma~\ref{lem: 2nd cohomology} that
\[
 H^2(S,\mathbb{C}) = H^2(C,\mathbb{C})\otimes_\mathbb{C}H^0(D,\mathbb{C})\bigoplus
 H^0(C,\mathbb{C})\otimes_\mathbb{C}H^2(D,\mathbb{C}).
\]
More geometrically,
$H^2(S,\mathbb{C})$ is spanned by the cohomology classes of the fibers of two natural projections $S\rightarrow C/G$ and $S\rightarrow D/G$.

Now one sees that $(G\times G)/\Delta_G$ acts trivially on $H^2(S,\CC)$ and hence also on $H^*(S,\CC)$. In other words, $(G\times G)/\Delta_G$ is included in $\au$. Note that $G$ can be one of the following groups (\cite[Theorem~0.1]{BC04}): $(\ZZ/2\ZZ)^3$, $(\ZZ/2\ZZ)^4$, $(\ZZ/3\ZZ)^2$, $(\ZZ/5\ZZ)^2$. In particular, there are regular surfaces with $|\au|\geq 25$.

\emph{Question}: Is it true that
 $\au = (G\times G)/\Delta_G $ for the surfaces in \cite{BC04}?
\end{ex}

\end{document}